\RequirePackage{fix-cm}
\documentclass[smallextended]{svjour3}       % onecolumn (second format)
\smartqed  % flush right qed marks, e.g. at end of proof
\usepackage[american]{babel}

\usepackage{stmaryrd}
\usepackage{amsmath,bm}
\usepackage{epsfig,color}
\usepackage{amsfonts}
\usepackage{amssymb}
\usepackage{url}
\usepackage{algorithm, algorithmic}

\usepackage{ifthen}
\usepackage{comment}

\usepackage{adjustbox}
\usepackage{overpic}

\usepackage{multirow}
\usepackage{tabularx}

\usepackage[subrefformat=parens,labelformat=parens]{subcaption}
\usepackage{user_macros}

\begin{document}

\title{Multiresolution-analysis for stochastic hyperbolic conservation laws}
\titlerunning{Multiresolution-analysis for stochastic hyperbolic conservation laws}

\author{M.~Herty \and A.~Kolb \and S.~M\"uller
}

\institute{M. Herty \at
  Institut f\"ur Geometrie und Praktische Mathematik, RWTH Aachen University, Templergraben 55, D-52056 Aachen, Germany \\
  %              Tel.: +49 241 80 94510 \\
  %              Fax:  +49 241 80 92317 \\
  \email{herty@igpm.rwth-aachen.de}           %  \\
  \and
  A. Kolb
  % \at DFG Research Training Group Energy, Entropy, and Dissipative Dynamics, RWTH Aachen University, Templergraben 55, D-52056 Aachen, Germany
  \at
  Institut f\"ur Geometrie und Praktische Mathematik, RWTH Aachen University, Templergraben 55, D-52056 Aachen, Germany \\
  %              Tel.: +49 241 80 93068 \\
  %              Fax:  +49 241 80 92317 \\
  \email{kolb@eddy.rwth-aachen.de}           %  \\
  \and
  S. M\"uller \at
  Institut f\"ur Geometrie und Praktische Mathematik, RWTH Aachen University, Templergraben 55, D-52056 Aachen, Germany \\
  %              Tel.: +49 241 80 96469 \\
  %              Fax:  +49 241 80 92317 \\
  \email{mueller@igpm.rwth-aachen.de}           %  \\
}


\maketitle   

%%%%%%%%%%%%%%%%%%%%%%%%%%%%%%%%%%%%%%%%%%%%%%%%%%%%%%%%%%%%%%%%%%%%%%%%%%%%%%%%%%%%%%%%%%%%%%%%%%%%%%%%%%
\begin{abstract}
  A multiresolution analysis for solving stochastic conservation laws is proposed. Using a novel adaptation strategy and a higher dimensional deterministic problem, a discontinuous Galerkin (DG) solver is derived. A multiresolution analysis of the DG spaces for the proposed adaptation strategy is presented. Numerical results show that in the case of general stochastic distributions the performance of the DG solver is significantly improved by the novel adaptive strategy. The gain in efficiency is validated in computational experiments.
\end{abstract}



%%%%%%%%%%%%%%%%%%%%%%%%%%%%%%%%%%%%%%%%%%%%%%%%%%%%%%%%%%%%%%%%%%%%%%%%%%%%%%%%%%%%%%%%%%%%%%%%%%%%%%%%%%

\section{Introduction}
\label{sec:intro}

In the past decades accurate and stable schemes for hyperbolic systems of conservation laws have been subject to intensive research. In many applications uncertainties have to be taken into account and thereby changing the deterministic problem to a higher-dimensional stochastic problem. Those uncertainties are usually modeled as random variables leading to stochastic hyperbolic conservation laws.
\par
Several approaches have been proposed in the past to deal with stochastic PDEs both from an analytical and numerical perspective. A broad classification distinguishes non-intrusive and intrusive methods. Among the non-intrusive methods, the Monte Carlo method and its variants are sampling-based methods. In the context of hyperbolic equations, they are used, for example, in \cite{Mishra2016,Mishra2012,Badwaik2021}. Another class of non-intrusive methods is based on stochastic collocation \cite{Xiu2005,Giesselmann2020,Wan2006,Sullivan2016,Nordstroem2015}, where the stochastic moments are obtained by applying adapted numerical quadrature. An intrusive approach on the contrary uses the representation of stochastic perturbations by  a series of orthogonal functions, known as generalized polynomial chaos (or Karhunen-Lo\`eve) expansions \cite{Cameron1947,Xiu2002}. Those expansions are substituted in the governing hyperbolic equations and projected on a lower-dimensional subspace. This leads to deterministic evolution equations for the coefficients of the series expansion.  In particular, in the context of partial differential equations this has been applied successfully to a large class of problems \cite{Ghanem1991,Cameron1947,Gottlieb2008,Hu2016,Pulch2011,Zanella2020}.
In the context of hyperbolic problems there have been contributions leading to a deterministic system that might encounter a loss of hyperbolicity.
Besides the theoretical obstacles of the intrusive approaches, several contributions towards numerical schemes and their convergence analysis have been proposed and we refer to \cite{Mishra2016,Mishra2012,Giesselmann2020,Oeffner2018,Nordstroem2005,Duerrwaechter2018} for further references.
\par
Typically, the computation of stochastic moments like expectation or variance for instance using a classical Monte Carlo method is very time-consuming due to low convergence rates. In the context of conservation laws with discontinuities in space, the convergence behavior has been observed to deteriorate because discontinuities may also be present in the stochastic directions \cite{Abgrall2017,Barth2013}. To handle these discontinuities in the stochastic directions, several approaches using decomposition of the random space have been developed \cite{Schlachter2020,Giesselmann2020,Wan2006}.

The objective of the present work is to overcome computational drawbacks of the interplay between spatial and the stochastic dynamics, e.g. using suitable grid adaptation. For this purpose, we rewrite the stochastic problem as deterministic conservation law in higher dimensions. The stochastic variables are then treated as additional (spatial-like) variables. We prove that the solution to the weak formulation is a solution to a stochastic hyperbolic conservation law. The deterministic approach allows to investigate the interplay between the dynamics of the spatial and stochastic dimensions. The key part will be the introduction of a novel adaptation strategy that allows to handle the increased dimensions of the problem efficiently and also exploits the particularities of the stochastic variables. We approximate the solution of the deterministic problem by a discontinuous Galerkin (DG) scheme. The DG solver is combined with local grid refinement that allows for adaptation in both the spatial and the stochastic directions.
Besides \textit{local error estimators}
cf.~\cite{Bey-Oden:96,Adjerid-Devine-Flaherty-Krivodonova:02,Hartmann-Houston:02a,Hartmann-Houston:02b,Houston-Senior-Sueli:02,Dedner-Makridakis-Ohlberger:07,Wang-Mavriplis:09,Giesselmann:2015}, which are not reliable and efficient
because the error is bounded only from above by the norm of the residual,
and \emph{sensor-based} methods, cf.~\cite{Pongsanguansin:2012,Remacle:2006,Hu2013,Remacle-Flaherty-Shephard:03,Arvanitis:2010}, which do not provide any error control,
another option to control local grid refinement is based on perturbation arguments. Here, the idea is to consider the discretization on an adaptive grid as a perturbation of a discretization on a uniform grid. We follow the latter since we then control the grid adaptation such that the asymptotic behavior of the uniform discretization error is maintained,  cf.~\cite{Harten:1995zr,GottschlichMueller:98,Bramkamp-Lamby-Mueller:05,Calle2005Wavelets-and-Ad,HovhannisyanMuellerSchaefer-2014}. This paradigm allows control of the perturbation error between reference and adaptive scheme and it is achieved by a multiresolution analysis (MRA).
In the context of perturbation methods  the term \textit{efficiency} is  interpreted as the reduction of the computational cost (memory and CPU) in comparison to the cost of a fully refined reference scheme. The term \textit{reliability} is used in the sense of the capability of the adaptation process to maintain the accuracy of the reference scheme.
Here, we propose  a novel suitable threshold procedure for the MRA of the approximate solution of the deterministic problem such that the perturbation error in the averaged stochastic quantities can be  controlled. In particular, the MRA is designed such that the threshold procedure can be performed efficiently in the adaptive scheme.
In Theorems \ref{thm:thres-error-1} and \ref{thm:thres-error-1-moments} we verify that perturbing the approximate solution depending on both spatial and stochastic variables  by applying this threshold procedure the perturbation in the corresponding  stochastic moments is uniformly bounded by the perturbation error in the approximate solution.
On the one hand, this can be considered a stability result for the stochastic moments. On the other hand, it provides us with an idea which local information in the deterministic solution is relevant for the stochastic moments. This is the key to improve compression rates and, thus, leads to a better  performance of the adaptive scheme.
The new threshold procedure is incorporated in a  multiresolution-based adaptive DG scheme  that is verified numerically to provide a reliable and efficient approximation of the stochastic moments, although the solution of the deterministic problem might be locally not reliable, i.e., the adaptation process is goal-oriented rather than solution-oriented.

The outline of the current work is thus as follows.
In Sect.~\ref{sec:problem} we introduce the scalar stochastic Cauchy problem and its deterministic reformulation.
In particular, we prove that the stochastic problem and the deterministic problem are equivalent.
The deterministic formulation allows the computation of the moments of the stochastic problem in a more explicit way rather than by Monte-Carlo methods.
Then we introduce in Sect.~\ref{sec:MRA-DG} the MRA on DG spaces where we first consider the general concept.
Since in the deterministic approach we distinguish directions in spatial and stochastic variables, we construct a MRA for suitable products of DG spaces.
This is tailored to ensure efficiency of the resulting adaptive DG scheme.
In Sect.~\ref{sec:MRA-moments} we analyze the influence of the threshold error on the computation of moments of the solution with respect to the stochastic variables and develop a new refinement strategy.
Finally, in Sect.~\ref{sec:NumRes} this strategy is incorporated into a multiresolution-based adaptive DG scheme. Its efficiency is numerically verified where we consider random Burgers' equation and random Euler equations.

%%%%%%%%%%%%%%%%%%%%%%%%%%%%%%%%%%%%%%%%%%%%%%%%%%%%%%%%%%%%%%%%%%%%%%%%%%%%%%%%%%%%%%%%%%%%%%%%%%%%%%%%%%
\section{The scalar stochastic Cauchy problem and its reformulation}
\label{sec:problem}
To investigate the interaction of the spatial scales with scales in the stochastic moments we rewrite the stochastic problem as a deterministic problem in higher dimensions. For this purpose, we first introduce in Sect.~\ref{subsec:stoch_cp} the scalar stochastic Cauchy problem and the definitions of the stochastic moments Assuming that the random variables are absolutely continuous we then introduce a corresponding higher-dimensional deterministic problem.

\subsection{The scalar stochastic Cauchy problem}
\label{subsec:stoch_cp}
To define scalar conservation laws with uncertain initial data we first introduce the probability space
$(\Omega, \mathcal F, \Prob)$ with $\Omega$  a non-empty set, $\mathcal F$ a $\sigma$-algebra over $\Omega$ and $\Prob$ a probability measure on $\mathcal F$. Let be $\xi:\Omega\rightarrow\Omega_\xi$ a random variable on $(\Omega, \mathcal F, \Prob)$ and let be $\mathcal F_\xi :=\mathcal B(\Omega_\xi)$ the Borel $\sigma$-algebra over $\Omega_\xi:=\R^m$. For $B\in\mathcal B(\R^m)$ we define the probability distribution of $\xi$ by $\Prob_\xi(B) \equiv \Prob(\xi^{-1}(B)):=\Prob(\{\omega \in \Omega: \xi(\omega)\in B\})$ on $(\R^m,\mathcal B(\R^m))$.

In contrast to \cite{Mishra2012}, we assume that the probability distribution of $\xi$ is an absolutely continuous random variable with respect to the Lebesgue measure. Then, due to \cite[Theorem 17.10]{Bauer2001}, there exists an essentially bounded probability density $p_{\xi}:{\R^{m}}\rightarrow[0,\infty)$ such that $\Prob_\xi(B) = \int_B p_\xi(\bxi)\diff\bxi$ for all $B\in \mathcal{B}(\R^m)$. Furthermore, the expectation for $u\in L^1(\R^m)$ is
\begin{align}
  \label{eq:expectation_value}
  \E[u](\xi) := \int_\Omega u(\xi(\omega))\dif\Prob(\omega)
  = \int_{\R^m} u(\bxi)\, p_\xi(\bxi) \dif\bxi
\end{align}
and its $k$-th centralized moments are
\begin{align}
  \label{eq:centralized_moments}
  \ckmoments{k}[u](\xi) := \E\left[(u - \E\left[u\right] )^k\right](\xi),\quad k\in\N.
\end{align}
The stochastic Cauchy problem for scalar conservation laws reads
\begin{subequations}
  \label{eq:stoch_problem}
  \begin{alignat}{2}
    \label{eq:scl}
     & \bar u_t(t,\bx;\omega_\xi) + \sum_{j = 1}^{d} \pardev{\bx_j} \bof_j(\bar u(t,\bx;\omega_\xi))  = 0, &                                                       & \quad \bx \in {\R^{d}},\ {\omega_\xi} \in \Omega_\xi,\ t \in (0,T) \\
    \label{eq:scl_ic}
     & \bar u(0,\bx;\omega_\xi) = \bar u_0(\bx;\omega_\xi),
     &                                                                                                             & \quad \bx \in {\R^{d}},\ {\omega_\xi} \in \Omega_\xi.
  \end{alignat}
\end{subequations}
Here, $\bar u(t,\bx;\omega_\xi)\in\R$ is the conserved variable, $\bof \in C^1(\R,{\R^{d}})$ is the flux field and $T \in (0,\infty)$ is the terminal time.
Uncertainty enters the problem in the initial condition \eqref{eq:scl_ic}. As in \cite{Mishra2012}, we assume that the initial condition \eqref{eq:scl_ic} is given by $\bar u_0 \in \Lk{1}{{\R^{d}}}$-valued random variable.

\begin{definition}[\cite{Mishra2012}, Definition 3.2]
  \label{def:random_entropy_solution}
  A random field
  $\bar u:\Omega_\xi\rightarrow C_b([0,T], L^1(\R^d))$
  with initial $L^1(\R^d)$-valued random variable ${\bar u}_0$ is said to be a random entropy solution if it satisfies the following two conditions:
  \begin{enumerate}
    \item Weak solution: For $\Prob_\xi$-a.s. ${\omega_\xi}\in\Omega_\xi$, $\bar u(\cdot,\cdot;{\omega_\xi})$ satisfies the weak formulation
          \begin{align}
            \label{eq:stochastic_weak_formulation}
            \begin{split}
              \int_0^\infty\int_{{\R^{d}}} & \left(\bar u(t,\bx;\omega_\xi) \bar \varphi_t(t,\bx) + \sum_{j=1}^{d} \bof_j(\bar u(t,\bx;\omega_\xi))\pardev{\bx_j}\bar \varphi(t,\bx)\right)\dif \bx\dif t \\
              & + \int_{{\R^{d}}}\bar{u}_0(\bx;\omega_\xi)\bar \varphi(0,\bx)\dif \bx = 0
            \end{split}
          \end{align}
          for all test functions $\bar \varphi \in C_0^1([0,T]\times{\R^{d}})$.
    \item Entropy condition: Let $(\eta,\bQ)$ be an entropy-entropy flux pair, i.e., $\eta:\R\rightarrow\R$ is a convex function and $\bQ:\R\rightarrow\R^d$ with $\bQ'_j(\bar u) = \eta'(\bar u)\bof'_j(\bar u),\  j=1,\dots,{d}$. For $\Prob_\xi$-a.s. ${\omega_\xi}\in\Omega_\xi$, $\bar u(\cdot,\cdot,\omega_\xi)$ satisfies the inequality
          \begin{align}
            \label{eq:stoachastic_entropy_condition}
            \begin{split}
              \int_0^\infty\int_{{\R^{d}}} & \left(\eta(\bar u(t,\bx;\omega_\xi))\bar \varphi_t(t,\bx) + \sum_{j=1}^{d}\bQ_j(\bar u(t,\bx;\omega_\xi))\pardev{\bx_j}\bar \varphi(t,\bx)\right)\dif \bx\dif t \\
              & + \int_{\R^d} \eta(\bar u_0(\bx;\omega_\xi))\bar\varphi(0,\bx)\dif \bx \geq 0
            \end{split}
          \end{align}
          for all test functions $\bar \varphi\in C^1_0([0,T]\times{\R^{d}})$ with $\bar \varphi \geq 0$.
  \end{enumerate}
\end{definition}
In \cite{Mishra2012} it is proven that there exists a unique random entropy solution for a general probability space $(\Omega,\mathcal F,\Prob)$, if the entropy solution exists for $\Prob$-a.s. $\omega\in\Omega$.
\begin{theorem}[\cite{Mishra2012}, Theorem 3.3]
  \label{thm:stochastic_entropy_solution}
  Consider the stochastic Cauchy problem \eqref{eq:scl} with random initial data \eqref{eq:scl_ic}
  given by a $\Lk{1}{{\R^{d}}}$-valued random variable $\bar u_0$ satisfying
  \begin{align}
    \label{eq:ic_linfty_bv}
    \bar{u}_0(\cdot;\omega_\xi)
    \in (L^1\cap\linf)({\R^{d}})\quad\text{for}\quad\Prob_\xi\text{-a.s. } \omega_\xi\in\Omega_\xi.
  \end{align}

  Furthermore, assume $\norm[\Lk{k}{\Omega_\xi;\Lk{1}{{\R^{d}}}}]{\bar u_0} < \infty$ for some $k\in\N$. Then, there exists a unique random entropy solution $\bar u:\Omega_\xi\rightarrow C_{b}([0,T]; \Lk{1}{{\R^{d}}})$
  such that for all $0\leq t\leq T$ and all $k\in\N$:
  \begin{align*}
    {\norm{\bar u}}_{L^k(\Omega_\xi; C([0,T];\Lk{1}{{\R^{d}}}))}    & \leq {\norm{\bar u_0}}_{L^k(\Omega_\xi;\Lk{1}{{\R^{d}}})}
    \intertext{and}
    \norm[(\lk{1}\cap \linf)({\R^{d}})]{\bar u(t,\cdot;\omega_\xi)} & \leq
    \norm[(\lk{1}\cap \linf)({\R^{d}})]{\bar{u}_0(\cdot;\omega_\xi)}
  \end{align*}
  for $\Prob_\xi$-a.s. ${\omega_\xi}\in\Omega_\xi$.
\end{theorem}

Furthermore, if the $k$-th stochastic moment of the initial condition \eqref{eq:scl_ic} exists for some $k\in\N$, we obtain existence of the $k$-th moment of the random entropy solution.

%%%%%%%%%%%%%%%%%%%%%%%%%%%%%%%%%%%%%%%%%%%%%%%%%%%%%%%%%%%%%%%%%%%%%%%%%%%%%%%%%%%%%%%%%%%%%%%%%%%%%%%%%%

\subsection{Deterministic reformulation}
\label{subsec:det_vs_stoch}

Motivated by \cite{Schwab2013,Tokareva2013} we introduce a deterministic approach to treat the stochastic parameter $\omega_\xi$. Since, there exists a random entropy solution, we introduce the stochastic variable $\omega_\xi$ as additional (spatial) variable resulting in a deterministic problem.

For $\bx\in\R^{d}$ and $\bxi\in\R^m$ we introduce the new variable
$\by := (\bx,\bxi) \in {\R^{d + m}}$.
Furthermore, we define a new flux $\bof\in C^1(\R,{\R^{d + m}})$ with zero flux in the (stochastic) directions, i.e.,
\begin{align}
  \label{eq:equiv_random_det_flux}
  \bof_{d+j} \equiv 0,\quad j=1,\dots,m
\end{align}
and we consider
\begin{subequations}
  \label{eq:det_problem}
  \begin{alignat}{2}
    \label{eq:det_sto_cl}
     & u_t(t,\by) + \sum_{j = 1}^{{d + m}} \pardev{\by_j} \bof_j(u(t,\by))  = 0, &  & \quad \by \in {\R^{d + m}},\ t \in (0,T) \\
    \label{eq:det_sto_cl_ic}
     & u(0,\by) = u_0(\by),                                                              &  & \quad \by \in {\R^{d + m}}.
  \end{alignat}
\end{subequations}
The new conserved variable is
$u(t,\by) \equiv u(t,(\bx,\bxi))$.
Following the classical theory of deterministic scalar conservation laws, cf.~\cite{Godlewski:1991}, the entropy solution is then defined as follows:

\begin{definition}
  \label{def:entropy_solution}
  A solution
  $u\in  C_b([0,T];L^1_{\text{loc}}(\R^{d+m}))$
  to the deterministic Cauchy problem \eqref{eq:det_problem} with initial data
  $u_0\in L^\infty(\R^{d+m})$
  is an entropy solution if it satisfies the following:
  \begin{enumerate}
    \item Weak solution: $u$ satisfies the weak formulation
          \begin{align}
            \label{eq:weak_formulation}
            \begin{split}
              \int_0^\infty\int_{{\R^{d + m}}} & \left(u(t,\by) \varphi_t(t, \by) + \sum_{j=1}^{d + m} \bof_j(u(t,\by))\pardev{\by_j}\varphi(t,\by)\right)\dif \by\dif t \\
              & + \int_{{\R^{d + m}}} u_0(\by)\varphi(0,\by)\dif \by = 0
            \end{split}
          \end{align}
          for all test functions $\varphi \in C_0^1([0,T]\times{\R^{d + m}})$.
    \item Entropy condition: Let $(\eta,\bQ)$ be an entropy-entropy flux pair, i.e., $\eta:\R\rightarrow\R$ is a convex function and $\bQ:\R\rightarrow\R^{d+m}$ with $\bQ'_j(u) = \eta'(u)\bof'_j(u),\ j=1,\dots,{d + m}$. Then, $u$ satisfies
          \begin{align}
            \begin{split}
              \label{eq:entropy_condition}
              \int_0^\infty\int_{{\R^{d + m}}} & \left(\eta(u(t,\by))\varphi_t(t,\by) + \sum_{j=1}^{d + m}\bQ_j(u(t,\by))\pardev{\by_j}\varphi(t,\by)\right)\dif \by\dif t \\
              & + \int_{\R^{d+m}} \eta(u_0(\by))\varphi(0,\by)\dif \by \geq 0
            \end{split}
          \end{align}
          for all test functions $\varphi\in C^1_0([0,T] \times {\R^{d + m}})$ with $\varphi \geq 0$.
  \end{enumerate}
\end{definition}

Note that due to the vanishing fluxes in the stochastic direction
\begin{align}
  \label{eq:equiv_random_det_entropyflux}
  \bQ'_{d+j} = 0, \ j=1,\dots,m
\end{align}
holds for the entropy flux.
The existence of a unique entropy solution is proven in \cite{Dafermos2016}.
\begin{theorem}[\cite{Dafermos2016}, Chapter VI, Theorem 6.2.2]
  \label{thm:entropy_solution}
  The deterministic Cauchy problem \eqref{eq:det_problem} with  initial data
  $u_0\in L^\infty(\R^{d+m})$ has a unique entropy solution
  $u\in  C_b([0,T];L^1_{\text{loc}}(\R^{d+m}))$
  for all $T>0$.
\end{theorem}

Note that the entropy solution of the deterministic problem \eqref{eq:det_problem} coincides with the entropy solution of the stochastic problem \eqref{eq:stoch_problem}.

\begin{theorem}
  \label{thm:equiv_random_det}
  Assume that the probability density $p_\xi$ of the absolute continuous random variable $\xi$ is either positive or is compactly supported.  Let $\bar u_0$ be a $L^1({\R^{d}})$-valued random variable fulfilling \eqref{eq:ic_linfty_bv} and let
  $u_0\in  \linf({\R^{d + m}})$
  be the initial data of problem \eqref{eq:det_problem} such that
  \begin{align}
    \label{eq:equiv_random_det_init}
    u_0((\bx,\omega_\xi))  = \bar u_0(\bx;\omega_\xi) \quad\mbox{for}\quad \Prob_\xi\mbox{-a.s.}~\omega_\xi\in\Omega_\xi\quad\mbox{and for a.e.}\quad \bx\in\R^d.
  \end{align}
  Furthermore, we assume that the flux fulfills \eqref{eq:equiv_random_det_flux}.
  \\
  Let $\bar u$ and $u$ denote the unique entropy solutions according to Theorem \ref{thm:stochastic_entropy_solution} and \ref{thm:entropy_solution} of the stochastic Cauchy problem \eqref{eq:stoch_problem} and the deterministic Cauchy problem \eqref{eq:det_problem}, respectively.
  Then it holds
  \begin{align}
    \label{eq:equiv_random_det}
    u(t,(\bx,\omega_\xi))=  \bar u(t,\bx;\omega_\xi)  \quad\mbox{for}\quad \Prob_\xi\mbox{-a.s.}~\omega_\xi\in\Omega_\xi\quad\mbox{and for a.e.}\quad \bx\in\R^d.
  \end{align}
\end{theorem}

The proof of Theorem \ref{thm:equiv_random_det} is given in Appendix \ref{appendix:proof_lemmata}.

\paragraph*{}
Instead of  approximating the stochastic moments   of the stochastic Cauchy problem \eqref{eq:stoch_problem} by means of Monte-Carlo type schemes, Theorem \ref{thm:equiv_random_det} allows to approximate these moments in a post-processing step. Therein,   we  apply classical deterministic discretization techniques such as finite volume schemes or DG schemes to the deterministic Cauchy problem \eqref{eq:det_problem}.

%%%%%%%%%%%%%%%%%%%%%%%%%%%%%%%%%%%%%%%%%%%%%%%%%%%%%%%%%%%%%%%%%%%%%%%%%%%%%%%%%%%%%%%%%%%%%%%%%%%%%%%%%%

\section{Multiresolution analysis for DG spaces}
\label{sec:MRA-DG}

The deterministic problem \eqref{eq:det_problem} is approximately solved by applying a modal DG scheme equipped with multiresolution-based grid adaptation \cite{HovhannisyanMuellerSchaefer-2014,Gerhard2014a,GerhardIaconoMayMueller-2015,GerhardMueller-2016}. The key ingredient is a multiresolution analysis (MRA)  applied to the DG approximation at each time step. Performing hard thresholding on the coefficients of the MRA is then employed to locally adapt the grid.
In the present work we will be interested in the control of the  error in the moments of the solution \eqref{eq:expectation_value} and \eqref{eq:centralized_moments} induced by the threshold error in  the DG approximation to \eqref{eq:det_problem}.
For this purpose, we first briefly summarize the general concept of a MRA for DG spaces following \cite{Gerhard:2017}. Then we specify this for a MRA for products of DG spaces that will allow us to investigate the aforementioned error in the moments.

\subsection{General concept of MRA}
The concept is based on a \textit{multiresolution sequence} ${\cal S} = \lbrace S_l\rbrace_{l\in\N_0}$ defined on some Hilbert space  ${\cal H}$, i.e.,   $S_l$ is a closed and linear subspace of ${\cal H}$, ${\cal S}$ is nested, i.e., $S_l\subset S_{l+1}$, $l\in\N_0$, and the union of all subspaces is dense in ${\cal H}$, cf.~\cite{Mallat:1989}. For our purposes we choose ${\cal H} = L^2(\Omega)$ where $\Omega\subset\R^n$ is some open and bounded domain with Lipschitz boundary. On this domain we introduce a \textit{hierarchy of nested grids}  $\mathcal{G}_l:=\{V_{\lambda}\}_{\lambda\in \mathcal{I}_l}$, $l\in\N_0$, i.e.,
\begin{align*}
  V_\lambda \cap V_\mu = \emptyset,\quad \lambda,\mu\in \calI_l, \lambda \ne \mu,     \quad
  \overline\Omega = \overline{\bigcup_{\lambda\in \calI_l} V_{\lambda}}
\end{align*}
where each cell $V_\lambda$ on level $l$,  open and bounded with Lipschitz boundary,
is composed of cells on level $l+1$, i.e.,
\begin{align*}
  \overline{V_{\lambda}} = \overline{\bigcup_{\mu\in\mathcal{M}_{\lambda}} V_{\mu}}, \qquad \forall\, \lambda\in \mathcal{I}_{l} ,
\end{align*}
where $\mathcal{M}_{\lambda}\subset\mathcal{I}_{l+1}$ is the refinement set of the cell $V_{\lambda}$.
On this grid hierarchy we define  the sequence ${\cal S} = \lbrace S_l\rbrace_{l\in\N_0}$ of DG spaces
\begin{align*}
  S_l := \{u\in L^2(\Omega) :\,  \left. u\right|_{V_{\lambda}}\in \Pi_{p-1}(V_{\lambda}),\  \lambda \in\mathcal{I}_l\}
\end{align*}
with $\Pi_{p-1}(V_{\lambda})$ the local polynomial space with maximal degree $p-1$. This sequence is a multiresolution sequence for $L^2(\Omega)$ if the hierarchy of nested grids  is dense, i.e.,
\begin{align*}
  \lim_{l\to\infty} \max_{\lambda\in\mathcal{I}_l} \diam(V_\lambda) = 0.
\end{align*}
Due to the nestedness, there exists the orthogonal complement space $W_l$ of $S_l$ with
respect to $S_{l+1}$ defined by
\begin{align*}
  W_l := \left\{ d^l\in S_{l+1} \,:\, (d^l,v)_{L^2(\Omega)} = 0,\ \forall\,v\in S_l \right\}
\end{align*}
such that
\begin{align*}
  S_{l+1} = S_l \oplus W_l.
\end{align*}
The decomposition
\begin{align*}
  S_L = S_0 \oplus W_0 \oplus \cdots \oplus W_{L-1}
\end{align*}
is called the \emph{multiscale decomposition} of $S_L, L\in\N$.
Due to the denseness in $L^2(\Omega)$ of the MRA, each function $u \in L^2(\Omega)$ can be represented by an infinite
multiscale decomposition
\begin{equation}
  \label{eq:msd}
  u = u^0 + \sum_{l\in\N_0} d^l
\end{equation}
with its contributions given by the orthogonal projections
\begin{equation}
  \label{eq:orthproj}
  u^l \equiv P_{S_l}(u) = P_{S_l}(u^{l+1}),\qquad
  d^l \equiv P_{W_l}(u) = P_{W_l}(u^{l+1}),\qquad
  l\in \N_0 .
\end{equation}
In particular, it holds
\begin{align*}
  u^{l+1} = u^l +   d^l ,\qquad
  l\in \N_0.
\end{align*}
Since the spaces $S_l$ as well as $W_l$ are piecewise polynomials, the orthogonal projections
\eqref{eq:orthproj} are computed locally on each element. This allows to spatially separate
the local contributions
\begin{align*}
  u_\lambda^l := u^l \cdot \mychi{V_\lambda} \in
  S_{l,\lambda}
  = \Pi_{p-1}(V_\lambda),\qquad
  d_\lambda^l := d^l \cdot \mychi{V_\lambda} \in
  W_{l,\lambda}
  \subset L^2(V_\lambda)
\end{align*}
in the multiscale decomposition \eqref{eq:msd} of the local DG space $S_{l,\lambda}$ and the
local complement space $W_{l,\lambda}$, respectively, where $\mychi{V_\lambda}$ is the indicator function on $V_\lambda$.

Due to orthogonality the local details may become small
\begin{equation}
  \label{eq:cancellation_property}
  \Vert d_\lambda^l \Vert_{L^2(V_\lambda)} \le
  \diam(V_\lambda)^p \sum_{\Vert \ibalpha\Vert_1 = p}
  \frac{1}{\balpha!} \Vert D^{\ibalpha} u  \Vert_{L^2(V_\lambda)}
\end{equation}
for $\lambda\in \mathcal{I}_l$, $l\in\N_0$, $V_\lambda$ convex and  $u|_{V_\lambda}\in H^p(V_\lambda)$. A proof of \eqref{eq:cancellation_property} is given in  \cite{Gerhard:2017}.
This motivates to discard small details from the multiscale decomposition of $u^L\in S_L$.
We introduce the notion of a \textit{$\Vert\cdot\Vert_\lambda$-significant} local detail, i.e.,
\begin{equation}
  \label{eq:significant_details}
  \Vert d_\lambda^l \Vert_\lambda > \veps_{\lambda,L},
\end{equation}
with $\Vert\cdot\Vert_\lambda: W_{l,\lambda} \to\R$ a local norm for the local complement space that is equivalent
to $ \Vert \cdot \Vert_{L^2(V_\lambda)} / \sqrt{|V_\lambda|}$,
i.e.,
\begin{equation}
  \label{eq:local-normequivalence}
  c\, \Vert d^l_\lambda \Vert_\lambda \le
  \frac{\Vert d^l_\lambda \Vert_{L^2(V_\lambda)}}{\sqrt{|V_\lambda|}} \le
  C\, \Vert d^l_\lambda \Vert_\lambda
\end{equation}
with constants $c,C>0$ independent of $l$ and $\lambda$.
Here, the local threshold values $\veps_{\lambda,L}$ are chosen such that
\begin{equation}
  \label{eq:loc-thres}
  \sum_{l=0}^{L-1} \max_{\lambda\in \mathcal{I}_l} \veps_{\lambda,L} \le \veps_{\text{max}}
\end{equation}
for a given global threshold value $\veps_{\text{max}}>0$.
For a dyadic grid hierarchy \eqref{eq:loc-thres} holds by the geometric sum when choosing
\begin{align}
  \label{eq:loc-thresh-uniform}
  \veps_{\lambda,L} = \frac{h_L}{h_l} \veps_{\text{max}},\ \lambda\in\calI_l,
\end{align}
where $h_l$ denotes the uniform diameter of the cells on level $l$.

To determine a sparse approximation for $u^L\in S_L$ the \textit{set of significant details} $\calD_{L,\veps} \subset \bigcup_{l=0}^{L-1} \mathcal{I}_l$ is defined as the smallest
set containing the indices of $\Vert\cdot\Vert_\lambda$-significant contributions, i.e.,
\begin{align*}
  \left\{ \lambda \in \bigcup_{l=0}^{L-1} \mathcal{I}_l \,:\,
  \Vert d_\lambda^l \Vert_\lambda > \veps_{\lambda,L}\right\}
  \subset \calD_{L,\veps},
\end{align*}
and being a tree, i.e.,
\begin{align*}
  \mu \in \calD_{L,\veps} \quad \Rightarrow \quad \lambda\in \calD_{L,\veps} \quad
  \forall\, \lambda \mbox{ with } V_\mu \subset V_\lambda.
\end{align*}
Then the \textit{sparse approximation} $u^{L,\veps}$ of $u^L$ is defined as
\begin{align*}
  u^{L,\veps} :=
  \sum_{\lambda\in\mathcal{I}_0} u_\lambda^0 +
  \sum_{l=0}^{L-1} \sum_{\lambda\in \calD_{L,\veps} \cap \mathcal{I}_l} d^l_\lambda .
\end{align*}
According to Thm.~3.2 \cite{Gerhard:2017} the thresholding error can be estimated  for fixed global threshold value $\veps_{\text{max}}$ and local threshold values
$\veps_{\lambda,L}$ satisfying \eqref{eq:loc-thres}
by
\begin{equation}
  \label{eq:threshold-error-gerhard}
  \Vert u^L - u^{L,\veps} \Vert_{L^q(\Omega)} \le C\, |\Omega |^{1/q} \veps_{\text{max}}
\end{equation}
for $q\in\{1,2\}$ and $C$ according to \eqref{eq:local-normequivalence}.

\begin{remark}[MRA on weighted $L^2$-spaces.]
  \label{rem:weightedspaces}
  Note that  MRA is described in terms of projections  avoiding the explicit representation of basis functions.  However, to perform  MRA in the computer we need to introduce  bases for the spaces on each element $V_{\lambda}$  of the grid hierarchy, i.e.,
  $ S_{l,\lambda} = \text{span} \{ \phi_{l,\lambda,i} \}$ and
  $ W_{l,\lambda} = \text{span} \{ \psi_{l,\lambda,i} \}$.
  In particular, to compute  MRA  we need to calculate the mask coefficients
  $(\phi_{l,\lambda,i}, \phi_{l,\mu,j})_{L^2(\Omega)}$,
  $(\phi_{l,\lambda,i}, \psi_{l,\mu,j})_{L^2(\Omega)}$ and
  $(\psi_{l,\lambda,i}, \psi_{l,\mu,j})_{L^2(\Omega)}$.
  In case of a weighted $L^2$-space $L^2_w(\Omega)$ with a weight function $w\in L^1(\Omega,\R_{\geq 0})$ and inner product
  $(f,g)_{L^2_w(\Omega)} := \int_\Omega f(x) g(x) w(x) \diff x$
  this becomes a severe obstruction for MRA-based schemes. In general, for a nonlinear weight function there is no orthogonal preserving affine mapping of the elements of the grid hierarchy onto a reference element. Hence, the mask coefficients have to be computed elementwise for all levels. This leads to increased computational complexity. Contrary for non-weighted spaces,  the mask coefficients can be computed a priori.
  Therefore, the weight should not be included in the norms.
\end{remark}

\subsection{Multiresolution analysis for products of DG spaces}
\label{subseq:MRA-product}
The solution of the deterministic problem \eqref{eq:det_problem} is defined on the product $\Omega=\Omega_1\times\Omega_2$ whereas the stochastic moments \eqref{eq:expectation_value}, \eqref{eq:centralized_moments} of the solution are functions on $\Omega_1$. For convenience of presentation we identify in the following the spatial directions $\bx\in \R^d$ and the stochastic directions $\bxi\in \R^m$ with $\bx_1\in\Omega_1\subset \R^{d_1}$ and $\bx_2\in\Omega_2\subset \R^{d_2}$, respectively.

For the design of an efficient adaptive DG scheme for the deterministic problem \eqref{eq:det_problem} in Section \ref{sec:NumRes} it will be important to understand the  interaction of the  spatial and stochastic variables. For this purpose, we establish here  MRAs of DG spaces for $L^2(\Omega)$ and $L^2(\Omega_1)$.
Since the product space $L^2(\Omega_1)\times L^2(\Omega_2)$ is not isomorphic to $L^2(\Omega)$,
a multiresolution sequence ${\cal S} = \lbrace S_l\rbrace_{l\in\N_0}$ for $L^2(\Omega)$ can not be  constructed  as the product of two multiresolution sequences ${\cal S}^i = \lbrace S^i_l\rbrace_{l\in\N_0}$ for $L^2(\Omega_i)$, $i=1,2$, respectively. In general, the products $S_l:=S_l^1\times S_l^2$ are  not  linear spaces of $L^2(\Omega)$.
Therefore, we construct local bases for the DG spaces and  wavelet spaces on $\Omega$ as products of local bases for the DG spaces and wavelet spaces on $\Omega_i$, $i=1,2$, respectively. Since the local spaces are composed of polynomials and piecewise polynomials, respectively, this is possible due to the following Lemma.

\begin{lemma}{(Basis for product of polynomial spaces).}
  \label{lemma:basisproductpolynomialspace}
  Let be $\bPhi^i := \{ \phi^i_\ibalpha \,:\, \balpha \in \calP_i \}$ a basis for the  space  $\Pi_p(\Omega_i)$, $i=1,2$, of all polynomials of maximal degree $p$ on $\Omega_i\subset \R^{d_i}$. Then a basis of the  space $\Pi_{p}(\Omega)$ of all polynomials of maximal degree $p$ on $\Omega:=\Omega_1\times\Omega_2$ is given by
  \begin{align*}
    \bPhi := \bPhi^1 \times \bPhi^2 =
    \{ \phi_\ibalpha(\bx_1,\bx_2) = \phi^1_{\ibalpha_1}(\bx_1)  \phi^2_{\ibalpha_2}(\bx_2) \,:\,
    \balpha = (\balpha_1,\balpha_2)\in\calP,\ \balpha_i\in \calP_i,\ i=1,2 \} .
  \end{align*}
\end{lemma}
The proof is elementary using  the following notation
\begin{align*}
   & \calP_i = \{ \balpha\in\N_0^{d_i} \,:\, \Vert \balpha \Vert_\infty \le p \},\ i=1,2, \quad
  \calP = \calP_1\times\calP_2 ,                                                                \\
   & \bx = (\bx_1,\bx_2),\ \bx_i\in\Omega_i \subset\R^{d_i},\ i=1,2 ,\quad
  \bx^\ibalpha = \bx_1^{\ibalpha_1} \bx_2^{\ibalpha_2} .
\end{align*}

We emphasize that by means of Fubini the separation of variables allows  for the splitting of  integrals over $(\boldmath{x}_1,\boldmath{x}_2)$ into integrals over $\boldmath{x}_i$, $i=1,2$.

Now let be ${\cal S}^i = \lbrace S^i_l\rbrace_{l\in\N_0}$, $i=1,2$ and
${\cal S} = \lbrace S_l\rbrace_{l\in\N_0}$ multiresolution sequences of DG spaces for $L^2(\Omega_i)$, $i=1,2$ and   $L^2(\Omega)$ with $\Omega=\Omega_1\times\Omega_2$, respectively. Then, these multiresolution sequences are intertwined as follows:

\paragraph{Hierarchy of nested grids:}
Let be $\calG_l^i=\{V^i_\lambda \}_{\lambda\in\calI^i_l}$, $l\in\N_0$,  hierarchies of nested grids on $\Omega_i\subset\R^{d_i}$, $i=1,2$. From this we construct the sequence
$ \calG_l = \{V_\iblambda\}_{\iblambda \in\calI_l}$, $l\in\N_0$,
of grids on the domain $\Omega=\Omega_1\times\Omega_2\subset \R^{d_1+d_2}$ with cells
$ V_\iblambda:=V^1_{\lambda_1} \times V^2_{\lambda_2}$,
$ \blambda:=(\lambda_1,\lambda_2)\in\calI^1_l\times\calI^2_l =:\calI_l$.
Then $\calG_l=\calG^1_l\times \calG^2_l$ is a grid for $\Omega$.
The hierarchy is nested because
$ \overline{V_{\iblambda}} =
  \overline{\bigcup_{\ibmu\in\mathcal{M}_{\iblambda}} V_{\ibmu}}
$
holds for $\blambda=(\lambda_1,\lambda_2)\in \calI_l$
where
$\mathcal{M}_{\iblambda}:= \mathcal{M}^1_{\lambda_1}\times \mathcal{M}^2_{\lambda_2}  \subset \calI^1_{l+1}\times \calI^2_{l+1} = \calI_{l+1}$
is the refinement set of the cell $V_{\blambda}$.
This hierarchy is dense whenever the hierarchies $\calG^i_l$, $i=1,2$, are dense.

\paragraph{Local DG spaces and local complement spaces:}
For $\lambda_i \in\calI^i_l$, $l\in\N_0$, $i=1,2$, the local DG space
$S_{l,\lambda_i}^i$ and the
local complement space $W_{l,\lambda_i}^i$ are spanned by the local bases
\begin{align*}
  \Phi^i_{l,\lambda_i} = \{ \phi^i_{l,\lambda_i, \ibi_i} \,:\, \bi_i\in\calP_i \},\quad
  \Psi^i_{l,\lambda_i} = \{ \psi^i_{l,\lambda_i, \ibi_i, \ibe_i} \,:\, \bi_i\in\calP_i, \be_i\in\calE_i^* \}
\end{align*}
with $\calE_i:=\{0,\ldots, \#\calM^i_{\lambda}-1\}$, $\calE_i^*:=\calE_i\backslash\{0\}$,
$\calP_i:=\{ \balpha\in\N_0^{d_i}\,:\, \Vert \balpha\Vert_\infty \le p-1 \}$.
Due to orthogonality of the global spaces $S_l$ and $W_l$, these local bases need to be orthogonal to each other. Furthermore, we assume that the two bases themselves are orthogonal, i.e., it holds
\begin{align*}
  (\phi^i_{l,\lambda_i, \ibi_i}, \phi^i_{l,\lambda_i, \ibi'_i})_{L^2(V_{\lambda_i})} = \delta_{\ibi_i,\ibi'_i},\
  (\psi^i_{l,\lambda_i, \ibi_i,e_i}, \psi^i_{l,\lambda_i, \ibi'_i,e'_i})_{L^2(V_{\lambda_i})} = \delta_{\ibi_i,\ibi'_i}\delta_{e_i,e'_i},\
  (\phi^i_{l,\lambda_i, \ibi_i}, \psi^i_{l,\lambda_i, \ibi'_i,e'_i})_{L^2(V_{\lambda_i})} = 0
\end{align*}
for $\bi_i,\bi'_i\in\calP_i$, $e_i,e'_i\in\calE^*_i$.

For $\blambda \in\mathcal{I}_l$, $l\in\N_0$, the local DG space
$S_{l,\iblambda}$ and the
local complement space $W_{l,\iblambda}$ are spanned by the local bases
\begin{align*}
  \Phi_{l,\iblambda} = \{ \phi_{l,\iblambda, \ibi} \,:\, \bi\in\calP \},\quad
  \Psi_{l,\iblambda} = \{ \psi_{l,\iblambda, \ibi, \ibe} \,:\, \bi\in\calP, \be\in\calE^* \}
\end{align*}
with $ \calP:= \calP_1\times\calP_2$,
$\calE:= \calE_1\times\calE_2 $, $\calE^*:=\calE\backslash\{\bzero\} = (\calE_1\times \calE_2)\backslash\{\bzero\}$. The basis functions are determined as in Lemma \ref{lemma:basisproductpolynomialspace} by the tensor products
\begin{align*}
  \phi_{l,\iblambda, \ibi}(\bx) = \phi^1_{l,\lambda_1, \ibi_1} (\bx_1)\phi^2_{l,\lambda_2, \ibi_2}(\bx_2),\quad
  \psi_{l,\iblambda, \ibi,\ibe}(\bx) = \psi^1_{l,\lambda_1, \ibi_1, e_1}(\bx_1) \psi^2_{l,\lambda_2, \ibi_2,e_2}(\bx_2)
\end{align*}
for
$\blambda=(\lambda_1,\lambda_2)\in \calI^1_l\times\calI^2_l=\calI_l$,
$\bi=(\bi_1,\bi_2)\in \calP^1\times\calP^2=\calP$,
$\be=(e_1,e_2)\in \calE^*$ and $\bx=(\bx_1,\bx_2)\in\Omega_1\times\Omega_2=\Omega$.
Due to Fubini, orthogonality of the bases $\Phi^i_{l,\lambda_i}$ and $\Psi^i_{l,\lambda_i}$, $i=1,2$, implies orthogonality of
the bases $\Phi_{l,\iblambda}$ and $\Psi_{l,\iblambda}$, i.e.,
\begin{align*}
  (\phi_{l,\iblambda, \ibi}, \phi_{l,\iblambda, \ibi'})_{L^2(V_{\iblambda})} =
  \delta_{\ibi,\ibi'},\quad
  (\psi_{l,\iblambda, \ibi,\ibe}, \psi_{l,\iblambda, \ibi',\ibe'})_{L^2(V_{\iblambda})} =
  \delta_{\ibi,\ibi'}\delta_{\ibe,\ibe'},    \quad
  (\phi_{l,\iblambda, \ibi}, \psi_{l,\iblambda, \ibi',\ibe'})_{L^2(V_{\iblambda})} =
  0
\end{align*}
for $\bi=(\bi_1,\bi_2),\bi'=(\bi_1',\bi_2')\in\calP_1\times\calP_2=\calP$,
$\be=(e_1,e_2),\be'=(e_1',e_2')\in\calE^*$.

\paragraph{}
The previous part is required for the novel adaptation strategy below.

\section{Error analysis for the novel MRA strategy}
\label{sec:MRA-moments}
We investigate the error in the moments by determining an appropriate local norm $\Vert\cdot\Vert_\lambda$ for the local complement space. The error is then bounded asymptotically by a given threshold value.
We first consider the error in the expectation \eqref{eq:expectation_value}.
The error in the higher order moments \eqref{eq:centralized_moments} will then be estimated by means of the error in the expectation.

Here, the main objective of the analysis is the derivation of a  threshold procedure for the multi\-resolution-based adaptive scheme to be introduced in Section \ref{sec:NumRes}. That allows to control the perturbation error in the stochastic moments. Those are  introduced by a perturbation of the underlying
function depending on both the spatial and stochastic variables. This is essential for the efficient performance of the adaptive scheme.
We emphasize that the MRA for the product of DG spaces introduced in Section \ref{subseq:MRA-product} has been tailored  to the efficiency of the adaptive scheme. In particular, due to Remark \ref{rem:weightedspaces} we need a MRA for the product space $L^2(\Omega_1)\times  L^2(\Omega_2)$ rather than $L^2(\Omega_1)\times  L_{p_\xi}^2(\Omega_2)$ -- even so  the latter might be considered to be more natural for the analysis below.
Therefore, we need orthogonality and vanishing moments with respect to $L^2(\Omega_2)$ instead of $L_{p_\xi}^2(\Omega_2)$. Moreover, introducing the basis functions as piecewise polynomials of the product space $\Pi_p(\Omega) = \Pi_p(\Omega_1) \times \Pi_p(\Omega_2)$ according to Lemma \ref{lemma:basisproductpolynomialspace} allows to employ separation of variables in the integrals.

In the following we assume that $\Omega_i\subset \R^{d_i}$, $i=1,2$, $d_1=d,\ d_2=m$, are open bounded domains with Lipschitz boundary.
Note that boundedness will be used in the analysis below. Whereas in Section \ref{sec:problem} we deliberately consider unbounded domains to avoid introducing boundary conditions. In Section \ref{sec:NumRes} the computations are performed on bounded domains using either periodic boundary conditions or constant data.
Furthermore, we assume that $p_\xi\in L^1(\Omega_2)$ is an essentially bounded probability density for an absolutely continuous random variable $\xi$ which is compactly supported on $\Omega_2$.

\begin{theorem}{(Error of expectation)}
  \label{thm:thres-error-1}
  Let be $u^L\in S_L$ and $u^{L,\veps} \in S_L$ its sparse approximation
  \begin{align*}
    u^{L,\veps} :=
    \sum_{\iblambda\in\mathcal{I}_0} u_\iblambda^0 +
    \sum_{l=0}^{L-1} \sum_{\iblambda\in \calD_{L,\veps} \cap \mathcal{I}_l} d^l_\iblambda .
  \end{align*}
  The set $\calD_{L,\veps}$ of significant details is  determined by the local norm
  \begin{align}
    \label{eq:loc-norm-expectation}
    \Vert d^l_\iblambda \Vert_{\iblambda} :=
    \max_{\ibi\in \calP, \ibe\in\calE^*} \{ |d_{l,\iblambda,\ibi,\ibe}|\,  \Vert\psi_{l,\iblambda,\ibi,\ibe}\Vert_{L^2(V_\iblambda)}\}/ \sqrt{|V_\iblambda|}
  \end{align}
  using local threshold values $\veps_{\iblambda,L,q}$
  depending on $q\in\{1,2\}$ with $1/q+1/q'=1$
  such that
  \begin{align}
    \label{eq:loc-thresh-cond-expectation}
    \sum_{l=0}^{L-1} \max_{\iblambda\in\calI_l} \left(\veps_{\iblambda,L,q} \Vert p_\xi \Vert_{L^{q'}(V^2_{\lambda_2})} \right) \le \veps_{\text{max}}.
  \end{align}
  Then the error in the expectation is estimated by
  \begin{align}
    \label{eq:thres-error-expectation-I}
    \Vert \E[u^L]- \E[u^{L,\veps}] \Vert_{L^q(\Omega_1)} \le
    \Vert \E[|u^L - u^{L,\veps}|] \Vert_{L^q(\Omega_1)} \le
    |\Omega|^{1/q} \times \#\calP\times\#\calE^* \times \veps_{\text{max}}.
  \end{align}
  Furthermore, the threshold error in $u$ is bounded by
  \begin{equation}
    \label{eq:u-thresh-error}
    \Vert u^L - u^{L,\veps} \Vert_{L^q(\Omega)} \le C \, |\Omega |^{1/q}\,  W^{-1}_{L,q'}\,\veps_{\text{max}}
  \end{equation}
  for the constant $C:=\sqrt{\#\calP\times\#\calE^*}$  and
  $
    W_{L,q'} := \min_{\lambda\in\calI^2_{L-1}} \Vert p_\xi \Vert_{L^{q'}(V^2_{\lambda})}
  $.
\end{theorem}

\begin{proof}
  Since thresholding is only performed on the details but not on the single-scale coefficients,   the threshold error can be written as
  $
    u^L- u^{L,\veps} =
    \sum_{l=0}^{L-1} \sum_{\iblambda\in \calI_l\backslash \calD_{L,\veps} }d^l_\iblambda .
  $
  The local details $d^l_\iblambda\in W_{l,\iblambda}$ can be expanded in terms of the local wavelet basis, i.e.,
  $
    d^l_\iblambda = \sum_{\ibi\in\calP}\sum_{\ibe\in\calE^*} d_{l,\iblambda,\ibi,\ibe} \psi_{l,\iblambda,\ibi,\ibe}.
  $
  Thus,  the threshold error can be estimated by
  \[
    \|u^L- u^{L,\veps}\|_{L^q(\Omega)} \le
    \sum_{l=0}^{L-1} \sum_{\iblambda\in \calI_l\backslash \calD_{L,\veps} }
    \sum_{\ibi\in\calP}\sum_{\ibe\in\calE^*} |d_{l,\iblambda,\ibi,\ibe}|\,  \|\psi_{l,\iblambda,\ibi,\ibe}\|_{L^q(\Omega)} =  \sum_{l=0}^{L-1} \sum_{\lambda\in\calM_{l,\veps}} |d_{l,\lambda}|\,  \|\psi_{l,\lambda}\|_{L^q(\Omega)}
  \]
  with the set of non-significant details on level $l$ defined as
  \begin{align*}
    \calM_{l,\veps} :=
    \{ \lambda= (\blambda,\bi,\be) \,:\,
    \blambda =(\lambda_1,\lambda_2)\in (\calI^1_l\times \calI^2_l)\backslash \calD_{L,\veps},\
    \bi=(\bi_1,\bi_2)\in\calP_1\times\calP_2,\
    \be=(e_1,e_2)\in\calE^*
    \} .
  \end{align*}

  To investigate the error in the expectation we have to exploit the basis expansion of $d_\blambda^l$ in the $L^q(\Omega_1)$-norm separately for $q=1$ and $q=2$. We show here only the case for $q=1$, for $q=2$ the assertion holds with similar arguments.

  For the expectation in the $L^1(\Omega_1)$-norm we directly conclude by linearity of the expectation
  \begin{align*}
    \Vert \E[u^L- u^{L,\veps}] \Vert_{L^1(\Omega_1)} \le
    \sum_{l=0}^{L-1} \sum_{\lambda\in\calM_{l,\veps}} |d_{l,\lambda}|\,
    \Vert \E[|\psi_{l,\lambda}|]\Vert_{L^1(\Omega_1)}.
  \end{align*}
  The expectation of the modulus of the wavelet functions can be estimated employing separation of variables
  \begin{align*}
    \E[|\psi_{l,\lambda}|](\bx_1) =
    \int_{\Omega_2}   |\psi_{l,\iblambda,\ibi,\ibe}(\bx_1,\bx_2)| p_\xi(\bx_2) \dif\bx_2 =
    |\psi^1_{l,\lambda_1,\ibi_1,e_1}(\bx_1)| (|\psi^2_{l,\lambda_2,\ibi_2,e_2}|, p_\xi )_{L^2(\Omega_2)}.
  \end{align*}
  Furthermore, by the Cauchy-Schwarz  inequality and the support of the wavelet functions it holds
  \begin{align*}
     & \Vert \psi^1_{l,\lambda_1,\ibi_1,e_1} \Vert_{L^1(\Omega_1)}
    \le \Vert \psi^1_{l,\lambda_1,\ibi_1,e_1} \Vert_{L^2(\Omega_1)} \sqrt{|V^1_{\lambda_1}|}, \\
     & ( |\psi^2_{l,\lambda_2,\ibi_2,e_2}| , p_\xi)_{L^2(\Omega_2)}  \le
    \Vert \psi^2_{l,\lambda_2,\ibi_2,e_2} \Vert_{L^2(\Omega_2)}
    \Vert p_\xi \Vert_{L^2(V^2_{\lambda_2})} \le
    \Vert \psi^2_{l,\lambda_2,\ibi_2,e_2} \Vert_{L^2(\Omega_2)}
    \Vert p_\xi \Vert_{L^\infty(V^2_{\lambda_2})}  \sqrt{|V^2_{\lambda_2}|} .
  \end{align*}
  This yields
  \begin{align*}
    \Vert \E[|\psi_{l,\lambda}|] \Vert_{L^1(\Omega_1)} & =
    \Vert \psi^1_{l,\lambda}\Vert_{L^1(\Omega_1)} (|\psi^2_{l,\lambda_2,\ibi_2,e_2}|, p_\xi )_{L^2(\Omega_2)}                                              \\
                                                       & \le
    \Vert \psi^1_{l,\lambda_1,\ibi_1,e_1} \Vert_{L^2(\Omega_1)} \sqrt{|V^1_{\lambda_1}|}\,
    \Vert \psi^2_{l,\lambda_2,\ibi_2,e_2} \Vert_{L^2(\Omega_2)}
    \Vert p_\xi \Vert_{L^\infty(V^2_{\lambda_2})}  \sqrt{|V^2_{\lambda_2}|}                                                                                \\
                                                       & = \|\psi_{l,\lambda}\|_{L^2(\Omega)}\, \sqrt{|V_\lambda|}\, \|p_\xi\|_{L^\infty(V^2_{\lambda_2})}
  \end{align*}
  using Fubini on the tensor product of the bases on each cell $V_\lambda$. Combining the above estimates we conclude with
  \begin{align*}
     & \Vert \E[|u^L- u^{L,\veps}|] \Vert_{L^1(\Omega_1)} \le
    \sum_{l=0}^{L-1} \sum_{\lambda\in\calM_{l,\veps}} |d_{l,\lambda}|\,   \Vert \psi_{l,\lambda} \Vert_{L^2(\Omega)}\, \sqrt{|V_\lambda|} \, \Vert p_\xi \Vert_{L^\infty(V^2_{\lambda_2})}                                                \\
     & = \sum_{l=0}^{L-1} \sum_{\iblambda=(\lambda_1,\lambda_2)\in (\calI^1_l\times \calI^2_l)\backslash \calD_{L,\veps} } \sum_{\ibi=(\ibi_1,\ibi_2)\in\calP_1\times\calP_2}\sum_{\ibe=(e_1,e_2)\in\calE^*} |d_{l,\iblambda,\ibi,\ibe} |
    \Vert \psi_{l,\iblambda,\ibi,\ibe} \Vert_{L^2(\Omega)}\, \sqrt{|V_\iblambda|} \, \Vert p_\xi \Vert_{L^\infty(V^2_{\lambda_2})} .
  \end{align*}
  Applying the definition of the local norm \eqref{eq:loc-norm-expectation} we obtain for $\blambda\in \calI_l$
  \begin{align*}
    |d_{l,\iblambda,\ibi,\ibe} |
    \Vert \psi_{l,\iblambda,\ibi,\ibe} \Vert_{L^2(\Omega)} \, \sqrt{|V_\iblambda|}\le
    \Vert d^l_\iblambda \Vert_{\iblambda} |V_\iblambda| .
  \end{align*}
  Using the local threshold value $\varepsilon_{\blambda, L, 1}$, non-significant details can be estimated based on assumption \eqref{eq:significant_details} by
  $
    \Vert d^l_\iblambda \Vert_\iblambda \le \veps_{\iblambda,L,1}
  $
  and by definition of the discretization it holds
  $
    \sum_{\iblambda\in \calI_l} |V_{\iblambda}| = |\Omega|,\,l=0,\ldots, L.
  $
  Then the error can be further estimated by
  \begin{align*}
     & \Vert \E[|u^L - u^{L,\veps}|] \Vert_{L^1(\Omega_1)}                                                                                                                                                                                                   \\
     & \le \sum_{l=0}^{L-1} \sum_{\iblambda=(\lambda_1,\lambda_2)\in (\calI^1_l\times \calI^2_l)\backslash \calD_{L,\veps} } \sum_{\ibi=(\ibi_1,\ibi_2)\in\calP_1\times\calP_2}\sum_{\ibe=(e_1,e_2)\in\calE^*} \veps_{\iblambda,L,1} |V_{\iblambda}|
    \Vert p_\xi \Vert_{L^{\infty}(V^2_{\lambda_2})}                                                                                                                                                                                                          \\
     & \le \sum_{l=0}^{L-1} \sum_{\iblambda=(\lambda_1,\lambda_2)\in (\calI^1_l\times \calI^2_l) }
    \sum_{\ibi=(\ibi_1,\ibi_2)\in\calP_1\times\calP_2}\sum_{\ibe=(e_1,e_2)\in\calE^*} \veps_{\iblambda,L,1} |V_{\iblambda}| \Vert p_\xi \Vert_{L^{\infty}(V^2_{\lambda_2})}                                                                                  \\
     & \le |\Omega|\times \#\calP\times\#\calE^* \sum_{l=0}^{L-1} \max_{\iblambda=(\lambda_1,\lambda_2)\in\calI_l}\veps_{\iblambda,L,1} \Vert p_\xi \Vert_{L^{\infty}(V^2_{\lambda_2})}\leq  |\Omega|\times \#\calP\times\#\calE^* \times \veps_{\text{max}}
  \end{align*}
  using assumption \eqref{eq:loc-thresh-cond-expectation} with $q=1$, thus, $q'=\infty$.

  Finally, to investigate the threshold error in $u$  we may apply \eqref{eq:threshold-error-gerhard}.
  For this purpose, we have to verify the condition \eqref{eq:loc-thres} on the local threshold values:
  \begin{align*}
    \sum_{l=0}^{L-1} \max_{\iblambda\in \calI_l} \veps_{\iblambda,L,q} \le
    \sum_{l=0}^{L-1} \max_{\iblambda\in \calI_l} \veps_{\iblambda,L,q} \Vert p_\xi \Vert_{L^{q'}(V^2_{\lambda_2})}/ \min_{\lambda\in \calI^2_l}\Vert p_\xi \Vert_{L^{q'}(V^2_{\lambda})}
    \le \veps_{\text{max}}/\min_{\lambda\in \calI^2_l,l=0,\ldots,L-1}\Vert p_\xi \Vert_{L^{q'}(V^2_{\lambda})}.
  \end{align*}
  Since the grids are nested and
  $ \Vert p_\xi \Vert_{L^{q'}(V)} \le \Vert p_\xi \Vert_{L^{q'}(V')} $  for $V\subset V'\subset \Omega_2$ it holds
  \begin{align*}
    \min_{\lambda\in \calI^2_l,l=0,\ldots,L-1}\Vert p_\xi \Vert_{L^{q'}(V^2_{\lambda})} =  \min_{\lambda\in \calI^2_{L-1}} \Vert p_\xi \Vert_{L^{q'}(V^2_{\lambda})} = W_{L,q'}.
  \end{align*}
  From this we finally conclude \eqref{eq:u-thresh-error} for the threshold error.
  \qed

\end{proof}

\begin{remark}{(Choice of local threshold value)}
  For a dyadic (Cartesian) grid hierarchy  we choose
  \begin{align}
    \label{eq:locthreshval}
    \veps_{\iblambda,L,q} = \frac{h_L}{h_l} \Vert p_\xi \Vert^{-1}_{L^{q'}(V^2_{\lambda_2})}  \veps_{\text{max}},\ \blambda=(\lambda_1,\lambda_2)\in\calI^1_l\times \calI^2_l,
  \end{align}
  as local threshold value where $h_l$ denotes the uniform diameter of the cells on level $l$. If $p_\xi$ is locally small, then the local threshold  value becomes very large and large details can be neglected without significantly contributing to the threshold error of the expectation $\E$ whereas the threshold error might be large for $u$.
  This will be the key ingredient to improve the efficiency of the adaptive scheme in Section \ref{sec:NumRes}.
  \begin{enumerate}
    \item
          To ensure uniform boundedness of the error in the expectation we have to verify the sufficient condition \eqref{eq:loc-thresh-cond-expectation}. Due to dyadic grid refinement it holds $h_L/h_l=a^{l-L}$ for $a>1$.
          Then \eqref{eq:loc-thresh-cond-expectation} holds because
          \begin{align*}
            \sum_{l=0}^{L-1} \max_{\iblambda\in\calI_l} \left(\veps_{\iblambda,L,q} \Vert p_\xi \Vert_{L^{q'}(V^2_{\lambda_2})} \right) = \veps_{\text{max}}\,\sum_{l=0}^{L-1}\frac{h_L}{h_l} = \veps_{\text{max}}\,\sum_{l=1}^{L}\left(\frac{1}{a}\right)^l
            \le \veps_{\text{max}} .
          \end{align*}
    \item
          According to \eqref{eq:u-thresh-error} the threshold error in $u$ is bounded by $ C \, |\Omega |^{1/q}\,  W^{-1}_{L,q'}\,\veps_{\text{max}}$
          where $W_{L,q'}$ depends on $L$. This is not admissible from the asymptotic point of view. However, in case of $q=1$, $q'=\infty$ it holds
          \begin{align*}
            \Vert p_\xi \Vert_{L^{q'}(V)} \ge \min_{\bx_2\in\Omega_2} |p_\xi(\bx_2)| >0,\qquad \forall\,V\subset\Omega_2,
          \end{align*}
          and the factor $W_{L,q'}$ as well as the right-hand side in  \eqref{eq:u-thresh-error} are independent of $L$.
  \end{enumerate}
\end{remark}

To estimate the error for the higher order centralized moments induced by the threshold error of the underlying DG approximation we derive estimates for the expectation.
Since the entropy solution of a scalar conservation law in multidimensions satisfies a maximum principle, we may confine our stability investigation of the error for the expectation and the centralized $k$-th moments to functions $u\in L^\infty(\Omega)$. Assuming that $\Omega$ is a bounded domain, it also holds $u\in L^2(\Omega)$. Then the projection $u^L\in S_L$ of $u$ onto $S_L$ is uniquely defined.
In practice, $u^L(t,\cdot)$ will be the DG approximation for a fixed time $t\in (0,T)$ that is assumed to converge to the entropy solution $u(t,\cdot)$, i.e.,
  $\Vert u^L(t, \cdot) - u(t, \cdot) \Vert_{L^1(\Omega)} \to 0$, $L\to\infty$ where $u_L(t,\cdot)$ and $u(t,\cdot)$ are uniformly bounded. Note that in this case $u^L(t,\cdot) \ne P_{S_L}(u(t,\cdot))$, i.e., $u^L(t,\cdot)$ is not the projection of $u(t,\cdot)$ onto $S_L$.

\begin{lemma}{(Estimates for expectation)}
  \label{la:estimates-expectation}
  Let be $u\in L^\infty(\Omega)$ and $p_\xi\in L^\infty(\Omega_2)$. Then it holds for $k\in\N$:
  \begin{align}
    \label{eq:estimate-expectation-1}
     & \Vert \E[u^k] \Vert_{L^q(\Omega_1)} \le
    \Vert p_\xi \Vert_{L^\infty(\Omega_2)}\,\Vert u^k \Vert_{L^q(\Omega)}  ,\quad q\in[1,\infty)                                                              \\
    \label{eq:estimate-expectation-2}
     & \Vert \E[u^k] \Vert_{L^1(\Omega_1)} \le  \E[\Vert u^k \Vert_{L^1(\Omega_1)}] \le \Vert p_\xi \Vert_{L^\infty(\Omega_2)}\,\Vert u \Vert^k_{L^k(\Omega)} \\
    \label{eq:estimate-expectation-3}
     & \Vert \E[u^k] \Vert_{L^\infty(\Omega_1)} \le  \Vert u \Vert^k_{L^\infty(\Omega)}
    \\
    \label{eq:estimate-expectation-1b}
     & \Vert \E^k[u] \Vert_{L^q(\Omega_1)} \le
    \Vert p_\xi \Vert^k_{L^\infty(\Omega_2)} \Vert u^k \Vert_{L^q(\Omega)},\quad q\in[1,\infty)
    \\
    \label{eq:estimate-expectation-1c}
     & \Vert \E^k[u] \Vert_{L^\infty(\Omega_1)} \le  \Vert u \Vert^k_{L^\infty(\Omega)}.
  \end{align}
\end{lemma}

\begin{lemma}
  \label{la:error-differences-expectation-moments}
  Let be $u,v\in L^\infty(\Omega)$. Assuming that the probability density function is uniformly bounded, i.e., $p_\xi\in L^\infty(\Omega_2)$, then it holds for $q\in [1,\infty]$ and all $k\in\N$:
  \begin{align}
    \label{eq:error-expectation-k-1}
    \Vert \E[u^k] - \E[v^k] \Vert_{L^q(\Omega_1)} & \le
    k\,(M(u,v))^{k-1} \, \Vert \E[|u-v|] \Vert_{L^q(\Omega_1)} , \\
    \label{eq:error-expectation-k-2}
    \Vert \E^k[u] - \E^k[v] \Vert_{L^q(\Omega_1)} & \le
    k\,(M_\E(u,v))^{k-1} \, \Vert \E[u-v] \Vert_{L^q(\Omega_1)} \le
    k\,(M(u,v))^{k-1} \, \Vert \E[|u-v|] \Vert_{L^q(\Omega_1)}
  \end{align}
  with
  \begin{align}
    \label{eq:M-E}
    M_\E(u,v) & := \max\{ \Vert \E[u] \Vert_{L^\infty(\Omega_1)}, \Vert \E[v] \Vert_{L^\infty(\Omega_1)} \}\le \max\{ \Vert u \Vert_{L^\infty(\Omega)}, \Vert v \Vert_{L^\infty(\Omega)} \} =: M(u,v).
  \end{align}
\end{lemma}

The proofs of Lemma \ref{la:estimates-expectation} and Lemma \ref{la:error-differences-expectation-moments} are given in Appendix \ref{appendix:proof_lemmata}.
By means of these estimates we may now verify the following  stability result for the  expectation and the higher order moments.

\begin{lemma}{(Stability of expectation and higher order moments)}
  \label{la:stability-expectation-moments}
  Let be $u,v \in S_L$. Assuming that the probability density function is uniformly bounded, i.e., $p_\xi\in L^\infty(\Omega_2)$, then the differences in the expectation and the higher order moments for $k\in\N$ can be estimated in the $L^q$-norm, $q\in [1,\infty)$, by the differences of the functions $u$ and $v$:
  \begin{align}
    \label{eq:difference-moments-1}
     & \Vert \E[u] - \E[v] \Vert_{L^q(\Omega_1)}  \le
    \Vert p_\xi \Vert_{L^\infty(\Omega_2)}  \Vert u - v\Vert_{L^q(\Omega)},                        \\
    \label{eq:difference-moments-2}
     & \Vert \M^k[u] - \M^k[v] \Vert_{L^q(\Omega_1)} \le |\Omega|^{1/qp}\, c_k\left(u,v\right)  \,
    \max\{ \Vert p_\xi \Vert_{L^\infty(\Omega_2)}, \Vert p_\xi \Vert^k_{L^\infty(\Omega_2)}\}  \,
    \Vert \E[ |u - v| ] \Vert_{L^{qp'}(\Omega_1)},
  \end{align}
  where $p,p'\in [1,\infty]$ such that $1/p+1/p'=1$ and
  \begin{align}
    \label{eq:ck}
    c_k\left(u,v\right) :=
     & \sum_{j=0}^k  \binom{k}{j} (
    2k - j )\,M(u,v)^{2k-j-1},
  \end{align}
  where $M(u,v)$ is defined by \eqref{eq:M-E}.
  \\
  For the case $q=\infty$, we estimate the differences in the expectation and the higher order moments by:
  \begin{align*}
     & \Vert \E[u] - \E[v] \Vert_{L^\infty(\Omega_1)}  \le \Vert u - v\Vert_{L^\infty(\Omega)}, \\
     & \Vert \M^k[u] - \M^k[v] \Vert_{L^\infty(\Omega_1)} \le c_k\left(u,v\right)\,
    \Vert \E[ |u - v| ] \Vert_{L^{\infty}(\Omega_1)}.
  \end{align*}
  Here we set $|\Omega|^{1/\infty}=1$ for a convention.
\end{lemma}

\begin{proof}~\\
  Due to the linearity of the expectation, inequality \eqref{eq:difference-moments-1} follows by  \eqref{eq:estimate-expectation-1}.
  The error of the $k$-th centralized moments is
  \begin{align*}
     & \Vert \M^k[u] - \M^k[v] \Vert_{L^q(\Omega_1)} =
    \left\Vert\sum_{j=0}^k  \binom{k}{j}  \Big(\E[u^{k-j}]\, \E^k[u] -
    \E[v^{k-j}]\, \E^k[v]\Big)\right\Vert_{L^q(\Omega_1)}                                    \\
     & \le \sum_{j=0}^k  \binom{k}{j} \left(
    \left\Vert \E^k[u]
    \left(\E[u^{k-j}] - \E[v^{k-j}]\, \right) \right\Vert _{L^q(\Omega_1)} +
    \left\Vert \E[v^{k-j}]
    \left(\E^k[u] - \E^k[v]\, \right) \right\Vert _{L^q(\Omega_1)}
    \right)                                                                                  \\
     & \le \sum_{j=0}^k  \binom{k}{j} \left(
    \left\Vert \E^k[u] \right\Vert _{L^{qp}(\Omega_1)}\,
    \left\Vert\E[u^{k-j}] - \E[v^{k-j}]\, \right\Vert _{L^{qp'}(\Omega_1)} \right. \nonumber \\
     & \hspace{20mm} \left. +
    \left\Vert \E[v^{k-j}] \right\Vert _{L^{qp}(\Omega_1)}\,
    \left\Vert\E^k[u] - \E^k[v]\, \right\Vert _{L^{qp'}(\Omega_1)}
    \right)                                                                                  \\
     & \le c_{k,qp}\left(u,v\right) \, \Vert \E[ |u - v| ] \Vert_{L^{qp'}(\Omega_1)}
  \end{align*}
  with
  \begin{align*}
     & c_{k,qp}\left(u,v\right) :=  \sum_{j=0}^k  \binom{k}{j} \left(
    \left\Vert \E^k[u] \right\Vert _{L^{qp}(\Omega_1)}\,
    (k-j)\,M(u,v)^{k-j-1}   \right.                                   \\
     & \hspace{40mm} \left. +
    \left\Vert \E[v^{k-j}] \right\Vert _{L^{qp}(\Omega_1)}\,
    k\,M(u,v)^{k-1}
    \right) .
  \end{align*}
  Then we can estimate the coefficients $c_{k,qp}$ by
  \begin{align*}
    c_{k,qp}\left(u,v\right) & \le
    \sum_{j=0}^k  \binom{k}{j} \left( |\Omega|^{1/qp}\,
    \Vert p_\xi \Vert^k_{L^\infty(\Omega_2)}\,  \Vert u \Vert^k_{L^\infty(\Omega)} \,
    (k-j)\,M(u,v)^{k-j-1}   \right.                                                                                                                         \\
                             & \hspace{20mm} \left. +\ |\Omega|^{1/qp}\,
    \Vert p_\xi \Vert_{L^\infty(\Omega_2)}\,  \Vert v \Vert^{k-j}_{L^\infty(\Omega)}\,
    k\,M(u,v)^{k-1} \,
    \right)                                                                                                                                                 \\
                             & \leq  |\Omega|^{1/qp}\,c_k(u,v)\, \max\{ \Vert p_\xi \Vert_{L^\infty(\Omega_2)}, \Vert p_\xi \Vert^k_{L^\infty(\Omega_2)}\}.
  \end{align*}
  For the case $q=\infty$, we analogously estimate $c_{k,\infty}$ by using \eqref{eq:estimate-expectation-3} and \eqref{eq:estimate-expectation-1c}.
  \qed
\end{proof}

From the stability result we may now state the main result on the error in  the higher order moments deduced from the error in the expectation.

\begin{theorem}{(Error in higher order moments)}
  \label{thm:thres-error-1-moments}
  Let be $u^L\in S_L$ and $u^{L,\veps} \in S_L$ its sparse approximation determined by applying thresholding to its multiscale decomposition such that
  \begin{align}
    \label{eq:pertub-error-exp}
    \Vert \E[|u^L - u^{L,\veps}|] \Vert_{L^q(\Omega_1)} \le
    C_\E\, \veps_{\text{max}},\qquad q\in\{1,2\}
  \end{align}
  with $C_\E$ independent of $L$ and $\veps_{\text{max}}$.
  Assuming that the probability density function is uniformly bounded, i.e., $p_\xi\in L^\infty(\Omega_2)$,
  then the error in the higher order moments for $k\in\N$ can be estimated by
  \begin{align*}
     & \Vert \M^k[u^L] - \M^k[u^{L,\veps}] \Vert_{L^q(\Omega_1)} \le c_k\left(u^L,u^{L,\veps}\right)\,
    \max\{ \Vert p_\xi \Vert_{L^\infty(\Omega_2)}, \Vert p_\xi \Vert^k_{L^\infty(\Omega_2)}\}
    \, C_\E\, \veps_{\text{max}}
  \end{align*}
  where the coefficients $c_k$ are defined by \eqref{eq:ck}.
\end{theorem}

\begin{proof}
    Since $u^L,u^{L,\varepsilon} \in S_L$ are piecewise polynomials of fixed degree defined on a bounded domain $\Omega$, these functions are bounded, i.e., $u^L,u^{L,\varepsilon} \in L^\infty(\Omega)$. 
    Therefore,
    we may apply Lemma \ref{la:stability-expectation-moments} choosing $p=\infty$ and, thus $p' = 1$. Here we directly use \eqref{eq:pertub-error-exp} in \eqref{eq:difference-moments-2}.
    \qed
  \end{proof}

We emphasize that the assumption \eqref{eq:pertub-error-exp} is satisfied by \eqref{eq:thres-error-expectation-I} when performing thresholding using the local norm \eqref{eq:loc-norm-expectation} with local threshold values \eqref{eq:loc-thresh-cond-expectation} according to Thm.~\ref{thm:thres-error-1}.

By the same arguments this result extends to the limit $L\rightarrow\infty$.

\begin{theorem}{(Convergence of expectation and higher order moments)}
  \label{thm:convergence-expectation-moments}
  Fix $q\in [1,\infty]$.
  Let be $u\in L^q(\Omega)$ the limit of the sequence  $\{u^L\}_{L\in\N} \subseteq L^q(\Omega)$, i.e.,
  \begin{align}
    \label{eq:convergence-2}
    \Vert u - u^L \Vert_{L^q(\Omega)} \to 0,\quad L\to\infty.
  \end{align}
  Assume that the probability density function is  bounded, i.e., $p_\xi\in L^\infty(\Omega_2)$.
  If $u$ and $u_L$ are uniformly bounded, i.e., there exists a constant $0<C<\infty$ independent of $L$ such that
  \begin{align}
    \label{eq:uniform-boundedness-2}
    \Vert u \Vert_{L^\infty(\Omega)} \le C,\quad
    \Vert u^L \Vert_{L^\infty(\Omega)} \le C\quad \forall\,L\in\N,
  \end{align}
  then the error in the expectation and the higher order moments for $k\in\N$ and $q\in[1,\infty)$ can be estimated by
  \begin{align}
    \label{eq:error-moments-1}
     & \Vert \E[u] - \E[u^L] \Vert_{L^q(\Omega_1)}  \le
    \Vert p_\xi \Vert_{L^\infty(\Omega_2)}  \Vert u - u^L \Vert_{L^q(\Omega)}, \\
    \label{eq:error-moments-2}
     & \Vert \M^k[u] - \M^k[u^L] \Vert_{L^q(\Omega_1)} \le \overline{C} \,
    \max\{ \Vert p_\xi \Vert_{L^\infty(\Omega_2)}, \Vert p_\xi \Vert^k_{L^\infty(\Omega_2)}\}
    \,  \Vert p_\xi \Vert_{L^\infty(\Omega_2)}\, \Vert u - u^L \Vert_{L^q(\Omega)}
    \intertext{and for $q=\infty$ by}
    \notag
     & \Vert \E[u] - \E[u^L] \Vert_{L^\infty(\Omega_1)}  \le
    \Vert u - u^L \Vert_{L^\infty(\Omega)},                                    \\
    \notag
     & \Vert \M^k[u] - \M^k[u^L] \Vert_{L^\infty(\Omega_1)} \le \overline{C}
    \, \Vert u - u^L \Vert_{L^\infty(\Omega)}
  \end{align}
  with
  \begin{align}
    \label{eq:ck-bound-2}
    \overline{C} =
    \sum_{j=0}^k  \binom{k}{j} \left(2k-j \right)\,C^{2k-j-1}.
  \end{align}
  In particular, the expectation $\E[u^L]$ and the centralized higher order moments $\M^k[u^L]$ converge to
  $\E[u]$ and $\M^k[u]$ in $L^q(\Omega_1)$.
\end{theorem}

\section{Numerical investigations}
\label{sec:NumRes}

We present computational results using Theorem \ref{thm:convergence-expectation-moments} and the local thresholding strategy \eqref{eq:locthreshval} for the stochastic Burgers' equation, see Sect.~\ref{subsec:burgers_cauchy_smooth}, and the random Euler equations, see Sect.~\ref{subsec:euler_cauchy}.

\subsection{Numerical Method}
\label{subsec:NumMeth}

For the approximation of the deterministic Cauchy problem \eqref{eq:det_problem} we  apply a Runge-Kutta discontinuous
Galerkin method \cite{Cockburn:1998jt} on Cartesian grids using quadratic polynomials, i.e., $p=3$,
and an explicit third-order SSP-Runge-Kutta method with three stages for the time-discretization. As numerical flux we choose the local Lax-Friedrichs flux with the Shu limiter \cite{Cockburn:1998jt}.

To enhance the performance of the DG solver it is combined with local grid refinement that allows for adaptation in \textit{both} the spatial and the stochastic directions.
For this purpose, we employ  \emph{multiresolution-based grid adaptation}.
This concept belongs to the class of perturbation methods.
Following the work by \cite{HovhannisyanMuellerSchaefer-2014} the DG solver is intertwined with the MRA in
Sect.~\ref{subseq:MRA-product}.
In each time step $t_n$ the adaptive grid $\calG^n_{L,\veps}$ is determined by means of the set of significant details $\calD^n_{L,\veps}$ corresponding to the DG approximation $u_{L,\veps}^n$
where the cells in the grid hierarchy are refined as long as there exists a significant detail. One time step consists of the following three steps summarized in Algorithm \ref{alg:adap}.
Here we apply the MRA to products of DG spaces according to Sect.~\ref{sec:MRA-DG}.
Essential for the performance of the adaptive solver is the choice of the threshold value $\veps_{\text{max}}$ and the prediction strategy.

\begin{algorithm}
  \caption{Timestep Multiresolution DG scheme}
  \begin{enumerate}
    \item \textbf{Grid refinement}:
          \begin{enumerate}
            \item Perform a local multiscale transformation to determine the multiscale
                  decomposition of $u^n_{L,\veps}$
                  and the set of significant details $\calD^n_{L,\veps}$.
            \item Determine a  prediction set $\tilde{\calD}^n_{L,\veps}\supset \calD^n_{L,\veps} $ by means of $\calD^n_{L,\veps}$.
            \item Perform a local inverse multiscale transformation to determine the adaptive grid $\tilde{\calG}^n_{L,\veps}$ from the prediction set $\tilde{\calD}^n_{L,\veps}$ and the corresponding single-scale representation $\tilde{u}_{L,\veps}^n$.
          \end{enumerate}
    \item \textbf{Time evolution}:\\
          Perform Runge-Kutta time evolution on the single-scale representation $\tilde{u}_{L,\veps}^n$ to compute $\tilde{u}_{L,\veps}^{n+1}$ where on each stage limiting is performed on all
          elements of the adaptive grid $\tilde{\calG}^{n+1}_{L,\veps}$ on the finest level.
    \item \textbf{Grid coarsening}:
          \begin{enumerate}
            \item Perform a local multiscale transformation to determine the multiscale
                  decomposition of $\tilde{u}_{L,\veps}^{n+1}$.
            \item Determine the set of significant details $\mathcal{D}_\veps^{n+1}$ by applying hard thresholding to  $\tilde{\calD}_{L,\veps}^{n+1}$.
            \item Perform a local inverse multiscale transformation to determine the adaptive grid $\calG^{n+1}_{L,\veps}$ from the set $\calD^{n+1}_{L,\veps}$  and the corresponding single-scale representation $u_{L,\veps}^{n+1}$.
          \end{enumerate}
  \end{enumerate}
  \label{alg:adap}
\end{algorithm}

Since the adaptive multiresolution based DG solver has been subject of numerous publications, we abstain from presenting the details of the solver except for the ingredients that have been modified for our purposes, namely, the threshold process. More details on the adaptive solver can be found in \cite{GerhardMueller-2016,GerhardIaconoMayMueller-2015}.
\begin{itemize}
  \item
        The set of significant details $\calD_{L,\veps}$ is determined by
        the local norm \eqref{eq:loc-norm-expectation} using local threshold values
        $\veps_{\iblambda,L}$ according to  \eqref{eq:loc-thresh-uniform} or
        local threshold values $\veps_{\iblambda,L,q}$ determined by \eqref{eq:locthreshval} for $q=1$. The two
        alternatives will be distinguished in the following by \textit{uniform} thresholding and \textit{weighted} thresholding, respectively.
        We emphasize that grid adaptation is performed always in both the spatial and the stochastic directions.
        However, when using weighted thresholding a detail may be non-significant if the norm of the local weight $p_\xi$ is small although, the detail is large. This will cause a larger error in the solution of the deterministic problem but not in its stochastic moments.
  \item
        To ensure that all significant details at the old \textit{and} the new time level are adequately resolved when performing the time step,  we need to inflate the set of significant details $\calD^n_{L,\veps}$ where we flag non-significant details at the old time to become significant at the new time. This results in the prediction set $\tilde{\calD}^n_{L,\veps} \supset \calD^n_{L,\veps}$. In the original adaptive scheme prediction is applied to all directions. Since in the deterministic problem \eqref{eq:det_problem} there is no flow in the stochastic directions, information cannot propagate in this direction. Therefore, prediction is only performed in the spatial directions although grid adaptation is performed in both directions. In our computations we apply Harten's prediction strategy to determine the prediction set $\tilde{\calD}^n_{L,\veps}$ in Algorithm \ref{alg:adap}, cf.~\cite{Gerhard:2017}.

  \item
        For the global threshold value $\veps_{\text{max}}$ we apply the heuristic strategy developed in \cite{GerhardIaconoMayMueller-2015}
        \begin{equation}
          \label{eq:global_threshold_error}
          \veps_{\text{max}} = C_\text{heuristic}\, h_L^{\beta}
        \end{equation}
        with $\beta$ the order of the discretization error. Since in our computations the solution always exhibit discontinuities, the order of convergence will be bounded by one ($\beta=1$).
        The choice of $C_\text{heuristic}$ is problem-dependent and will be discussed below.
        The heuristic strategy  was verified numerically to preserve the accuracy of the reference scheme
        when choosing Harten's prediction strategy.
\end{itemize}

\begin{remark}[On the reliability of the adaptive scheme.]
  \label{rem:reliability}
  When performing the adaptive scheme a threshold error is introduced in each time step. In Section \ref{sec:MRA-moments} we investigated the perturbation error caused by thresholding the data. In particular, the thresholding has been designed such that the perturbation error in the stochastic moments is uniformly bounded, see
  Thm.~\ref{thm:thres-error-1-moments}.
  In the following we discuss the uniform boundedness of the agglomerated perturbation error over \textit{all} time steps:\\
  Let be $u\in L^1(\Omega)\cap L^\infty(\Omega)$ the entropy solution of a scalar conservation law in multi-dimensions at some finite time $0 < T < \infty$. By  $\{u^L\}_{L\in\N} \in \calS$ we denote a sequence of (DG-) approximations that is assumed to converge to $u$ in $L^1(\Omega)$, i.e.,
  \begin{align}
    \label{eq:convergence-3}
    \Vert u - u^L \Vert_{L^1(\Omega)} \to 0,\quad L\to\infty.
  \end{align}
  (Note that $u^L$ is not the $L^2$-projection of $u$.)
  The sequence is assumed to be uniformly bounded, i.e., \eqref{eq:uniform-boundedness-2} holds true.
  The probability density function is assumed to be bounded, i.e., $p_\xi\in L^\infty(\Omega_2)$. Then the error in the expectation and the centralized $k$-th moments are bounded by \eqref{eq:error-moments-1} and \eqref{eq:error-moments-2}, respectively, for $q=1$. In particular, the expectation $\E[u^L]$ and the centralized higher order moments $\M^k[u^L]$ converge to
  $\E[u]$ and $\M^k[u]$ in $L^1(\Omega_1)$.

  Let $u^{L,\veps}\in S_L$ be the uniformly bounded  solution of the adaptive DG scheme,
  i.e., $u^{L,\veps}$ is not the approximation of $u^L$ obtained by thresholding of its multiscale decomposition according to Theorem \ref{thm:thres-error-1}.
  The prediction strategy is assumed to be reliable, i.e.,
  \begin{align}
    \label{eq:reliabilitycond}
    \tilde{\calD}^n_{L,\veps} \supset \calD^n_{L,\veps} \cup  \calD^{n+1}_{L,\veps}
  \end{align}
  and the
  threshold value $\veps_{\text{max}}$ is chosen such that the  perturbation error is uniformly bounded, i.e.,
  \begin{align}
    \label{eq:perturbationerror-agglomerated}
    \Vert u^L- u^{L,\veps} \Vert_{L^1(\Omega)} \le
    C_{\text{thres}}\,\veps_\text{max},
  \end{align}
  where the constant $C_\text{thres}$ is independent of the discretization and the threshold value but may depend on the final time $T$, the initial data $u_0$, the Lipschitz constant of the numerical flux, the CFL number, etc..
  Then we may estimate the error due to discretization and perturbation by  \eqref{eq:error-moments-1}, \eqref{eq:error-moments-2}
  and \eqref{eq:difference-moments-1},  \eqref{eq:difference-moments-2} of Lemma \ref{la:stability-expectation-moments}
  \begin{align}
    \label{eq:error-discr-pert-moments-1}
    \Vert \E[u] - \E[u^{L,\veps}] \Vert_{L^q(\Omega_1)}     & \le
    \Vert \E[u] - \E[u^L] \Vert_{L^q(\Omega_1)} + \Vert \E[u^L]  - \E[u^{L,\veps}] \Vert_{L^q(\Omega_1)}  \nonumber                                         \\
                                                            & \le
    \Vert p_\xi \Vert_{L^\infty(\Omega_2)}\left(  \Vert u - u^L \Vert_{L^q(\Omega)} + C_\text{thres}\,\veps_\text{max} \right)                              \\[2mm]
    \label{eq:error-discr-pert-moments-2}
    \Vert \M^k[u] - \M^k[u^{L,\veps}] \Vert_{L^q(\Omega_1)} & \le
    \Vert \M^k[u] - \M^k[u^L] \Vert_{L^q(\Omega_1)} + \Vert \M^k[u^L] - \M^k[u^{L,\veps}] \Vert_{L^q(\Omega_1)} \nonumber                                   \\
                                                            & \le \max\{ \Vert p_\xi \Vert_{L^\infty(\Omega_2)}, \Vert p_\xi \Vert^k_{L^\infty(\Omega_2)}\}
    \left(\,  \overline{C} \,  \Vert p_\xi \Vert_{L^\infty(\Omega_2)}\, \Vert u - u^L \Vert_{L^q(\Omega)}  + \overline{C}_{L,\veps}  \, C_{\text{thres}}\,\veps_{\text{max}} \right)
  \end{align}
  with
  $\overline{C},\ \overline C_{L,\veps}$ according to \eqref{eq:ck-bound-2} with $C$ the uniform bound for $u^L$ and $u^{L,\veps}$.\\
  From this discussion we conclude that the adaptive scheme is reliable, i.e., \eqref{eq:error-discr-pert-moments-1} and \eqref{eq:error-discr-pert-moments-2} hold, whenever we are able to verify the reliability condition \eqref{eq:reliabilitycond} and the uniform boundedness of the  perturbation error of the adaptive scheme applied to the deterministic problem \eqref{eq:det_problem} in both the spatial and stochastic variables.
  So far, such a result has been verified for the DG scheme in the mean in \cite{HovhannisyanMuellerSchaefer-2014}  for one-dimensional nonlinear conservation laws. Here,  a prediction strategy that is strongly intertwined with a particular limiter and an a priori  strategy for the threshold procedure has been used. Since there is no flux in the stochastic directions and, thus, no limiting necessary, it might be possible to extend the result     to the deterministic problem in one space dimension and arbitrary stochastic dimensions.
\end{remark}

For the numerical simulations we use multiresolution-based grid adaptation with a dyadic grid hierarchy. For this purpose, let $L\in\N$ be the maximum number of refinement levels, i.e., for each level $l=0,\dots,L$ we have $N_{l,\bx}=2^lN_{0,\bx}$ cells in the spatial direction $\bx$ and $N_{l,\bxi}=2^lN_{0,\bxi}$ cells in the stochastic direction $\bxi$, where $N_{0,\bx},N_{0,\bxi}$ are the number of cells in the initial grid in the spatial and stochastic direction, respectively.

To determine the stochastic moments of the solution $u^{L, \veps}(t,\cdot,\cdot)$ of an adaptive DG scheme for a fixed time $t>0$, we use an approach similar to that in \cite{Abgrall2014,Tokareva2022}. For each fixed $x_1\in\Omega_1$ exists a set $\mathcal I^2(x_1)$ such that
\begin{align*}
  \overline{\Omega_2} = \overline{\bigcup_{\lambda\in\mathcal I^2(x_1)}V_{\lambda}^2},\quad V^2_\lambda \cap V^2_\mu=\emptyset,\quad \lambda,\mu\in \mathcal I^2(x_1),\quad  \lambda \neq \mu
\end{align*}
describing the adaptive grid in the stochastic direction at $x_1$. The stochastic moments are then calculated in a post-processing step by integration over the stochastic domain $\Omega_2$, i.e.,
\begin{align*}
  \E[u^{L,\veps}](t,x_1)   & = \sum_{\lambda\in\mathcal I^2(x_1)}\int_{V^2_\lambda}u^{L,\veps}(t,x_1,x_2)p_\xi(x_2)\diff x_2,                                       \\
  \M^k[u^{L,\veps}](t,x_1) & = \sum_{\lambda\in\mathcal I^2(x_1)}\int_{V^2_\lambda}\left(u^{L,\veps}(t,x_1,x_2) -\E[u^{L,\veps}](t,x_1)\right)^kp_\xi(x_2)\diff x_2
\end{align*}
where each integral is calculated by a quadrature formula of the corresponding cell.

The multiresolution analysis takes into account the local structure of the stochastic problem by resolving regions with large local changes, such as discontinuities, at  higher resolution than smooth regions. Thus, we avoid the Gibb's phenomenon leading to  an accurate approximation of the stochastic moments.

\subsection{Burgers' equation with uncertain smooth initial values}
\label{subsec:burgers_cauchy_smooth}

In this section we consider the one-dimensional Burgers' equation with uncertain smooth initial data and  non-uniform random variables:
\begin{align}
  \label{eq:burgers_cauchy_sto}
  \partial_t \bar u(t,\bx;\omega_\xi) + \partial_{\bx}\left(\frac{\bar u^2(t,\bx;\omega_\xi)}{2}\right) = 0,\quad  \bx \in [0,1],\ t > 0
\end{align}
with uncertain initial condition
\begin{align}
  \label{eq:burgers_cauchy_init_sto}
  \bar u(0,\bx;\omega_\xi) = \sin(2\pi \bx)\sin(2\pi \omega_\xi),\quad \bx \in [0,1]
\end{align}
for all realizations $\omega_\xi$ of the random variable $\xi$. In addition, we assume periodic boundary conditions in the spatial direction.
We consider the following random variables:
\begin{align}
  \label{eq:non_uniform_rv}
  \xi_1 \sim \mathcal U(0,1),\quad \xi_2 \sim \mathcal N(0.5,0.15),\quad \xi_3 \sim \mathcal B(2,5), \quad \xi_4 \sim \mathcal B(2,20),
\end{align}
where $\mathcal U(a,b)$ is the uniform distribution in $(a,b)\subset\R$, $\mathcal N(\mu,\sigma^2)$ is the normal distribution with mean $\mu\in\R$ and variance $\sigma^2 > 0$ and $\mathcal B(\alpha,\beta)$ is the beta distribution for values of $\alpha,\beta > 0$.

To reformulate the stochastic Cauchy problem \eqref{eq:burgers_cauchy_sto}, \eqref{eq:burgers_cauchy_init_sto} in the deterministic formulation we choose $\bxi \in [0,1]$ for the random variables $\xi_1,\, \xi_3$ and $\xi_4$. In order to handle the non-compact support of the random variable $\xi_2$, we cut off the stochastic domain for the numerical solutions and set $\bxi\in[0,1]$.
Although, we obtain an additional error by cutting off the stochastic domain, the domain can be chosen so that this error becomes negligible compared to the discretization and perturbation error. For the random variable $\xi_2$ holds $\Prob(\xi_2 \in [0,1])\approx 0.9991$. Therefore, more than $99.9 \%$ of all drawn realizations take values in $[0,1]$. Thus, we obtain the deterministic formulation of \eqref{eq:burgers_cauchy_sto}, \eqref{eq:burgers_cauchy_init_sto}:
\begin{align}
  \label{eq:burgers_cauchy_det}
  \partial_t u(t,\bx,\bxi) + \partial_{\bx}\left(\frac{u^2(t,\bx,\bxi)}{2}\right) = 0,\quad  (\bx,\bxi) \in [0,1]\times[0,1],\ t > 0
\end{align}
with deterministic initial condition
\begin{align}
  \label{eq:burgers_cauchy_init_det}
  u(0,\bx,\bxi) = \sin(2\pi \bx)\sin(2\pi \bxi),\quad (\bx,\bxi) \in [0,1]\times[0,1].
\end{align}

In this section, we have chosen the maximum number of refinement levels $L=6$ and the number of cells of the initial grid $N_{0,\bx}=8$ and $N_{0,\bxi}=16$. We set the CFL number to 0.1. For uniform thresholding we choose the constant $C_\text{heuristic}=0.1$ for the global threshold value \eqref{eq:global_threshold_error}. On the other hand,
motivated by
Thm.~\ref{thm:thres-error-1-moments},
for weighted thresholding we set $C_i = C_\text{heuristic} / \max_{\bxi\in[0,1]} p_{\xi_i}(\bxi)$, where $p_{\xi_i}$ is the probability density of the corresponding random variable $\xi_i$ for $i=1,\dots,4$.

\paragraph{Computations with uniform thresholding.}
We first investigate the numerical solution of an adaptive multiresolution-based DG scheme, as described in Sec.~\ref{subsec:NumMeth}, where the MRA is applied with uniform thresholding.
The numerical solution of \eqref{eq:burgers_cauchy_det}, \eqref{eq:burgers_cauchy_init_det} is presented in Fig.~\ref{fig:2d_cauchy_plot}.

\begin{figure}[htbp]
  \centering
  \includegraphics[width=.32\textwidth]{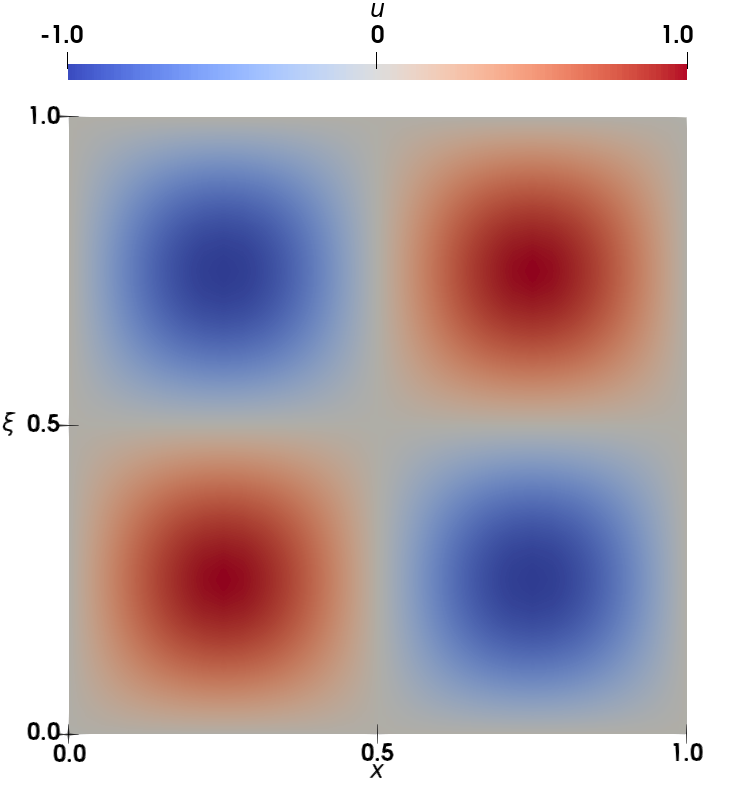}\hfill
  \includegraphics[width=.32\textwidth]{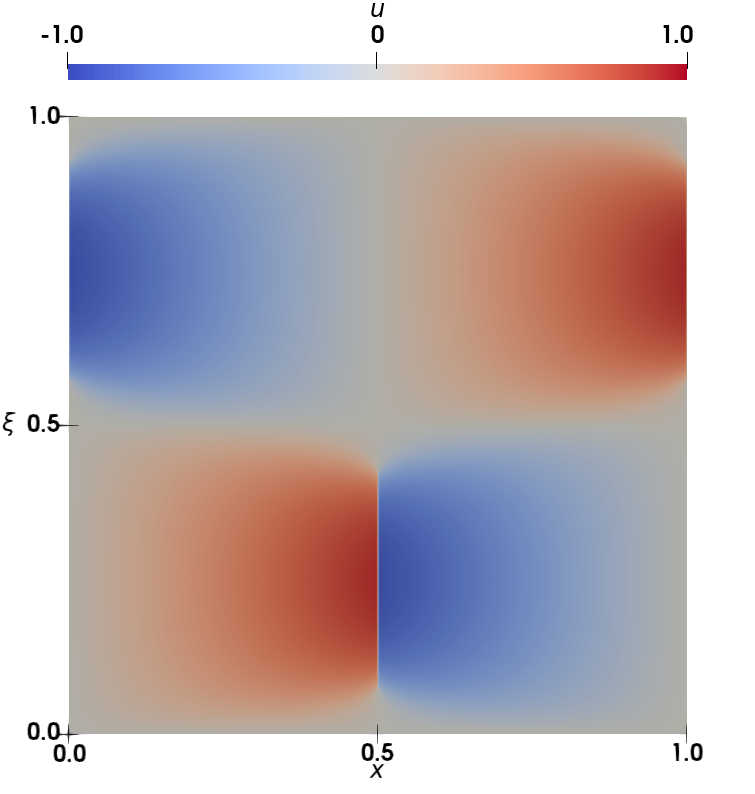}\hfill
  \includegraphics[width=.32\textwidth]{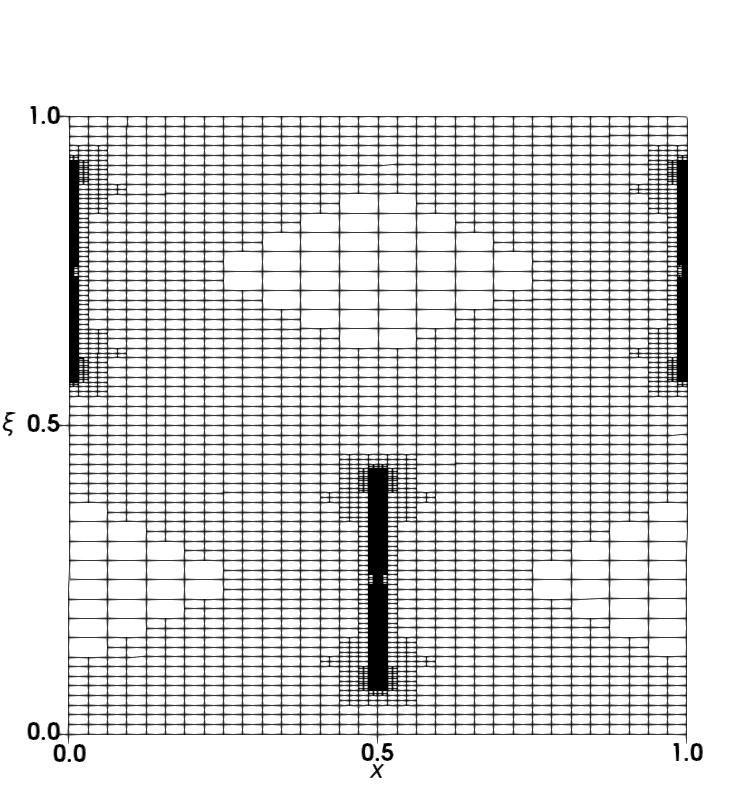}
  \caption{Solution for the Burgers' equation \eqref{eq:burgers_cauchy_det} with uncertain initial data \eqref{eq:burgers_cauchy_init_det}. Left: Initial data at time $t=0$; Middle: Numerical solution at time $t=0.35$; Right: Corresponding adaptive grid
    with uniform thresholding strategy and $L=6$ refinement levels
    for the numerical solution at time $t=0.35$.}
  \label{fig:2d_cauchy_plot}
\end{figure}

We interpret each horizontal line as a realization of the corresponding random variable. For $\bxi < 0.5$ a stationary shock is located at $\bx=0.5$, whereby for $\bxi > 0.5$ there is a rarefaction wave. Due to the periodic boundary conditions, the roles are reversed at the boundaries $\bx=0$ and $\bx=1$. Thus, for $\bxi < 0.5$ a rarefaction wave develops at the boundaries whereby for $\bxi > 0.5$ a stationary shock occurs.
The corresponding adaptive grid is also shown in Fig.~\ref{fig:2d_cauchy_plot}. Obviously, the grid is refined in regions with large local changes and less refined in regions with smooth data. Up to now, the grid adaptation is only based on the data of the solution and does not consider stochasticity. Therefore, the adaptive grid is the same for all random variables $\xi_1,\dots,\xi_4$.

We emphasize that we have to calculate the numerical solution of \eqref{eq:burgers_cauchy_det}, \eqref{eq:burgers_cauchy_init_det} only once for all random variables $\xi_1,\dots,\xi_4$ in \eqref{eq:non_uniform_rv}. The stochastic moments of these problems are then computed in a post-processing step where we have to adjust the evaluation of the solution for each random variable.

\begin{figure}[htbp]
  \centering
  \begin{subfigure}[b]{0.495\textwidth}
    \begin{adjustbox}{width=\linewidth}
      \includegraphics{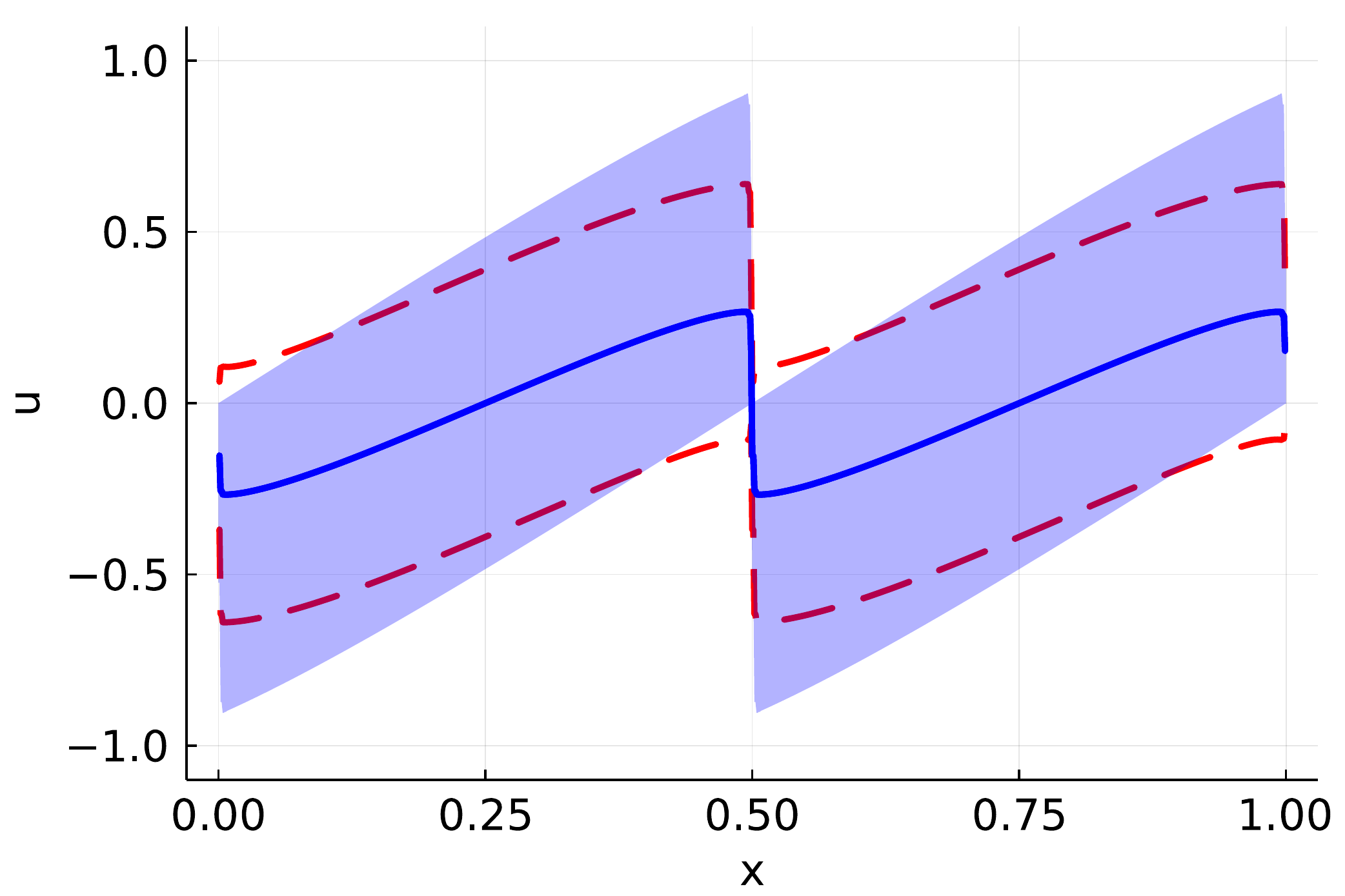}
    \end{adjustbox}
    \caption{Uniform distribution $\mathcal U(0,1)$}
  \end{subfigure}
  \hfill
  \begin{subfigure}[b]{0.495\textwidth}
    \begin{adjustbox}{width=\linewidth}
      \includegraphics{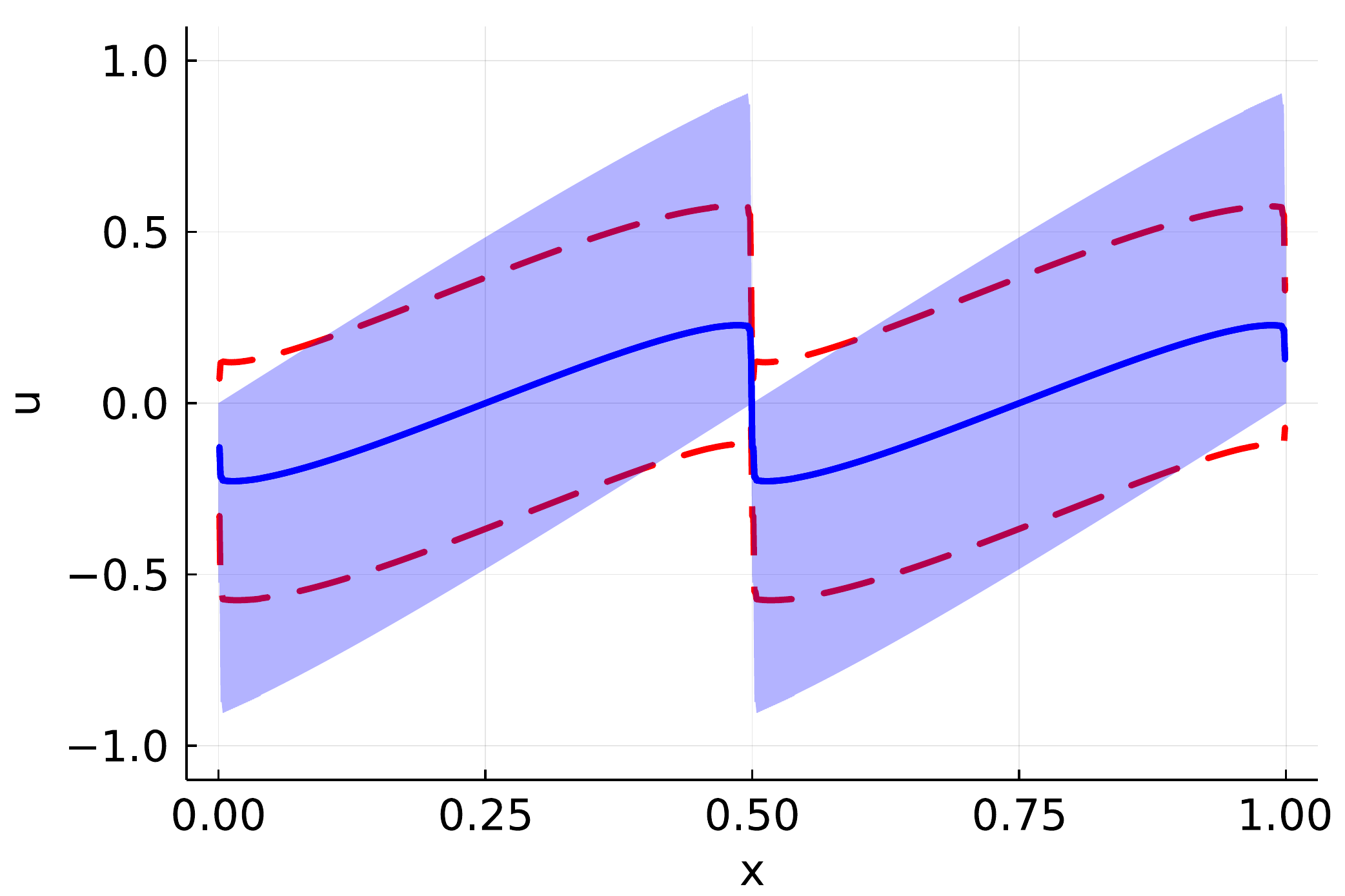}
    \end{adjustbox}
    \caption{Normal distribution $\mathcal N(0.5,0.15)$}
  \end{subfigure}
  \\
  \begin{subfigure}[b]{0.495\textwidth}
    \begin{adjustbox}{width=\linewidth}
      \includegraphics{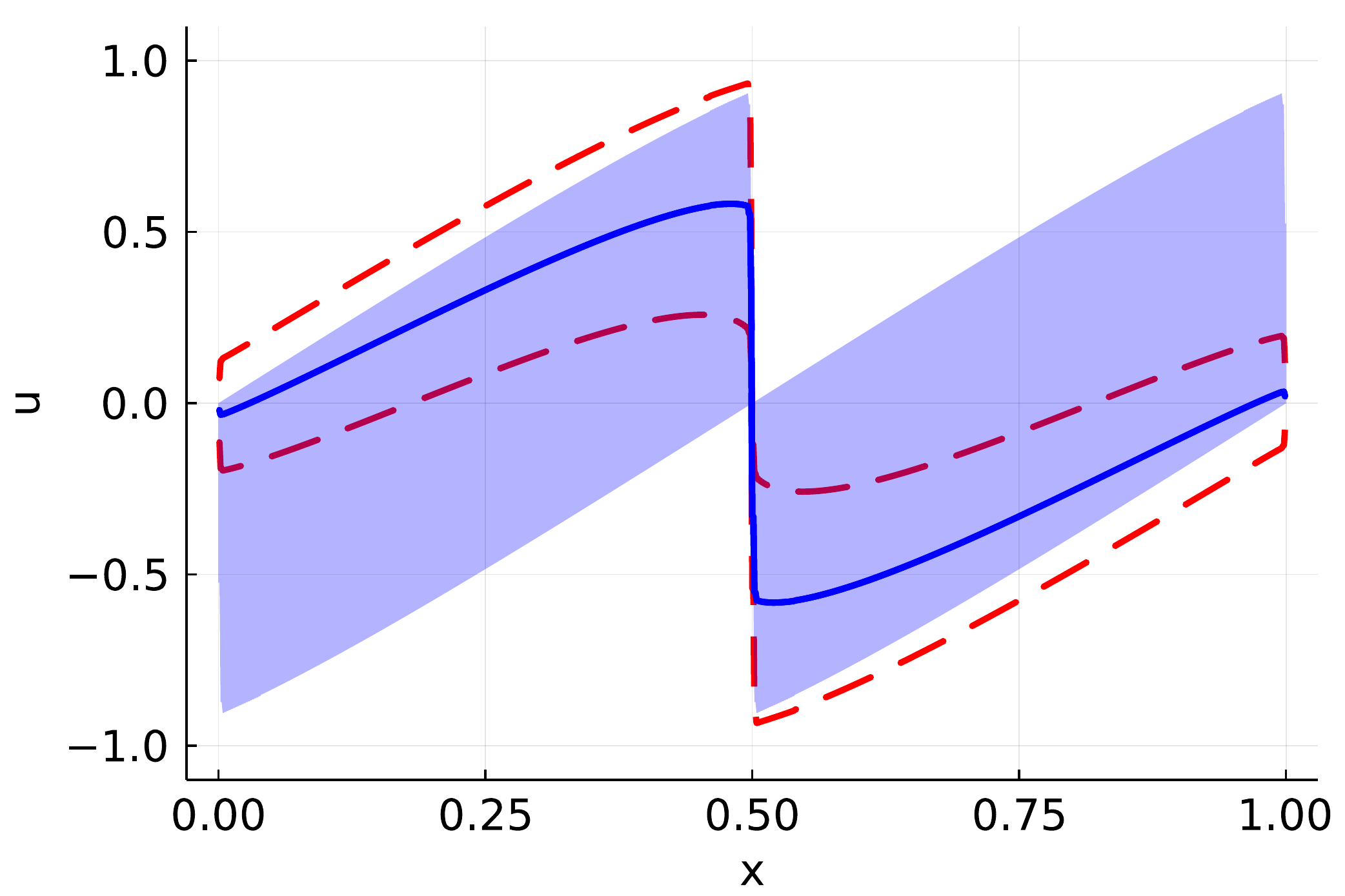}
    \end{adjustbox}
    \caption{Beta distribution $\mathcal B(2,5)$}
  \end{subfigure}
  \hfill
  \begin{subfigure}[b]{0.495\textwidth}
    \begin{adjustbox}{width=\linewidth}
      \includegraphics{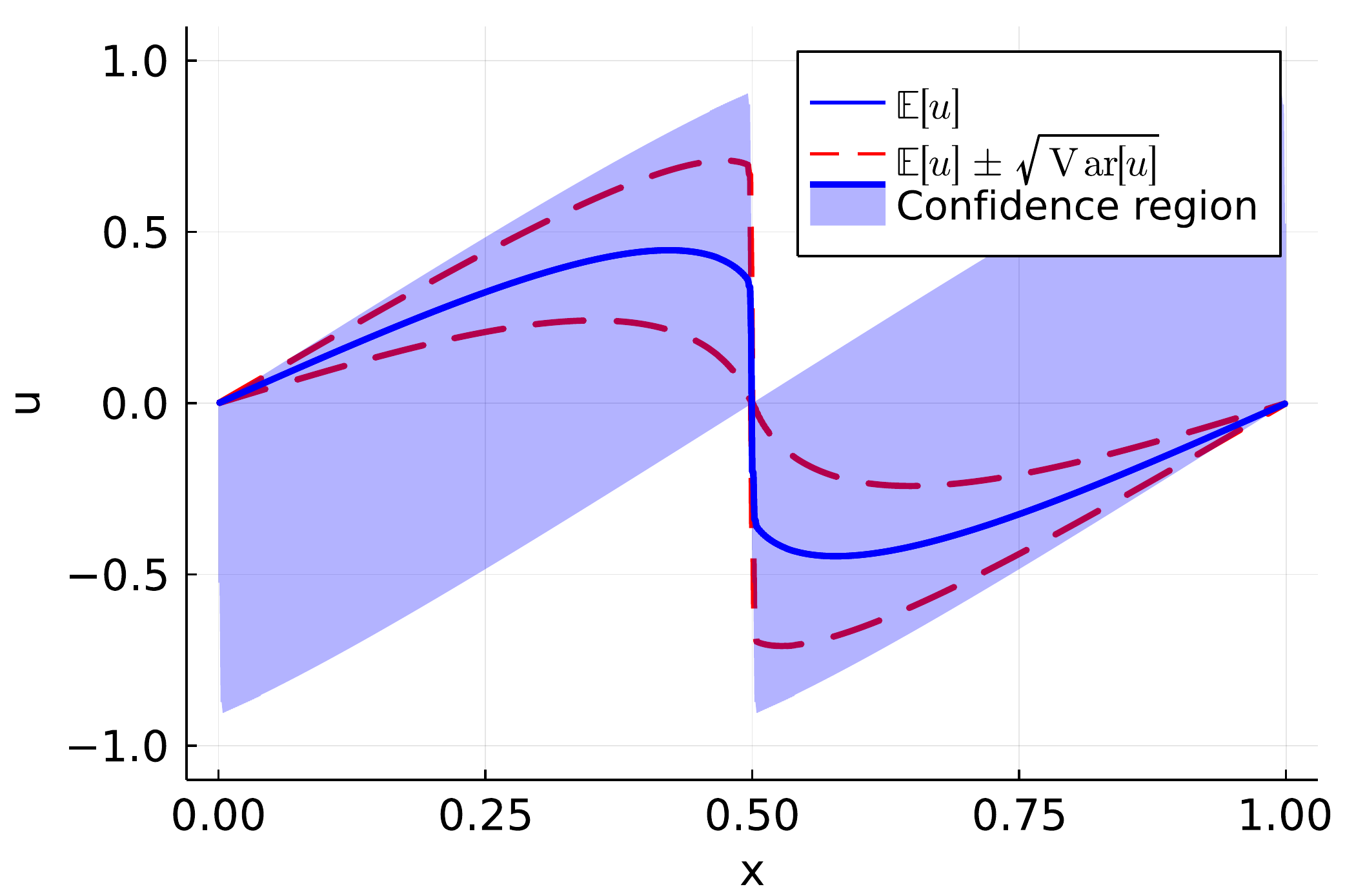}
    \end{adjustbox}
    \caption{Beta distribution $\mathcal B(2,20)$}
  \end{subfigure}
  \caption{Stochastic moments of the problem \eqref{eq:burgers_cauchy_det}, \eqref{eq:burgers_cauchy_init_det} for the different random variables \eqref{eq:non_uniform_rv} at time $t=0.35$. Computations are performed with uniform thresholding and $L=6$ refinement levels.}

  \label{fig:cauchy_moments}
\end{figure}
The stochastic moments obtained by our deterministic approach \eqref{eq:burgers_cauchy_det}, \eqref{eq:burgers_cauchy_init_det} for all random variables \eqref{eq:non_uniform_rv} are shown in Fig.~\ref{fig:cauchy_moments}.
The uniform distributed random variable $\xi_1$ and the normal distributed random variable $\xi_2$ behave similarly due to the symmetry of the corresponding densities at $\bxi = 0.5$. Since the mass of the normal distributed random variable $\xi_2$ is more concentrated around $\bxi=0.5$, the standard deviation of the normal distributed random variable $\xi_2$ is slightly smaller than the standard deviation of the uniform random variable $\xi_1$.
In contrast, the mass of the beta distributed random variables $\xi_3$ is strongly concentrated for $\bxi < 0.5$. Thus, the effects of the stationary shock at $\bx=0.5$ dominate the stochastic moments.
This behavior is amplified for the stochastic variables $\xi_4$ where the mass is highly concentrated for $\bxi < 0.25$. For example, the shock at the spatial boundaries has almost no effect on the stochastic moments for the beta distributed random variable $\xi_4$.
In Fig.~\ref{fig:cauchy_moments} we additionally show the confidence interval of our approach to illustrate the affected regions of the different random variables.

\paragraph{Computations with weighted thresholding.}
Next we investigate the numerical solution of \eqref{eq:burgers_cauchy_det}, \eqref{eq:burgers_cauchy_init_det} using the novel weighted thresholding. The results are shown in Fig.~\ref{fig:2d_cauchy_diff_rv} for the normal distributed random variable $\xi_2$ as well as for the beta distributed random variables $\xi_3$ and $\xi_4$.
The weighted
thresholding strategy results in an adaptive grid that is influenced by the underlying probability density. Thus, grid refinement is triggered more in regions with high mass of the corresponding probability density whereas regions with almost no mass of the corresponding probability density are not refined at all. This is particularly noticeable for distributions with highly concentrated mass of the relevant density functions, as in the case of the beta distributed random variable $\xi_4$. We also note that the corresponding probability density function dominates the effects of the shock for $\bxi > 0.5$ at the boundaries in the spatial directions, which is fully refined in the adaptive grid when using uniform thresholding, cf.~Fig.~\ref{fig:2d_cauchy_plot}.

\begin{figure}[htbp]
  \centering
  \includegraphics[width=.32\textwidth]{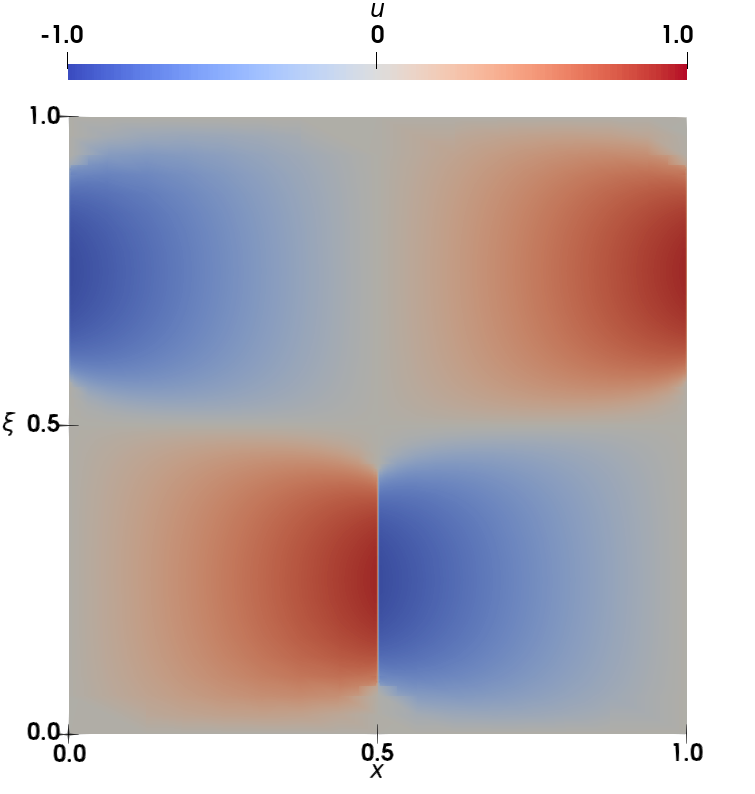}\hfill
  \includegraphics[width=.32\textwidth]{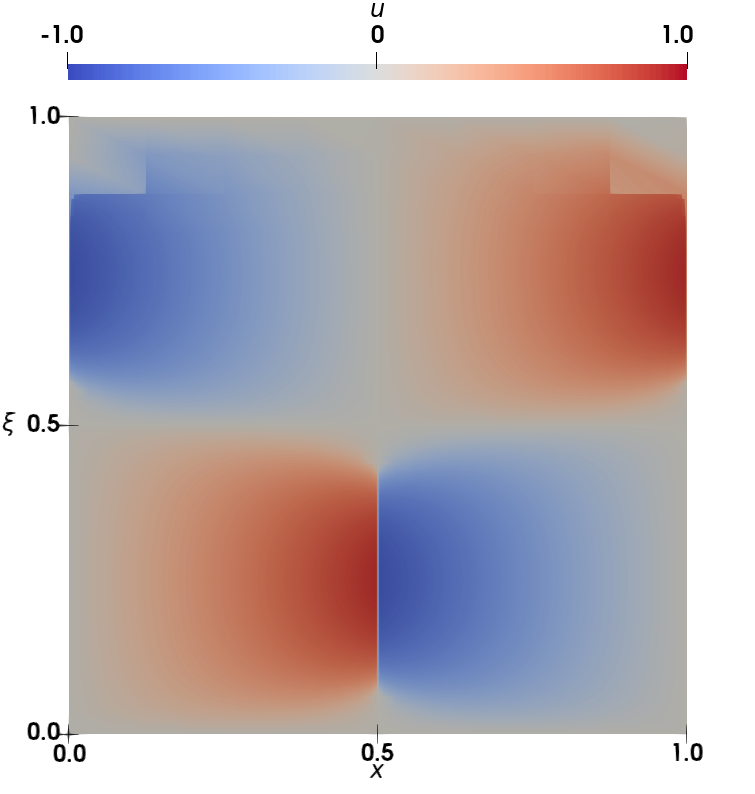}\hfill
  \includegraphics[width=.32\textwidth]{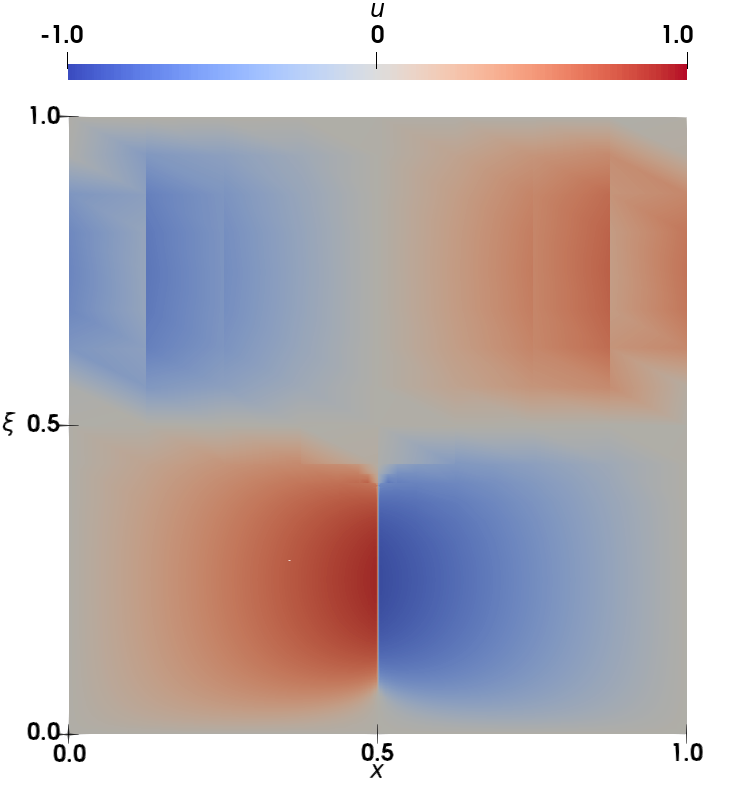} \\
  \includegraphics[width=.32\textwidth]{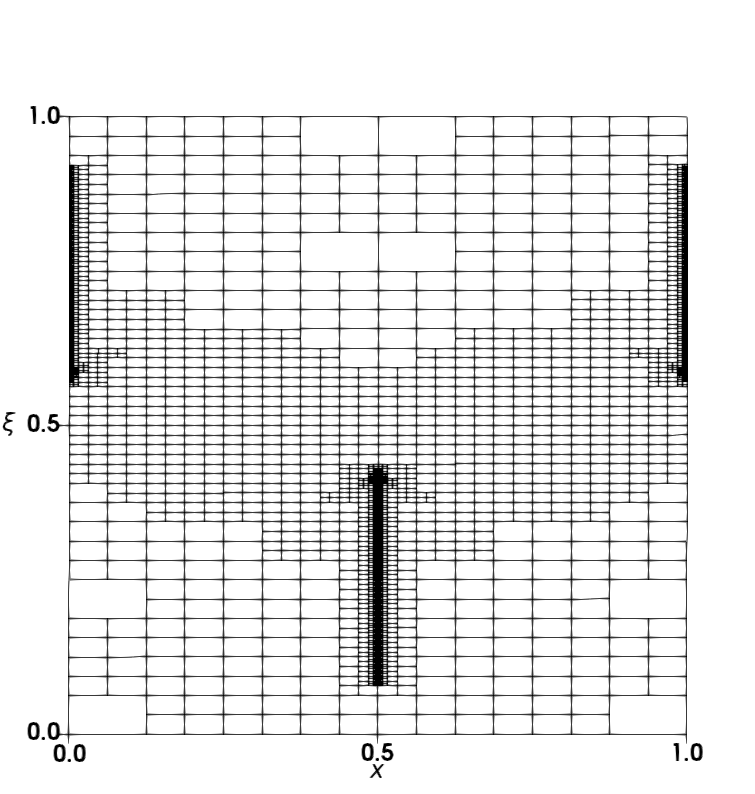}\hfill
  \includegraphics[width=.32\textwidth]{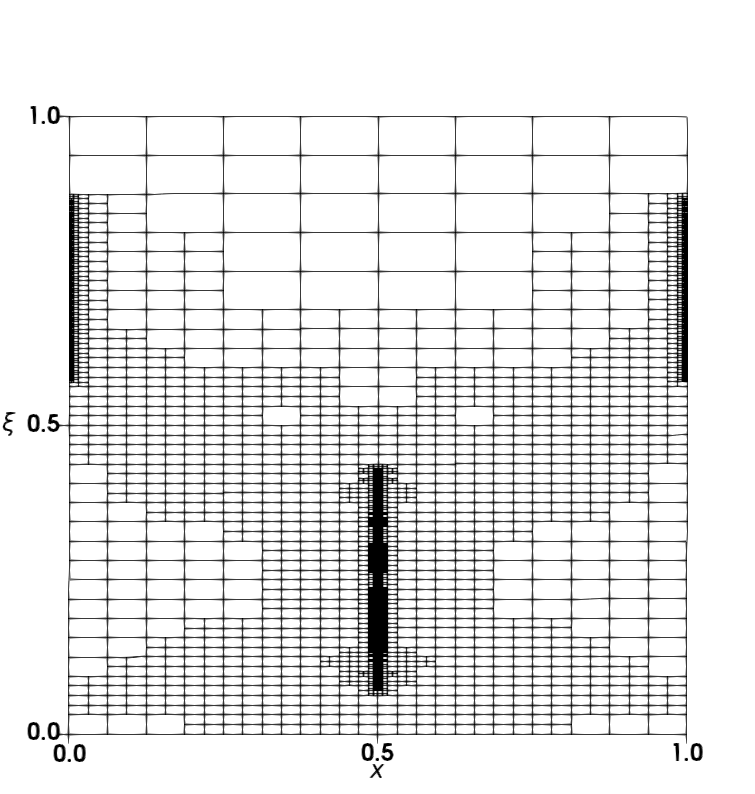}\hfill
  \includegraphics[width=.32\textwidth]{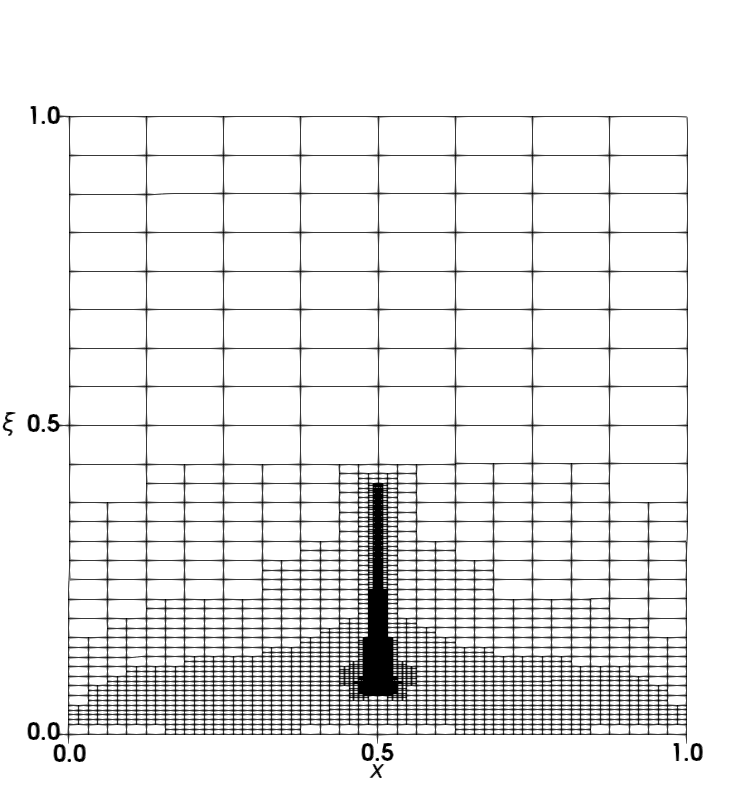}
  \caption{Solution for the Burgers' equation \eqref{eq:burgers_cauchy_det} with uncertain initial data \eqref{eq:burgers_cauchy_init_det} at time $t=0.35$ with its adaptive grid using
    weighted thresholding up to $L=6$ refinement levels. Left: Normal distribution $\mathcal N(0.5,0.15)$; Middle: Beta distribution $\mathcal B(2,5)$; Right: Beta distribution $\mathcal B(2,20)$.}
  \label{fig:2d_cauchy_diff_rv}
\end{figure}

Therefore, computations with weighted thresholding for non-uniform random variables have sparser grids than computations with uniform thresholding. We emphasize that the solution itself may look poor compared to the solution in Fig.~\ref{fig:2d_cauchy_plot},
i.e., the discretization error may be large. This can be particularly seen in regions with shocks where we usually need a locally high level of refinement to properly resolve the discontinuities.
However, the novel weighted thresholding strategy still leads to good results for the stochastic moments as seen in Fig.~\ref{fig:burgers_cauchy_comparison_weighted_vs_non_weighted}. For the uniform random variable $\xi_1$,  $p_{\xi_1}\equiv 1$ holds and thus the resulting grid with weighted grid adaptation coincides with the adaptive grid with uniform thresholding, cf. Fig. \ref{fig:2d_cauchy_plot}.

\begin{figure}[htbp]
  \centering
  \includegraphics[width=.5\textwidth]{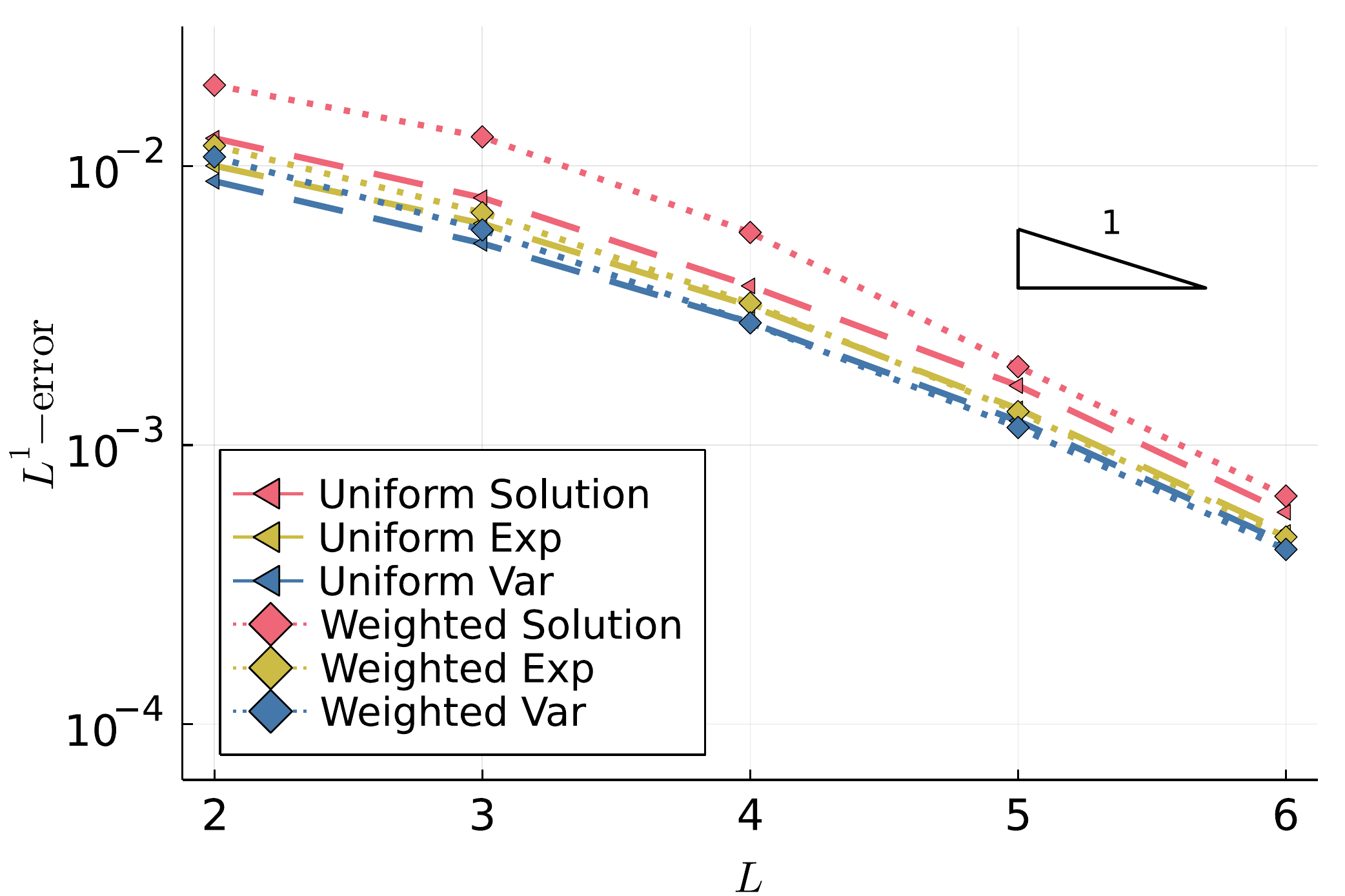}\hfill
  \includegraphics[width=.5\textwidth]{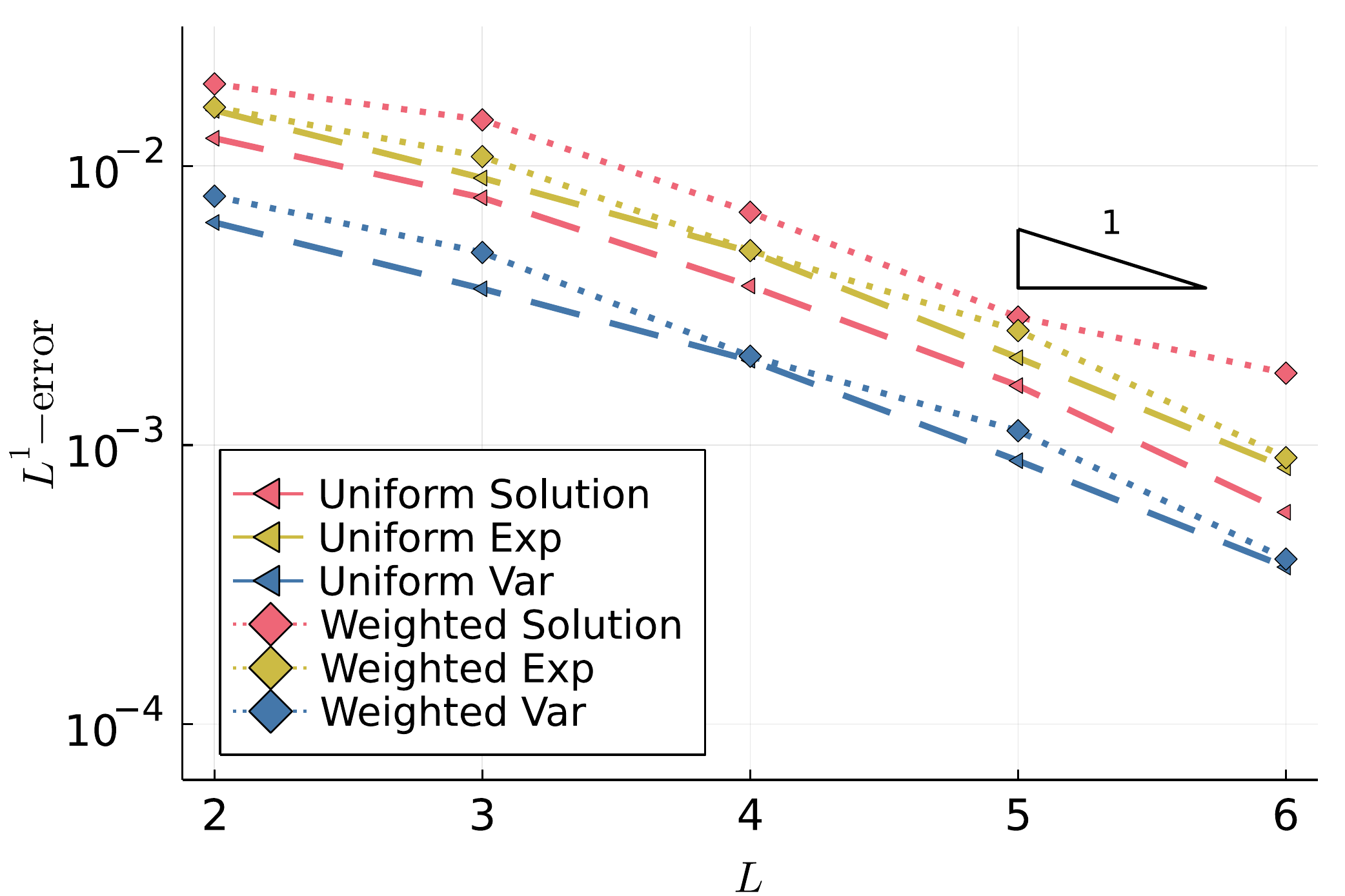}\\
  \includegraphics[width=.5\textwidth]{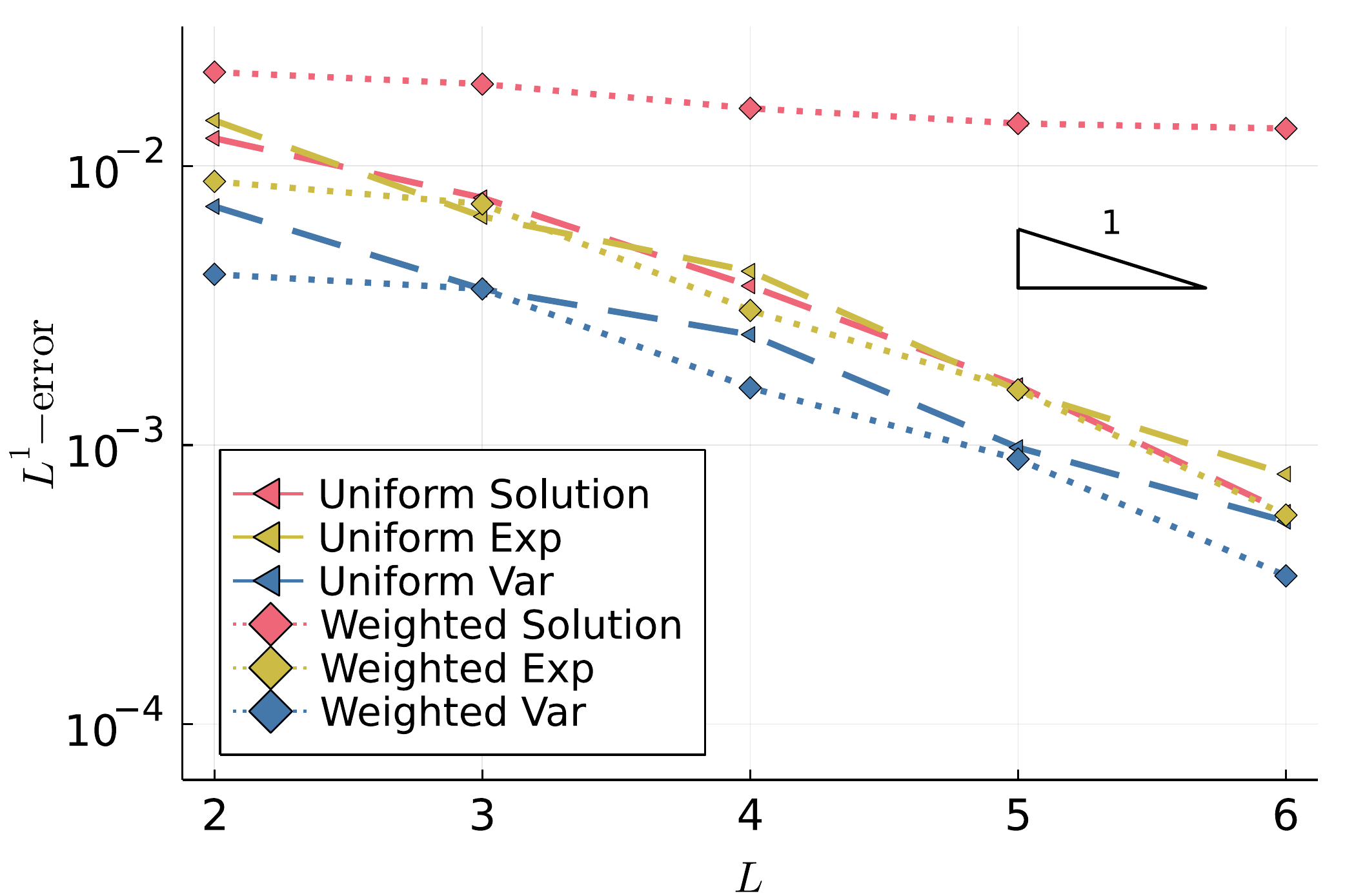}\hfill
  \includegraphics[width=.5\textwidth]{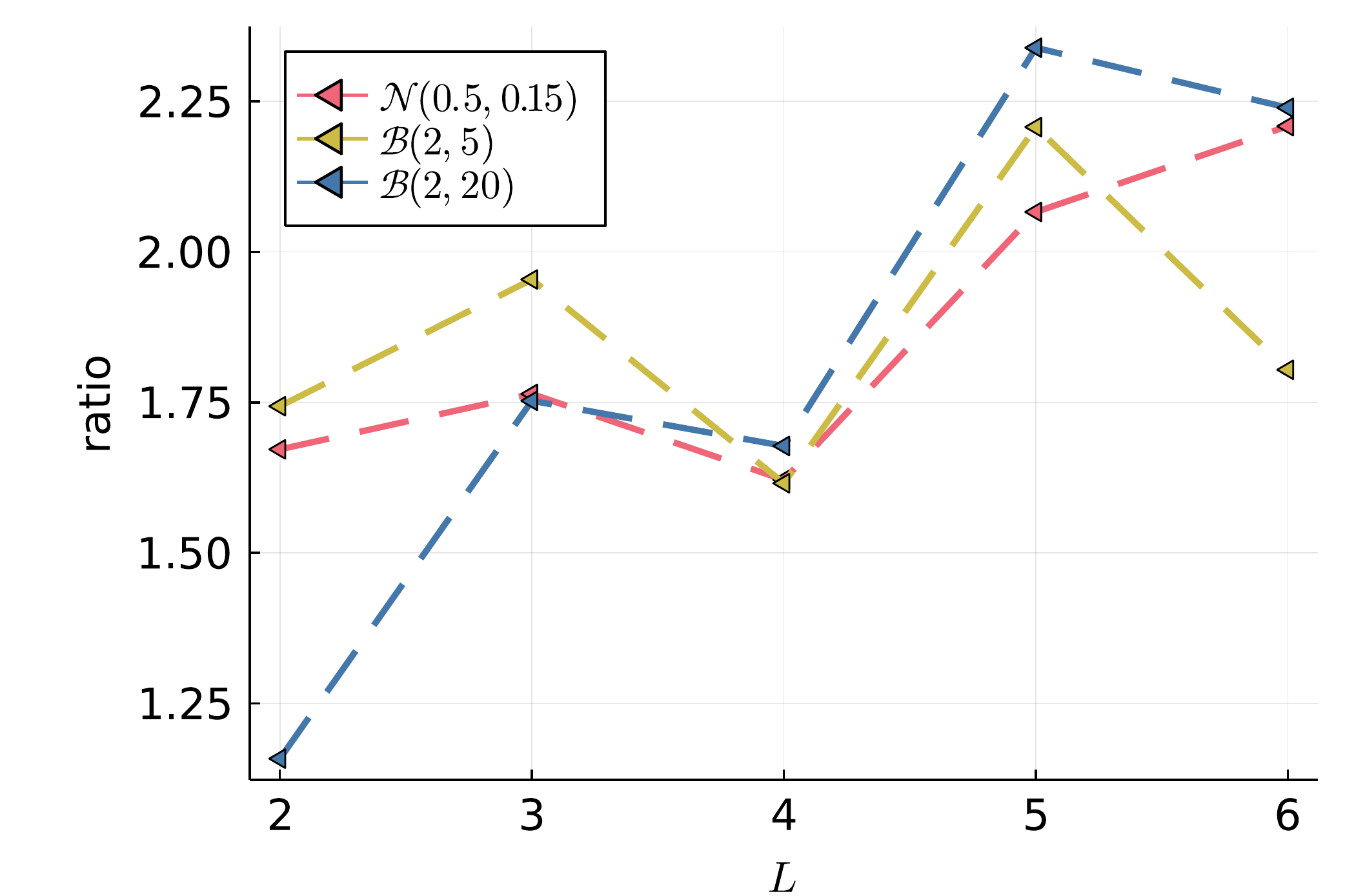}
  \caption{$L^1$-error of the stochastic moments and the solution of \eqref{eq:burgers_cauchy_det}, \eqref{eq:burgers_cauchy_init_det} at time $t=0.35$ using
    both uniform and weighted thresholding up to $L=6$ refinement levels and the ratio of the number of total cells between both methods.
    Top left: Normal distribution $\mathcal N(0.5,0.15)$; Top right: Beta distribution $\mathcal B(2,5)$; Bottom left: Beta distribution $\mathcal B(2,20)$. Bottom right: Ratio of the number of total cells.}
  \label{fig:burgers_cauchy_comparison_weighted_vs_non_weighted}
\end{figure}

\paragraph{Comparison of uniform and weighted thresholding.}

We compare the $L^1$-error of the stochastic moments for the novel and the classic strategy. As reference solution we perform a computation with uniform thresholding on a grid hierarchy with $L=12$ refinement levels.
We observe that for all random variables \eqref{eq:non_uniform_rv} the error of the stochastic moments
decreases by the empirical order of about 1, cf.~Table \ref{tab:ex01_eoc}. Using weighted thresholding, the $L^1$-error is comparable to the $L^1$-error using uniform thresholding.

\begin{table*}[htbp]
  \centering
  \setlength{\extrarowheight}{.6em}
  \setlength{\tabcolsep}{0.8em}
  \begin{tabular}{|l|c||c|c|c|c|c|c|c|c|c|}
    \hline
    \multirow{2}{*}{} &     & \multicolumn{3}{c|}{$\mathcal N(0.5,0.15)$} & \multicolumn{3}{c|}{$\mathcal B(2,5)$} & \multicolumn{3}{c|}{$\mathcal B(2,20)$}                                                       \\
    \cline{2-11}
                      & $L$ & sol                                         & exp                                    & var                                     & sol    & exp    & var    & sol    & exp    & var    \\
    \hline
    uniform           & 3   & 0.7082                                      & 0.6888                                 & 0.7401                                  & 0.7082 & 0.8063 & 0.7888 & 0.7082 & 1.1440 & 0.9833 \\
    thresholding      & 4   & 1.0490                                      & 0.9773                                 & 0.9407                                  & 1.0490 & 0.8775 & 0.8508 & 1.0490 & 0.6503 & 0.5412 \\
                      & 5   & 1.1872                                      & 1.2277                                 & 1.1710                                  & 1.1872 & 1.2612 & 1.1953 & 1.1872 & 1.4187 & 1.3449 \\
                      & 6   & 1.5073                                      & 1.4707                                 & 1.4532                                  & 1.5073 & 1.3098 & 1.2662 & 1.5073 & 0.9962 & 0.8778 \\
    \hline
    weighted          & 3   & 0.6156                                      & 0.7986                                 & 0.8693                                  & 0.4278 & 0.5877 & 0.6703 & 0.1421 & 0.2683 & 0.1737 \\
    thresholding      & 4   & 1.1374                                      & 1.0759                                 & 1.1059                                  & 1.0998 & 1.1191 & 1.2348 & 0.2887 & 1.2678 & 1.1766 \\
                      & 5   & 1.5974                                      & 1.2916                                 & 1.2455                                  & 1.2473 & 0.9518 & 0.8866 & 0.1803 & 0.9454 & 0.8491 \\
                      & 6   & 1.5400                                      & 1.4918                                 & 1.4516                                  & 0.6675 & 1.5135 & 1.5276 & 0.0597 & 1.4901 & 1.3894 \\
    \hline
  \end{tabular}
  \caption{Empirical order of convergence of the $L^1$-error of the solution and the $L^1$-error of the expectation and variance of \eqref{eq:burgers_cauchy_det}, \eqref{eq:burgers_cauchy_init_det} using adaptive grid with uniform thresholding and weighted thresholding.}
  \label{tab:ex01_eoc}
\end{table*}

In Fig. \ref{fig:2d_cauchy_plot} and Fig. \ref{fig:2d_cauchy_diff_rv} we observe that we need less cells using weighted thresholding than uniform thresholding. To quantify these savings, we consider the ratio of the total number of cells between the uniform thresholding strategy and the weighted thresholding strategy, i.e.,
\begin{align}
  \label{eq:ratio_cell}
  \text{ratio} = \frac{N_{\text{total,uniform}}}{N_{\text{total,weighted}}},
\end{align}
where $N_{\text{total,uniform}}$ and $N_{\text{total,weighted}}$ are the total number of cells over all timesteps using uniform thresholding and weighted thresholding, respectively. In Fig. \ref{fig:burgers_cauchy_comparison_weighted_vs_non_weighted} we show the ratio for random variables $\xi_2,\dots,\xi_4$.

Although we have to fully refine parts of the shock using weighted thresholding for the normal distribution $\xi_2$ and the beta distributed random variable $\xi_3$ we need less than half as many cells than using grid adaptation with uniform thresholding.
Also, if the beta distributed random variable $\xi_4$ has a highly concentrated mass, again we need about  half the cells in the grid adaptation with weighted thresholding than grid adaptation with uniform thresholding. This is because the grid has a higher refinement level for $\bxi < 0.25$ using weighted thresholding compared to the adaptive grid using uniform thresholding due to the high concentrated mass of the random variable $\xi_4$. Thus, weighted thresholding does not only save cells in regions where the influence of probability is low, but also improves regions with high mass of the corresponding probability density.

\subsection{Euler equations with non-uniform uncertain initial values}
\label{subsec:euler_cauchy}
Here we consider the one-dimensional Euler equations for a perfect gas with uncertain initial data. In particular, we investigate Sod's shock tube problem \cite{Sod1978} assuming uncertain initial pressure on the left. For a realization $\omega_\xi$ of a random variable $\xi$ we introduce the conserved variable $\bar{\mathbf u}(t,\bx;\omega_\xi) := (\bar\rho, \bar\rho\bar v,\bar\rho \bar E)^T$ describing the conservation of mass, momentum and energy. Here, $\bar\rho \equiv \bar\rho(t,\bx;\omega_\xi)$, $\bar v \equiv \bar v(t,\bx;\omega_\xi)$ and $\bar E \equiv \bar E(t,\bx;\omega_\xi)$ denote the density, momentum and total energy, respectively. The total energy is the sum of kinetic and internal energy $\bar e \equiv \bar  e(t,\bx;\omega_\xi)$, i.e.,
\begin{align*}
  \bar E = \frac{1}{2}\bar v^2 + \bar e.
\end{align*}
Assuming a perfect gas the internal energy is determined by
\begin{align*}
  \bar e = \frac{\bar p}{(\gamma -1)\bar \rho}
\end{align*}
with $\gamma = 1.4$ \cite{Toro2009}.  We investigate the behavior of the system with uncertain initial pressure on the left
\begin{align}
  \label{eq:sto_initial_pressure}
  \bar p(0,\bx;\omega_\xi) = \begin{cases}
                               \omega_\xi + 0.2, & \quad \bx < 0.5 \\
                               0.1,              & \quad \bx > 0.5
                             \end{cases}.
\end{align}
Finally, the uncertain Riemann problem is determined by
\begin{subequations}
  \label{eq:sto_1d_euler}
  \begin{alignat}{2}
    \label{eq:sto_1d_euler_mass_cons}
     & \partial_t\bar\rho + \partial_\bx(\bar \rho \bar v)                          &  & = 0,                         \\
    \label{eq:sto_1d_euler_mom_cons}
     & \partial_t(\bar \rho \bar v) + \partial_\bx(\bar \rho \bar v^2 + \bar p)     &  & = 0, \quad \bx\in\R,\ t > 0, \\
    \label{eq:sto_1d_euler_energy_cons}
     & \partial_t(\bar \rho \bar E) + \partial_\bx(\bar v(\bar \rho \bar E+\bar p)) &  & = 0.
  \end{alignat}
\end{subequations}
and uncertain initial data
\begin{align}
  \label{eq:sto_ic_euler}
  \bar {\mathbf u}(0,\bx;\omega_\xi) = \begin{cases}
                                         (1.0,\ 0.0,\ 0.5 + 2.5 \omega_\xi)^T, & \quad \bx < 0.5 \\
                                         (0.125,\ 0.0,\ 0.25)^T,               & \quad \bx > 0.5
                                       \end{cases}.
\end{align}
For our investigations we consider the random variable $\xi \sim \mathcal B(2,5)$, where $\mathcal B(\alpha,\beta)$ is the beta distribution for values of $\alpha,\beta > 0$.

Again, we replace the stochastic parameter $\omega_\xi$ at the expense of an additional space dimension. Therefore, the conserved variable becomes $\mathbf u(t,\bx,\bxi) := (\rho, \rho v, \rho E)^T$ for  $(\bx,\bxi) \in \R\times[0,1]$, where $\rho\equiv\rho(t,\bx,\bxi)$, $v \equiv v(t,\bx,\bxi)$, $E\equiv E(t,\bx,\bxi)$ and $p \equiv p(t,\bx,\bxi)$.
The initial condition of the pressure is given by
\begin{align}
  \label{eq:det_initial_pressure}
  p(0,\bx,\bxi) = \begin{cases}
                    \bxi + 0.2, & \quad \bx < 0.5 \\
                    0.1,        & \quad \bx > 0.5
                  \end{cases},
  \quad\bxi\in[0,1].
\end{align}
Thus, the deterministic approach of the system \eqref{eq:sto_1d_euler} reads
\begin{subequations}
  \label{eq:det_1d_euler}
  \begin{alignat}{2}
    \label{eq:det_1d_euler_mass_cons}
     & \partial_t\rho + \partial_\bx(\rho v)           &  & = 0,                                           \\
    \label{eq:det_1d_euler_mom_cons}
     & \partial_t(\rho v) + \partial_\bx(\rho v^2 + p) &  & = 0, \quad (\bx,\bxi)\in\R\times[0,1],\ t > 0, \\
    \label{eq:det_1d_euler_energy_cons}
     & \partial_t(\rho E) + \partial_\bx(v(\rho E+p))  &  & = 0,
  \end{alignat}
\end{subequations}
with initial condition
\begin{align}
  \label{eq:det_ic_euler}
  \mathbf u(0,\bx,\bxi) = \begin{cases}
                            (1.0,\ 0.0,\ 0.5 + 2.5\bxi)^T, & \quad \bx < 0.5 \\
                            (0.125,\ 0.0,\ 0.25)^T ,       & \quad \bx > 0.5
                          \end{cases}
  ,\quad \bxi\in[0,1].
\end{align}

For this example we choose the maximum number of refinement levels $L=6$ and the number of the cells in the initial grid $N_{0,\bx} = N_{0,\bxi} = 8$ and set the CFL number to 0.1. For uniform thresholding we choose the constant $C_\text{heuristic} = 0.1$ as global threshold value \eqref{eq:global_threshold_error}. As in Sect. \ref{subsec:burgers_cauchy_smooth} for weighted thresholding we set $\hat C = C_\text{heuristic} / \max_{\bxi\in[0,1]}p_{\xi}(\bxi)$ as global threshold value, where $p_{\xi}$ is the corresponding probability density of the random variable $\xi\sim\mathcal B(2,5)$.

\paragraph{Computations with uniform thresholding.} The solution of \eqref{eq:det_1d_euler}, \eqref{eq:det_ic_euler} for the final time $t=0.2$ using MRA with uniform thresholding is presented in Figure \ref{fig:2d_euler_plot_uniform}. Each horizontal cut represents the solution of a single realization of the problem \eqref{eq:det_1d_euler}, \eqref{eq:det_ic_euler}. We observe that for higher initial pressure the shock wave, the contact wave and the rarefaction wave propagate faster. This leads to discontinuities in the stochastic direction for the leading shock wave. For the contact wave we only observe jumps in the conserved variables $(\rho, \rho v, \rho E)^T$ in the stochastic direction but no discontinuities for velocity $v$ and pressure $p$. Thus, the solution only exhibits discontinuities in the stochastic direction when there are discontinuities in the spatial direction, too.
Furthermore, the grid is only fully refined along the discontinuities caused by the shock wave and by the contact discontinuity. In smooth regions, the adaptive grid has a low refinement level, so the grid is refined only in regions with high local changes.

The stochastic moments of the deterministic approach \eqref{eq:det_1d_euler}, \eqref{eq:det_ic_euler} for the beta distributed random variable $\xi$ for the density $\rho$, momentum $\rho v$, density of energy $\rho E$, pressure $p$ and velocity $v$, respectively, are shown in Figure \ref{fig:euler_beta_2_5_moments}. Obviously, the moments are smooth where the discontinuities are smoothened due to the averaging process.
\begin{figure}[htbp]
  \centering
  \includegraphics[width=.33\textwidth]{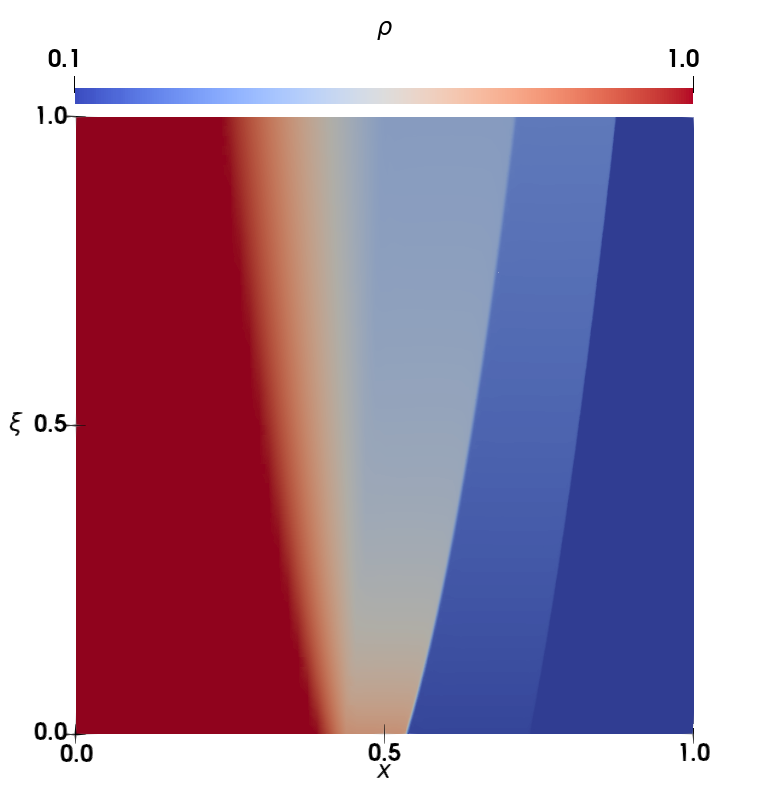}\hfill
  \includegraphics[width=.33\textwidth]{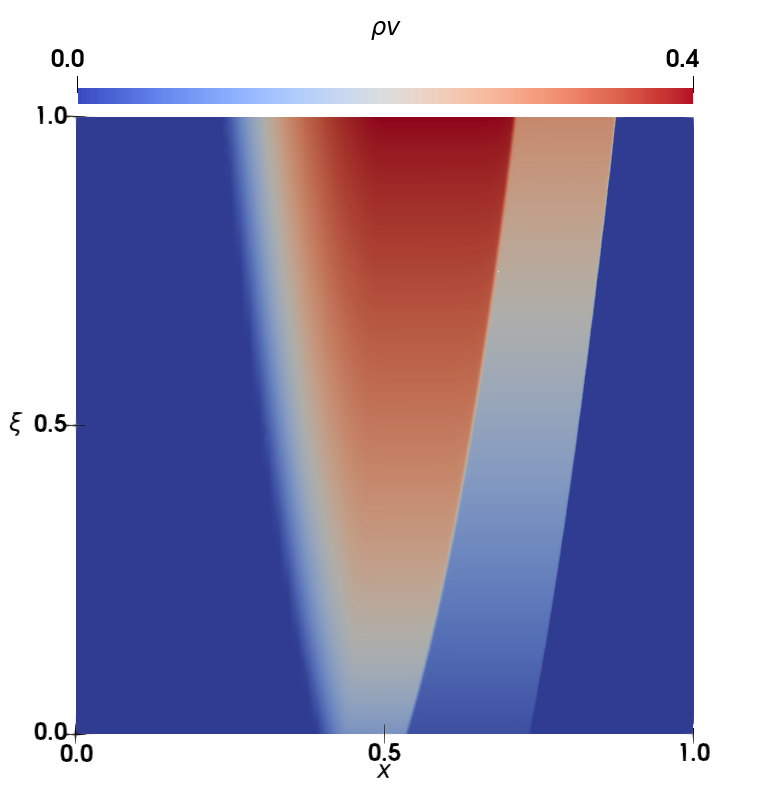}\hfill
  \includegraphics[width=.33\textwidth]{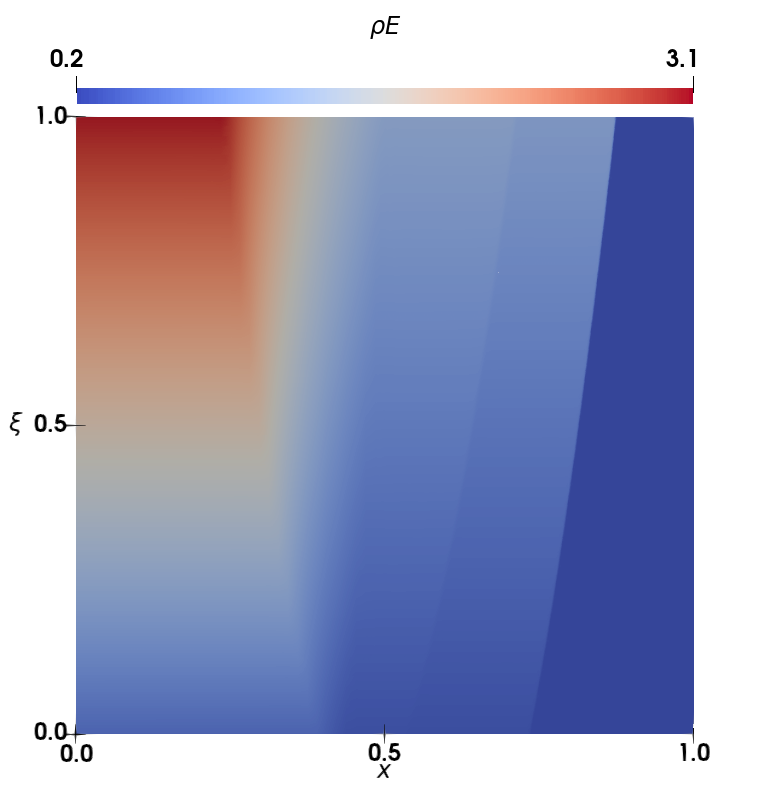}\\
  \includegraphics[width=.33\textwidth]{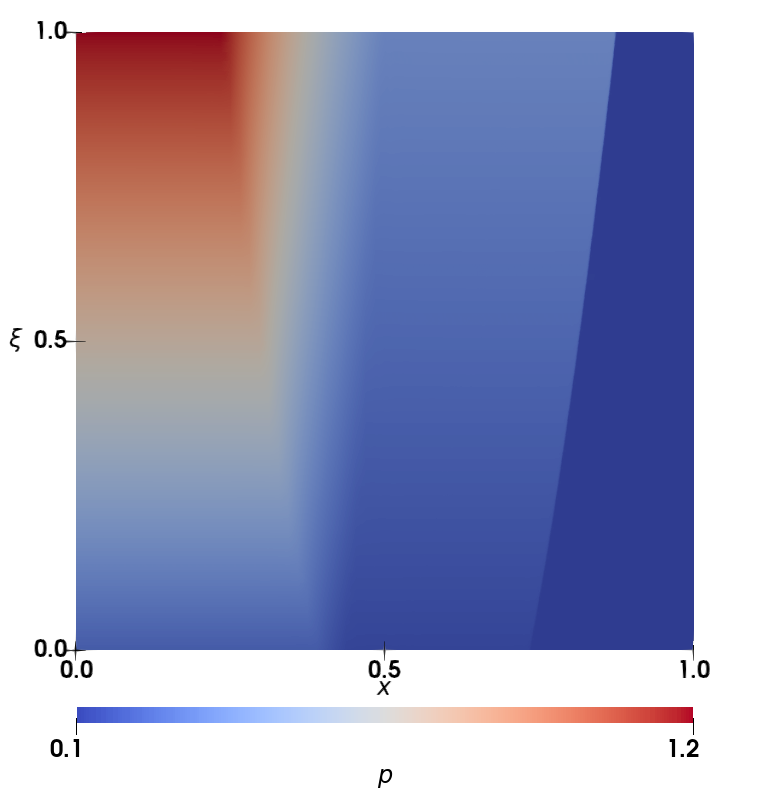}\hfill
  \includegraphics[width=.33\textwidth]{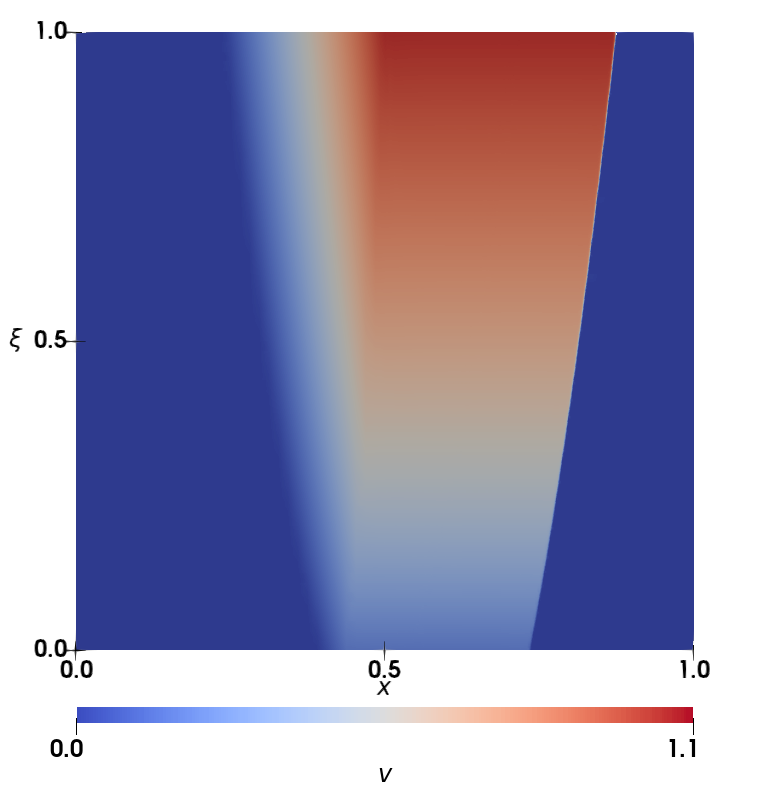}\hfill
  \includegraphics[width=.33\textwidth]{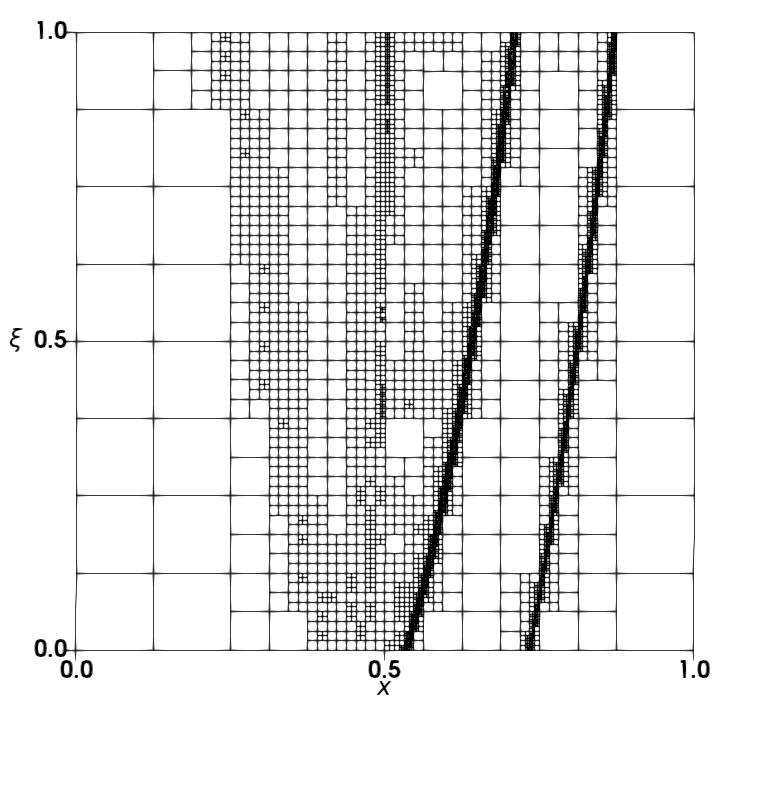}
  \caption{Solution for the Euler equations \eqref{eq:det_1d_euler} with uncertain initial data \eqref{eq:det_ic_euler} at time $t=0.2$ with $L=6$ refinement levels using uniform thresholding. Top row (from left to right): density $\rho$; momentum $\rho v$; density of energy $\rho E$. Bottom row (from left to right): pressure $p$; velocity $v$; corresponding adaptive grid.}
  \label{fig:2d_euler_plot_uniform}
\end{figure}

\begin{figure}[htbp]
  \centering
  \includegraphics[width=.33\textwidth]{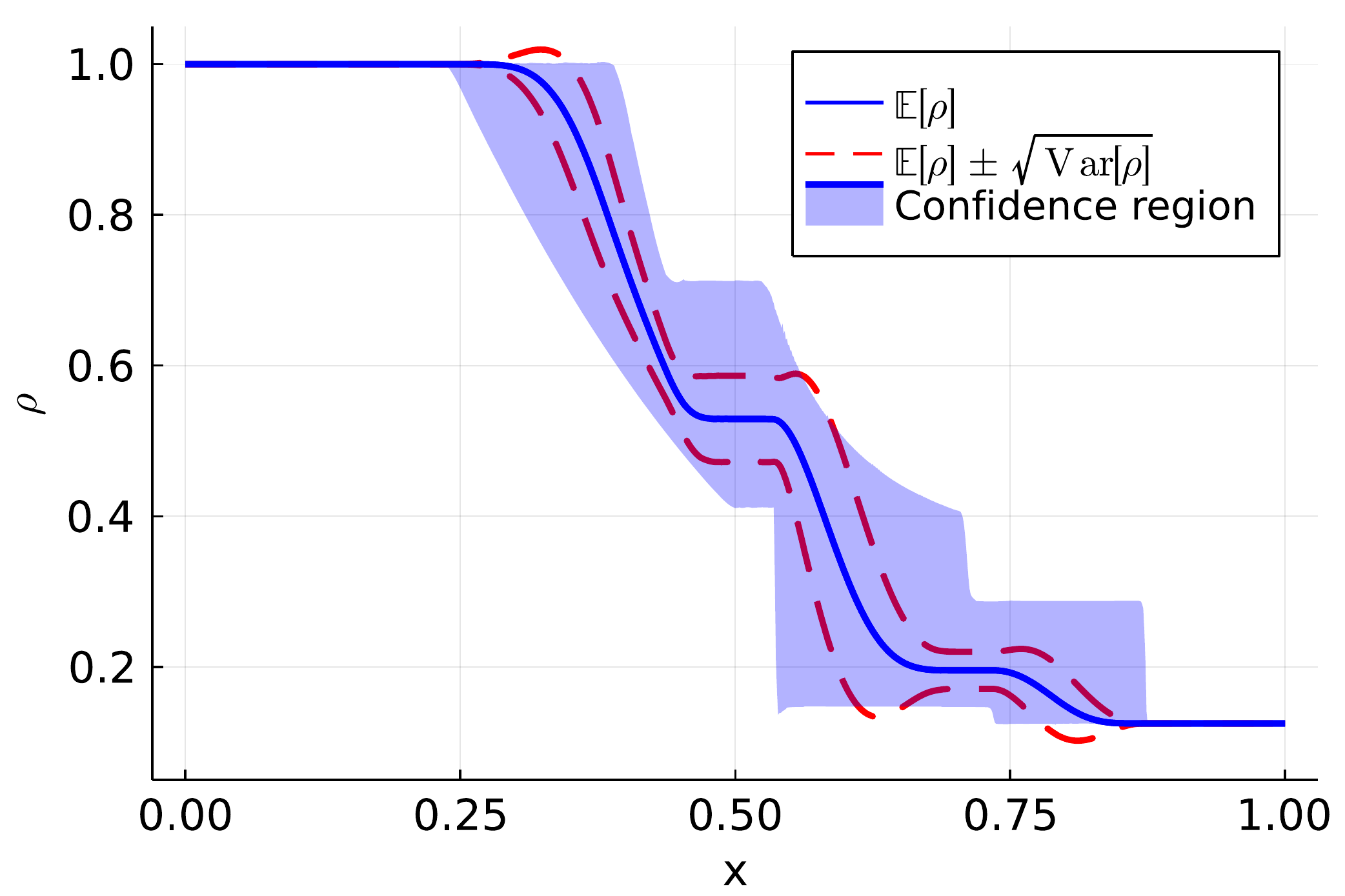}\hfill
  \includegraphics[width=.33\textwidth]{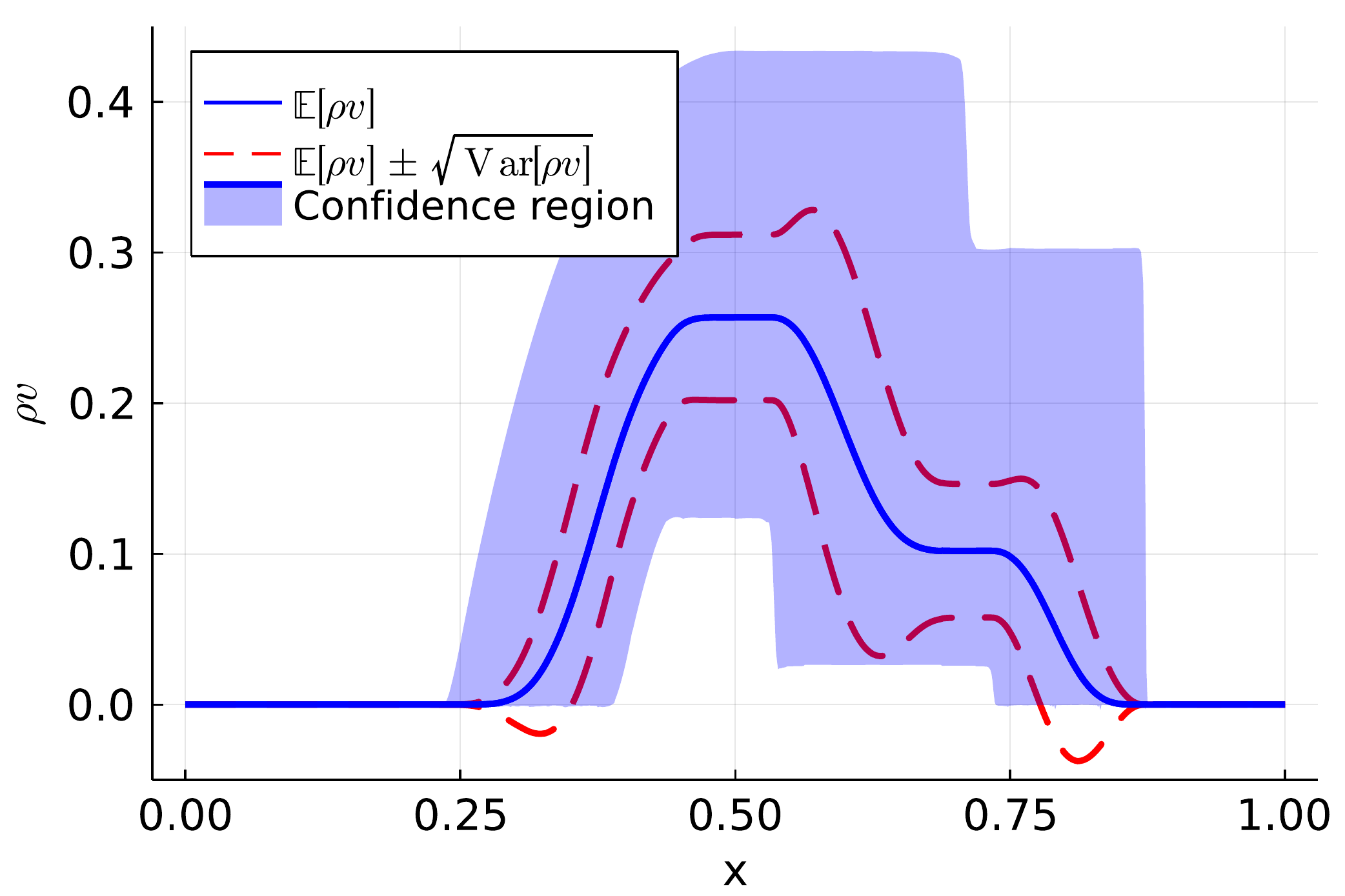}\hfill
  \includegraphics[width=.33\textwidth]{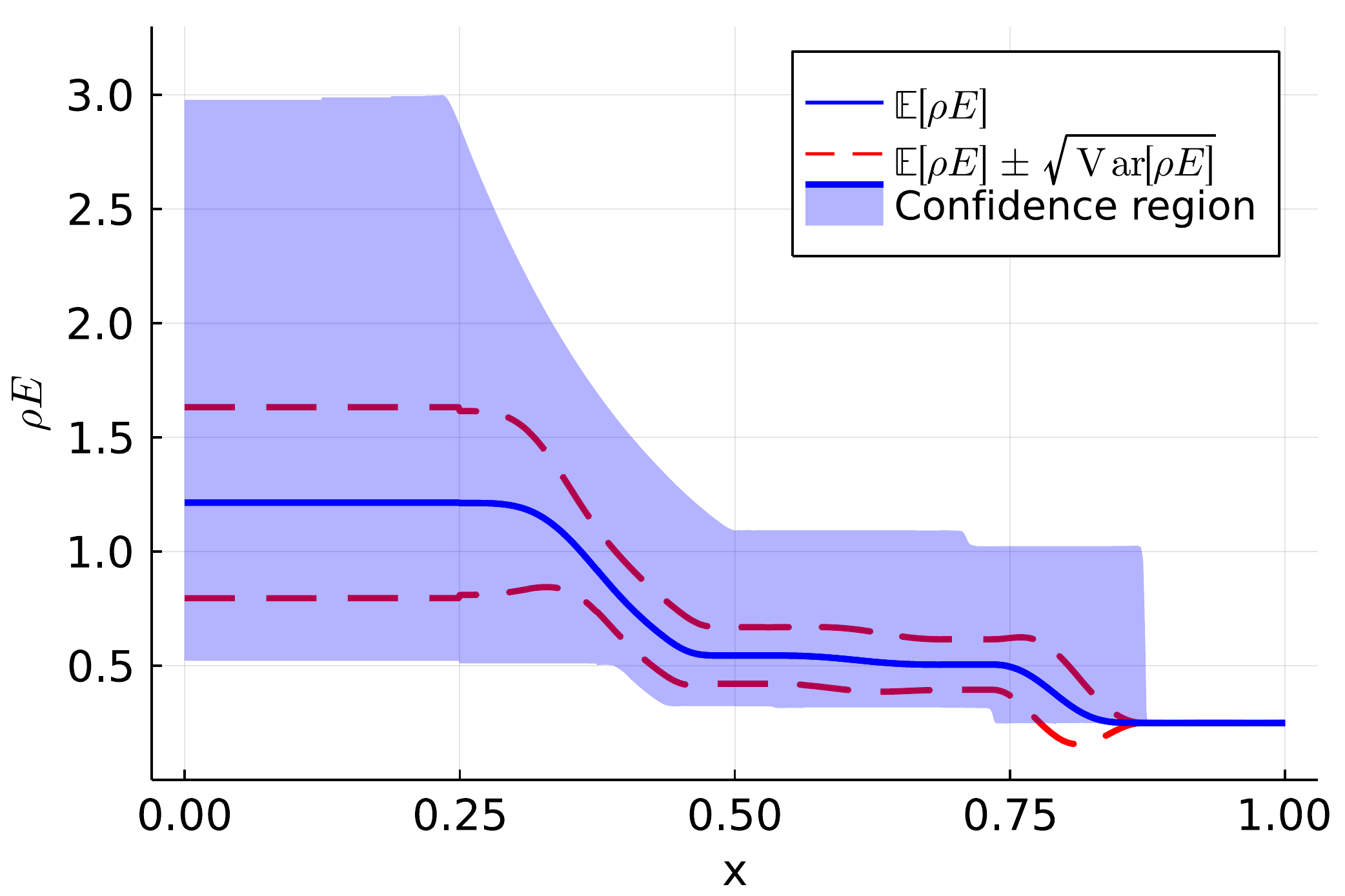}\\
  \includegraphics[width=.33\textwidth]{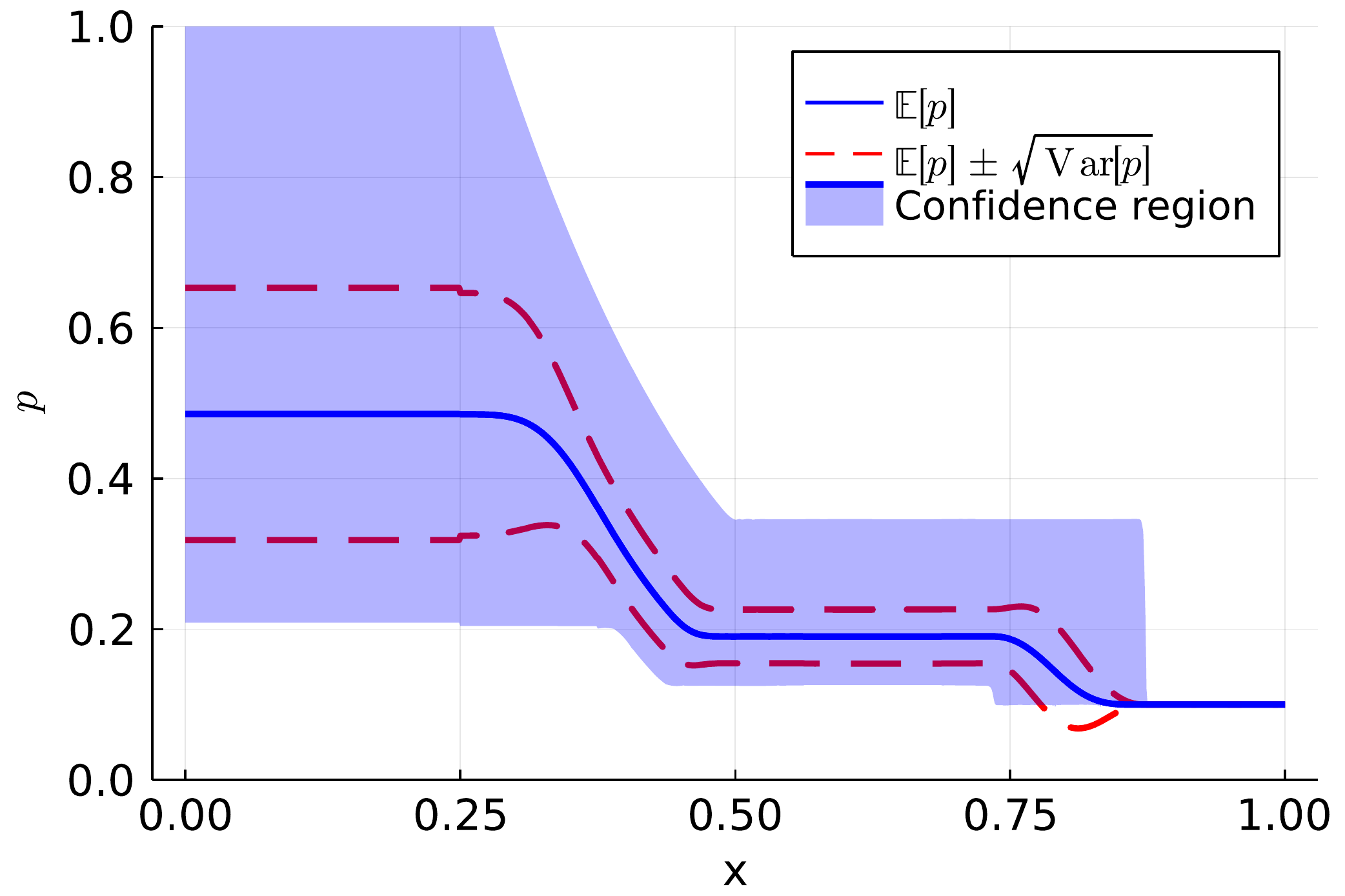}
  \includegraphics[width=.33\textwidth]{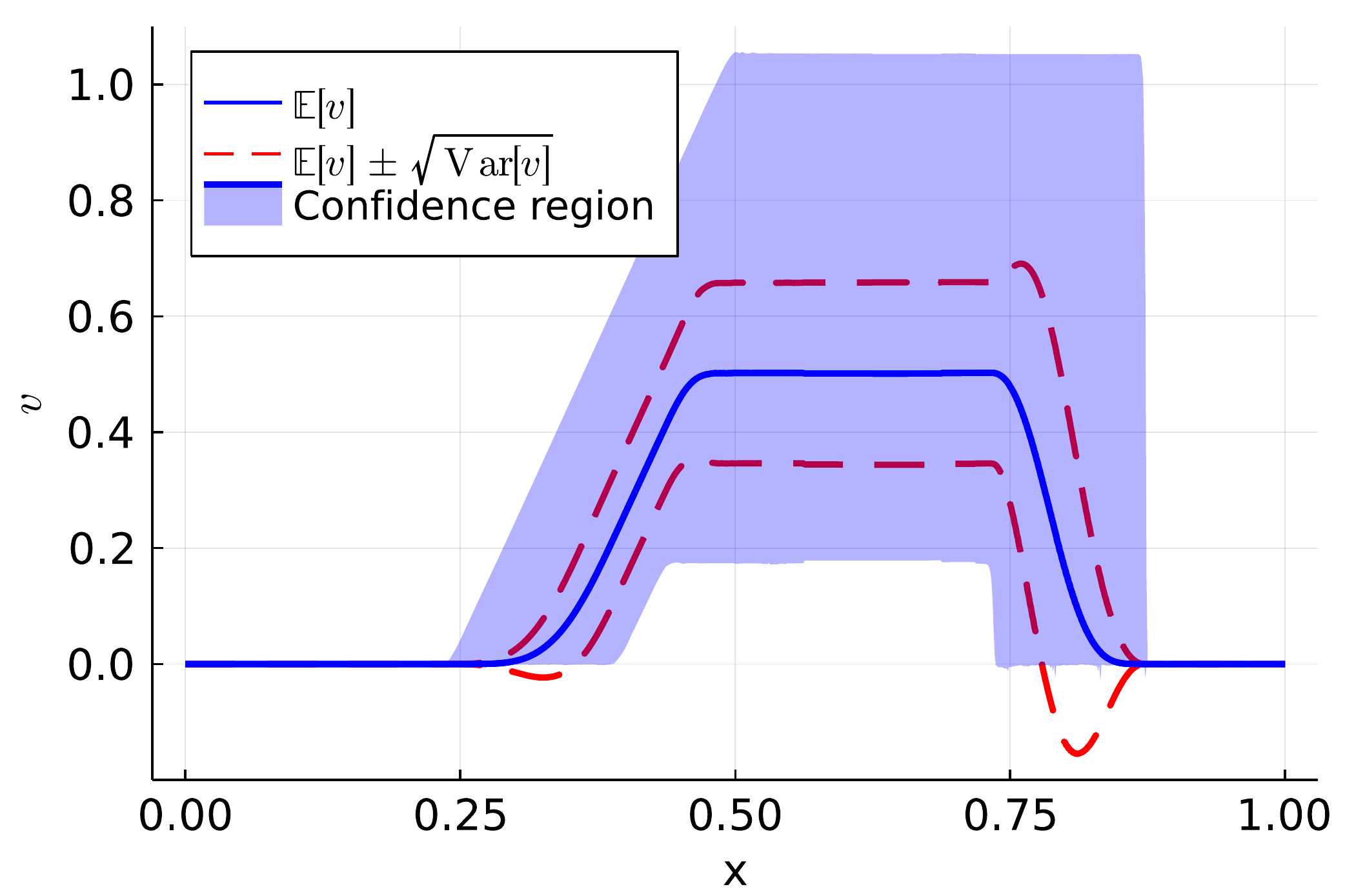}\hfill
  \caption{Stochastic moments for the solution of the Euler equations \eqref{eq:det_1d_euler}, \eqref{eq:det_ic_euler} for the random variable $\xi\sim\mathcal B(2,5)$ at time $t=0.2$. Computations are performed with uniform thresholding and $L=6$ refinement levels. Top row (from left to right): density $\rho$; momentum $\rho v$; density of energy $\rho E$. Bottom row (from left to right): pressure $p$; velocity $v$.}
  \label{fig:euler_beta_2_5_moments}
\end{figure}

\paragraph{Computations with weighted thresholding.}
The numerical simulation of \eqref{eq:det_1d_euler}, \eqref{eq:det_ic_euler} using the novel thresholding strategy is shown in Figure \ref{fig:2d_euler_plot_weighted} for the beta distributed random variable $\xi$. As illustrated in Sect. \ref{sec:MRA-moments}, grid adaptation now depends strongly on the underlying probability density of the corresponding random variable. Since the random variable $\xi$ has a higher mass concentration for $\bxi < 0.5$, the grid is more refined than for $\bxi > 0.5$. Moreover, grid refinement is triggered more along the shock than along the region of the contact discontinuity. Although we do not have high stochastic influence for $\bxi > 0.5$, the shock is fully refined up to $\bxi = 0.75$ whereby grid adaptation is not triggered for the rarefaction wave. Our novel thresholding strategy thus takes into account both the stochasticity and the local behavior of the solution itself. This leads to an adaptive grid with significantly fewer cells than with uniform thresholding, see Figure \ref{fig:2d_euler_plot_uniform}. Due to the novel adaptation strategy with respect to the stochastic moments, the solution itself may look poor. This is especially the case, for example, for the velocity $v$ and the momentum $\rho v$, respectively, along the shock.

\begin{figure}[htbp]
  \centering
  \includegraphics[width=.33\textwidth]{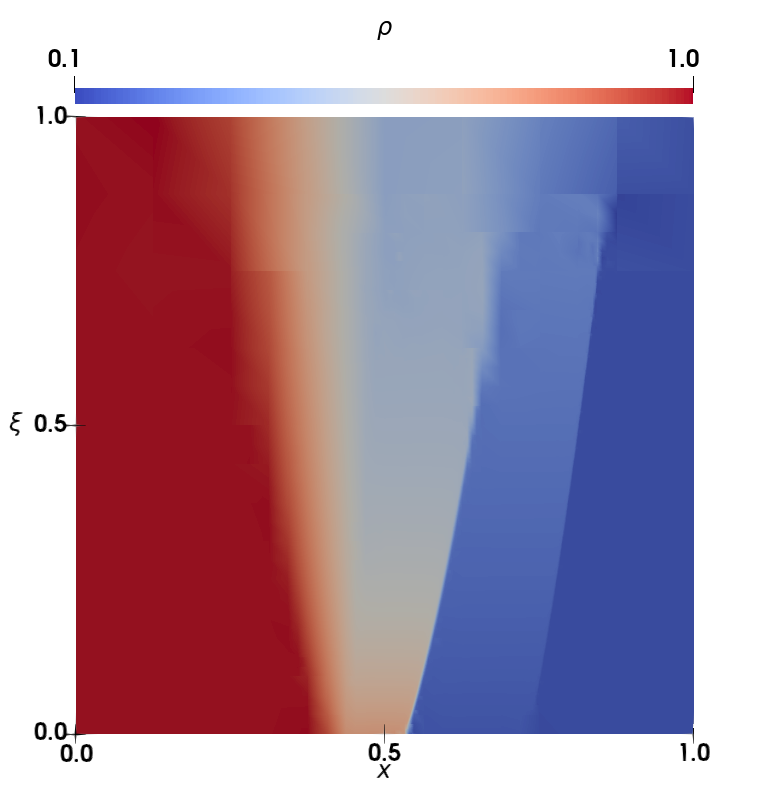}\hfill
  \includegraphics[width=.33\textwidth]{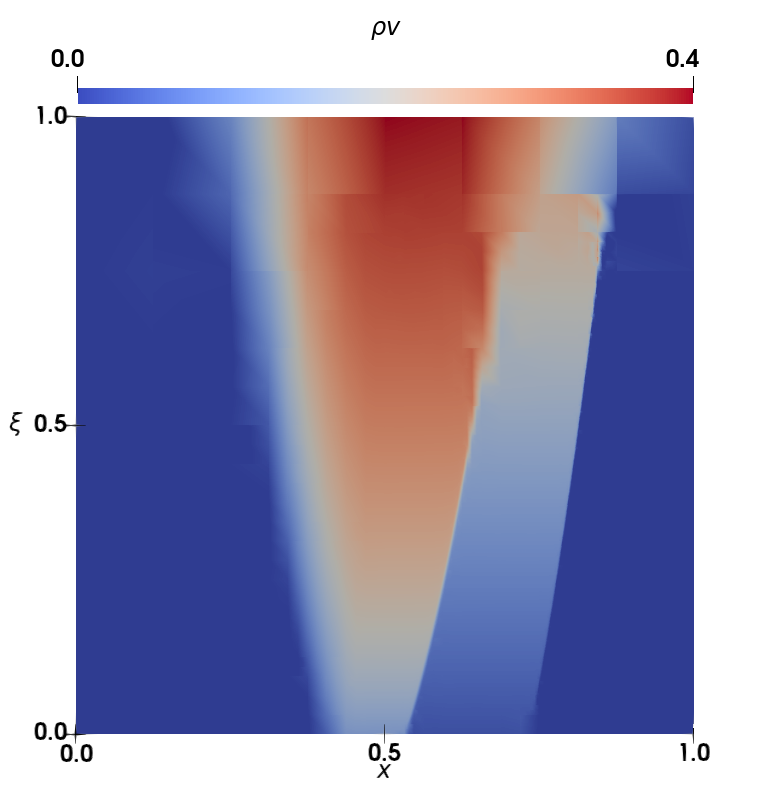}\hfill
  \includegraphics[width=.33\textwidth]{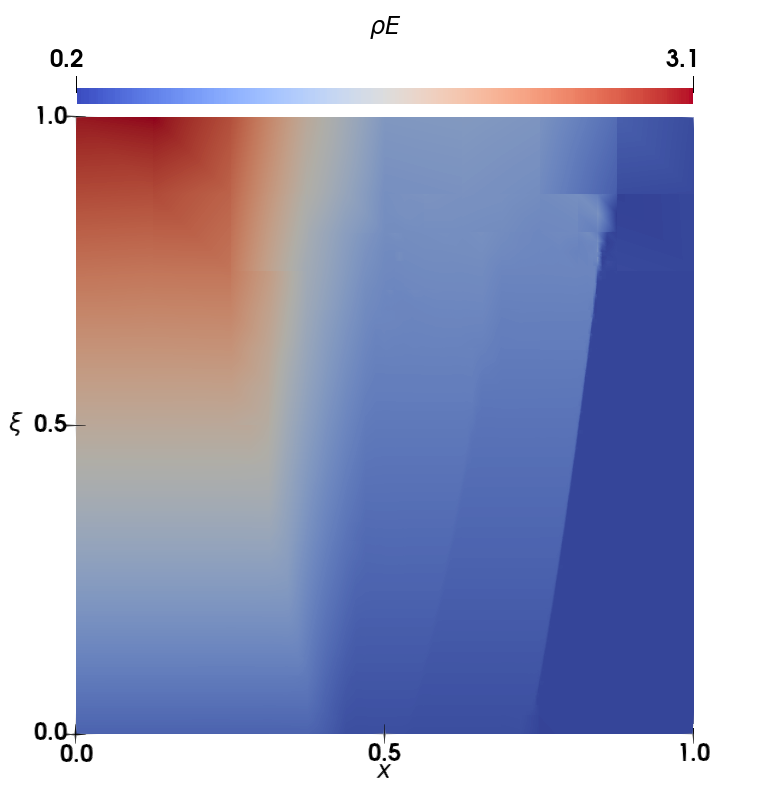}\\
  \includegraphics[width=.33\textwidth]{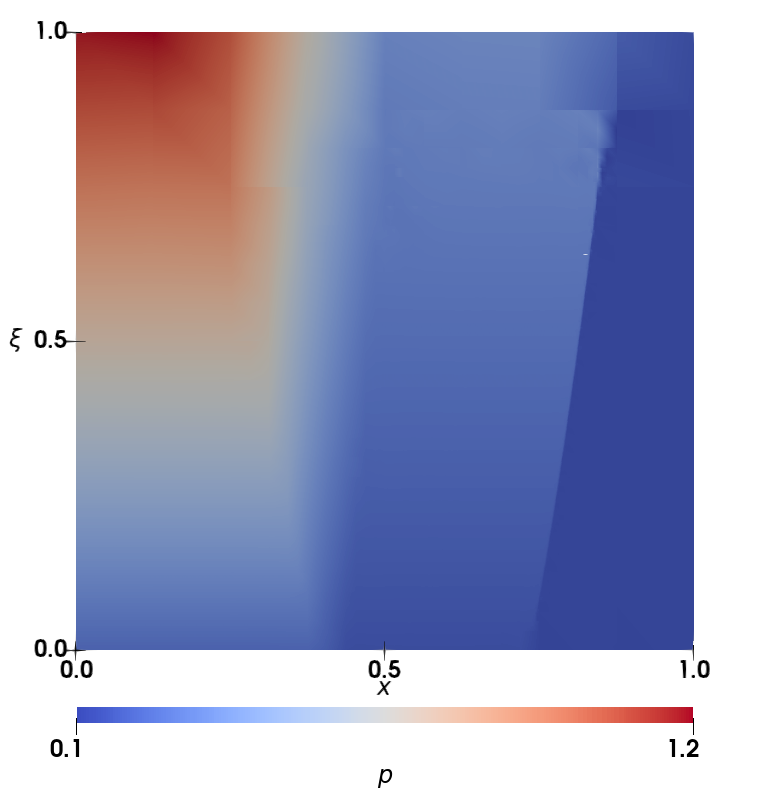}\hfill
  \includegraphics[width=.33\textwidth]{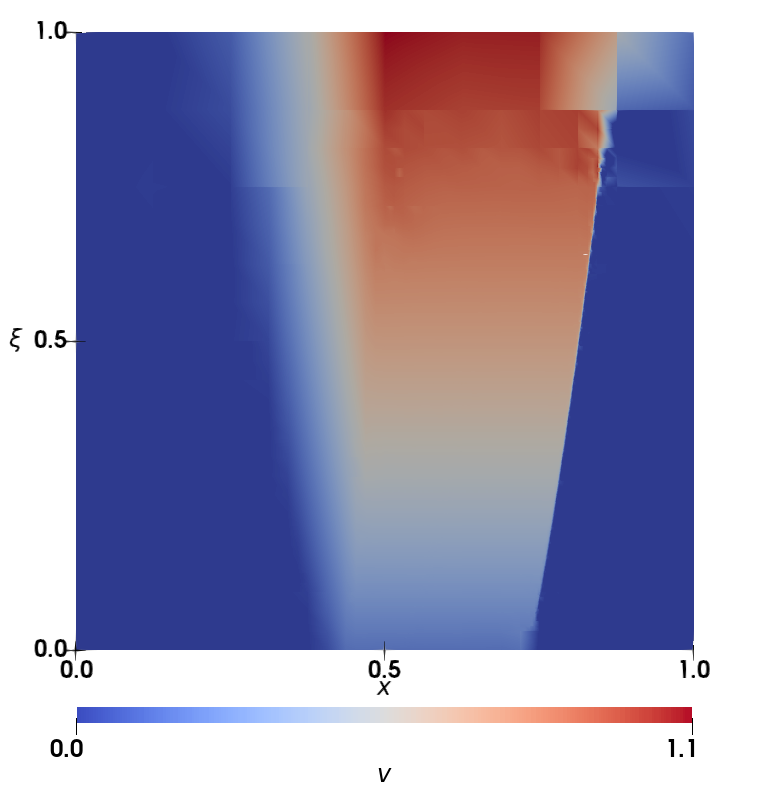}\hfill
  \includegraphics[width=.33\textwidth]{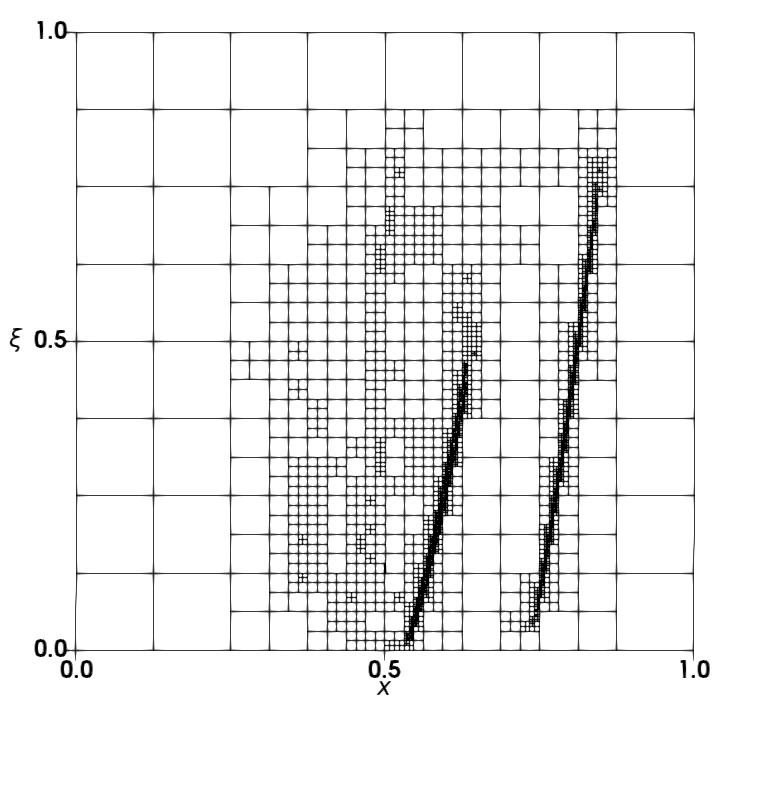}
  \caption{Solution for the Euler equations \eqref{eq:det_1d_euler} with uncertain initial data \eqref{eq:det_ic_euler} at time $t=0.2$ with $L=6$ refinement levels using weighted thresholding for the random variable $\xi\sim\mathcal B(2,5)$. Top row (from left to right): density $\rho$; momentum $\rho v$; density of energy $\rho E$. Bottom row (from left to right): pressure $p$; velocity $v$; corresponding adaptive grid.}
  \label{fig:2d_euler_plot_weighted}
\end{figure}

\paragraph{Comparison of uniform and weighted thresholding.}
We compare the $L^1$-error of the stochastic moments between our novel thresholding strategy and the classical thresholding method. As reference solution we performed a solution with $L=10$ refinement levels and uniform thresholding. These results are presented in Figure \ref{fig:euler_l1_err_comp}. With both methods, the $L^1$-error of the stochastic moments decreases by the empirical order of accuracy of about $1$ for all state variables, cf.~Table \ref{tab:ex02_eoc}. The $L^1$-error of the stochastic moments using weighted thresholding is of the same order of magnitude as the $L^1$-error of the moments with uniform thresholding and therefore comparable. If weighted thresholding is used instead, the solution itself has a poor convergence rate because the adaptation process is optimized for the stochastic moments. This can be seen, for example, in the $L^1$-error of the velocity which reflects the poor behavior of the solution in Figure \ref{fig:2d_euler_plot_weighted}. In addition, in Figure \ref{fig:euler_l1_err_comp} we show the ratio of the number of cells \eqref{eq:ratio_cell} between weighted thresholding and uniform thresholding. Using our novel strategy saves more than half the cells than using uniform thresholding.
With increasing refinement levels this ratio decreases. Figure \ref{fig:2d_euler_weighted_diff_grids} shows the adaptive grids with weighted thresholding for different maximum refinement levels $L=2,4,6$. As the refinement level increases, both the grids at the shock and the contact wave become more and more refined in the stochastic direction since the local threshold value \eqref{eq:locthreshval} becomes smaller with increasing refinement levels. Thus, cells with non-significant details with respect to the probability density function may become significant for higher refinement levels and the corresponding cell has therefore to be refined. This effect is problem-dependent and may not occur at all, as for example in Sect. \ref{subsec:burgers_cauchy_smooth}, cf. Fig. \ref{fig:burgers_cauchy_comparison_weighted_vs_non_weighted}.

\begin{figure}[htbp]
  \centering
  \begin{subfigure}[b]{0.32\textwidth}
    \begin{adjustbox}{width=\linewidth}
      \includegraphics{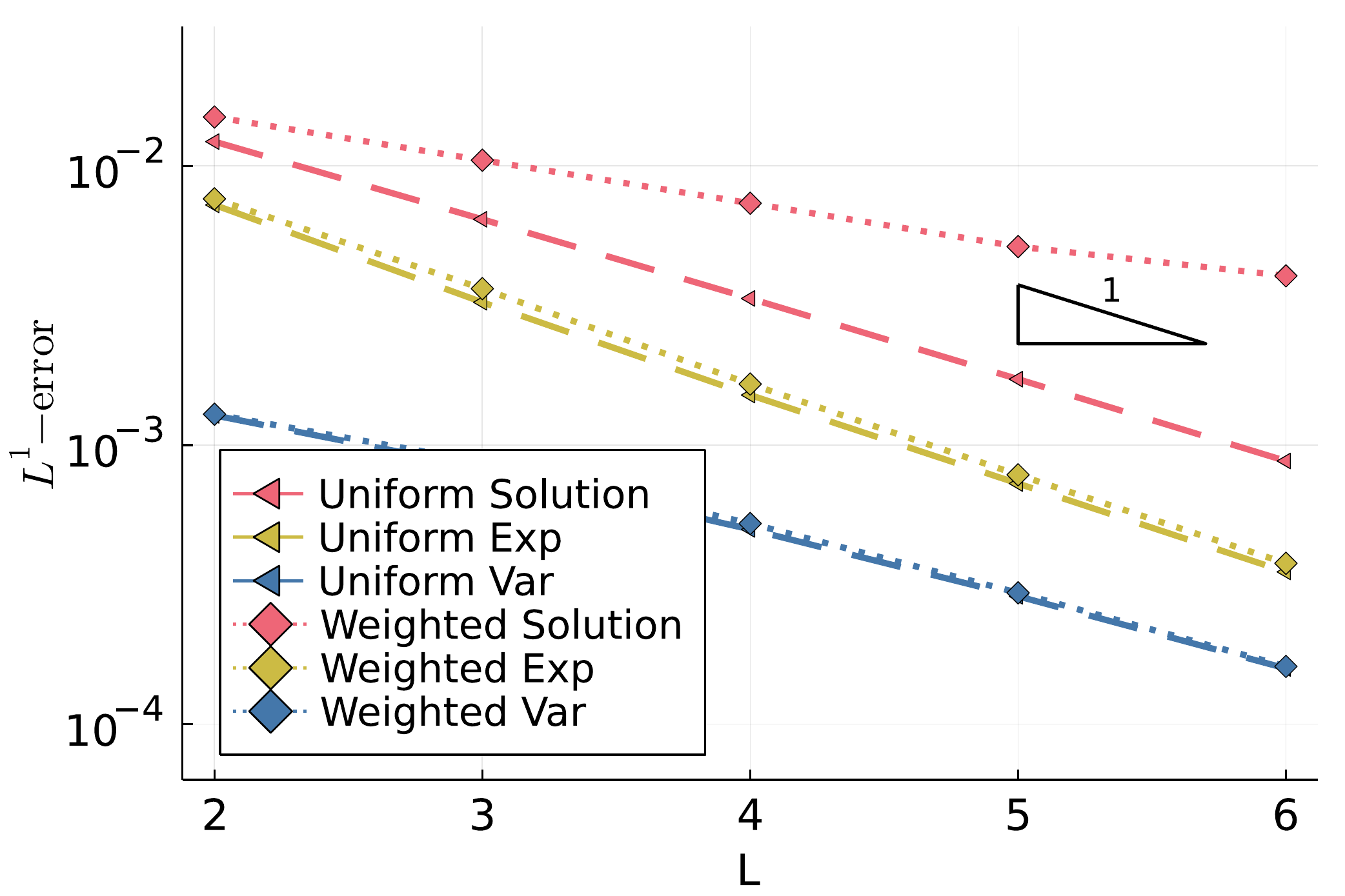}
    \end{adjustbox}
    \caption{Density $\rho$}
  \end{subfigure}
  \hfill
  \begin{subfigure}[b]{0.32\textwidth}
    \begin{adjustbox}{width=\linewidth}
      \includegraphics{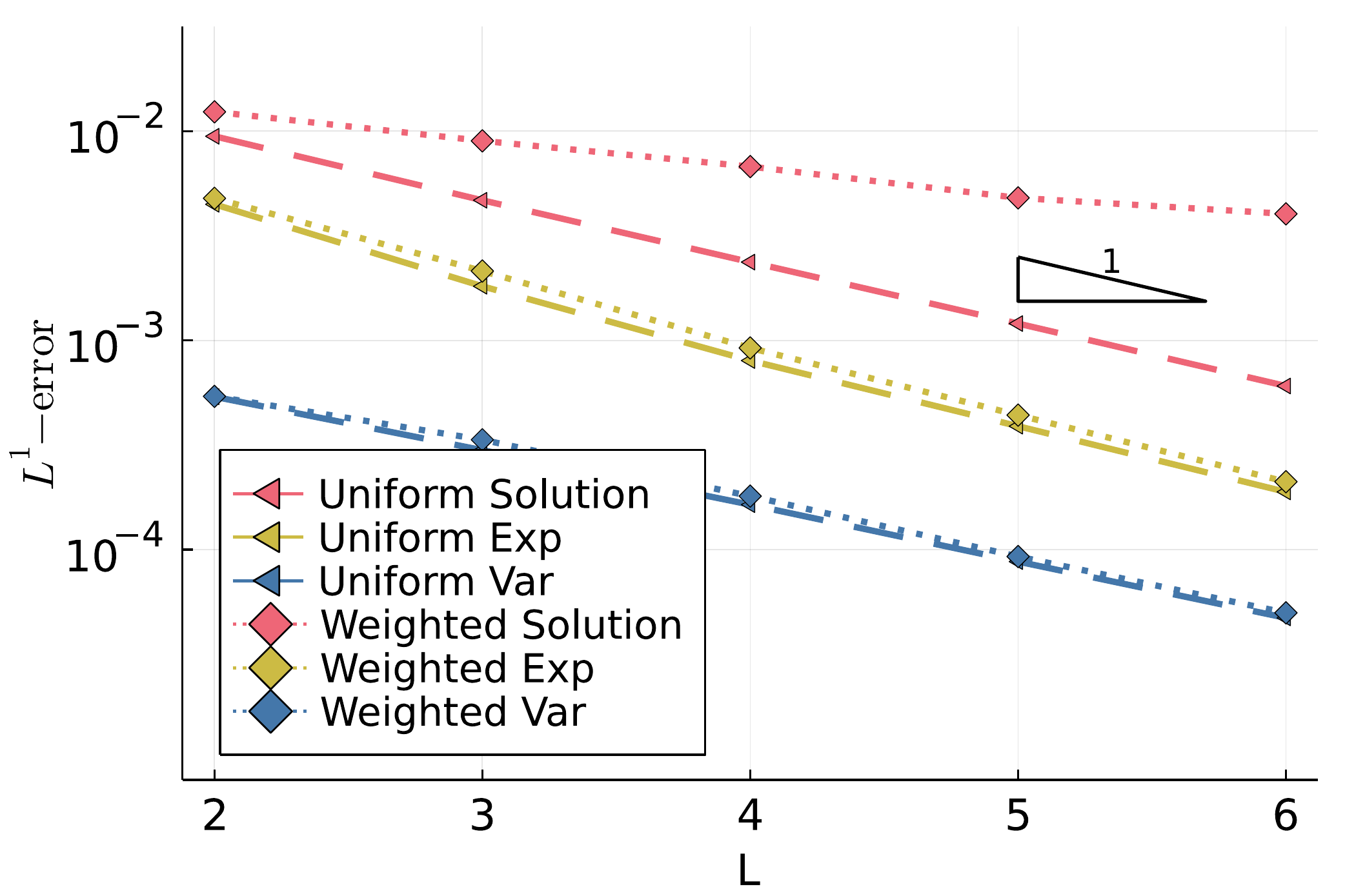}
    \end{adjustbox}
    \caption{Momentum $\rho v$}
  \end{subfigure}
  \hfill
  \begin{subfigure}[b]{0.32\textwidth}
    \begin{adjustbox}{width=\linewidth}
      \includegraphics{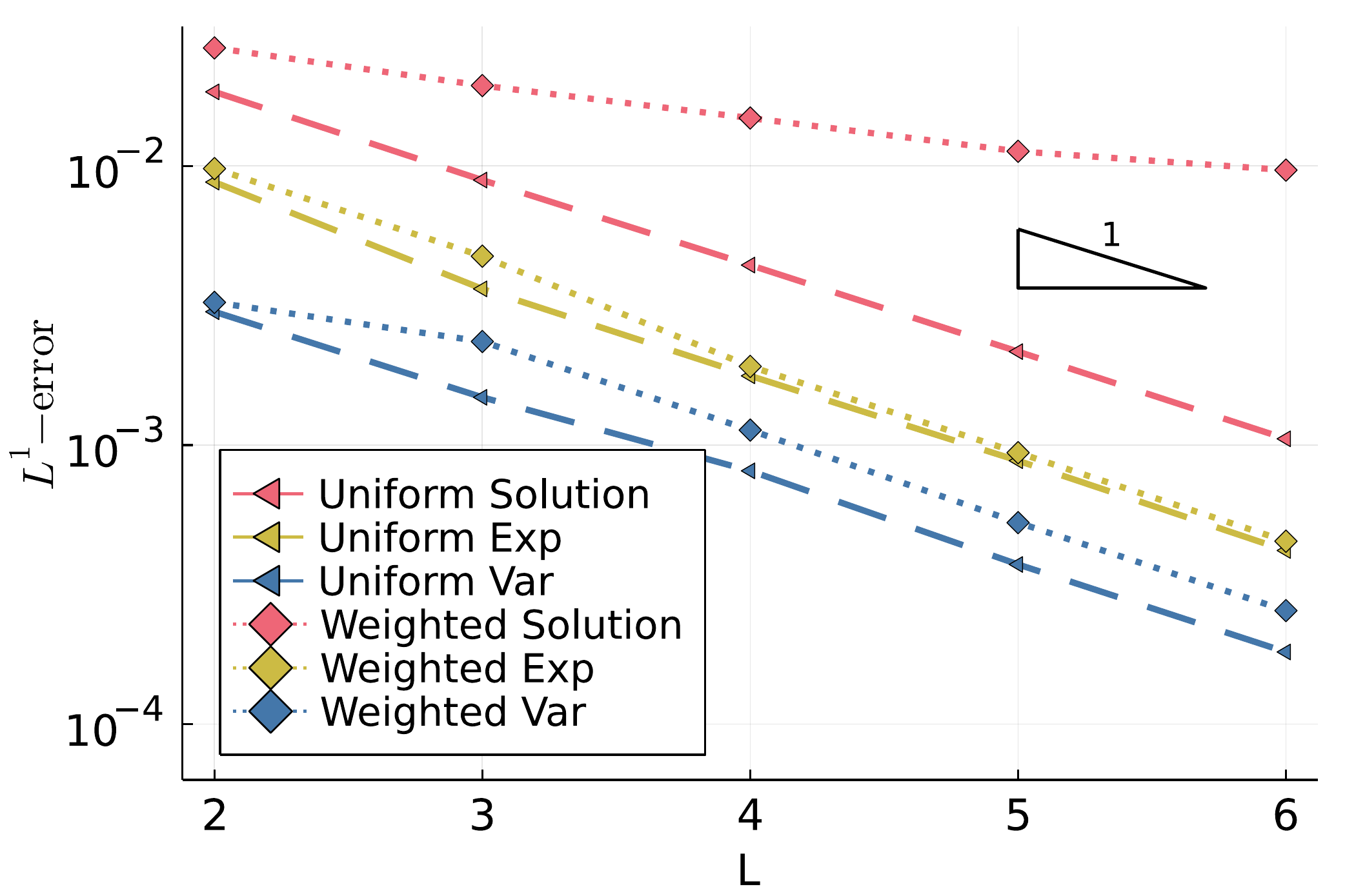}
    \end{adjustbox}
    \caption{Density of energy $\rho E$}
  \end{subfigure}
  \\
  \begin{subfigure}[b]{0.32\textwidth}
    \begin{adjustbox}{width=\linewidth}
      \includegraphics{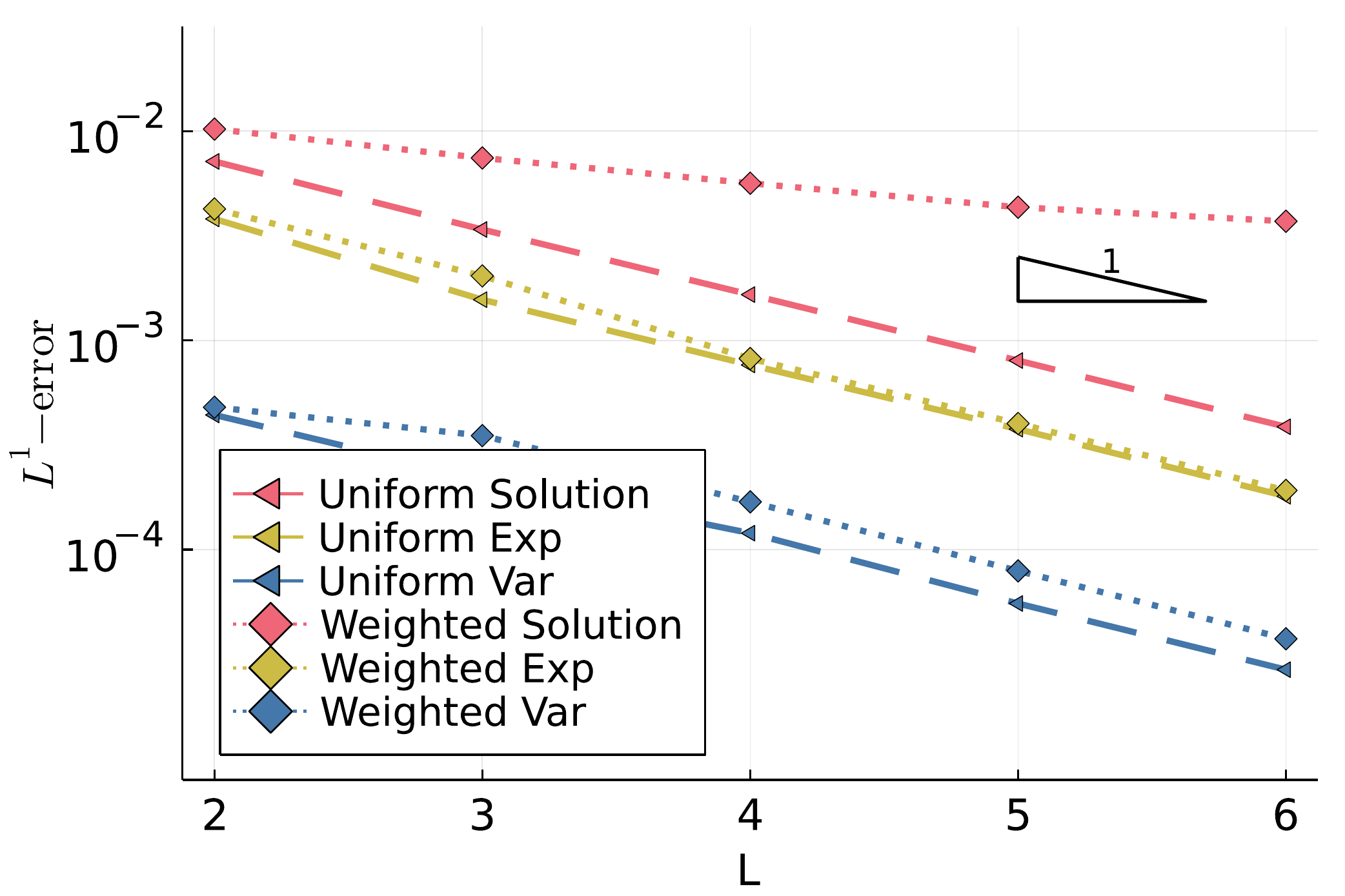}
    \end{adjustbox}
    \caption{Pressure $p$}
  \end{subfigure}
  \hfill
  \begin{subfigure}[b]{0.32\textwidth}
    \begin{adjustbox}{width=\linewidth}
      \includegraphics{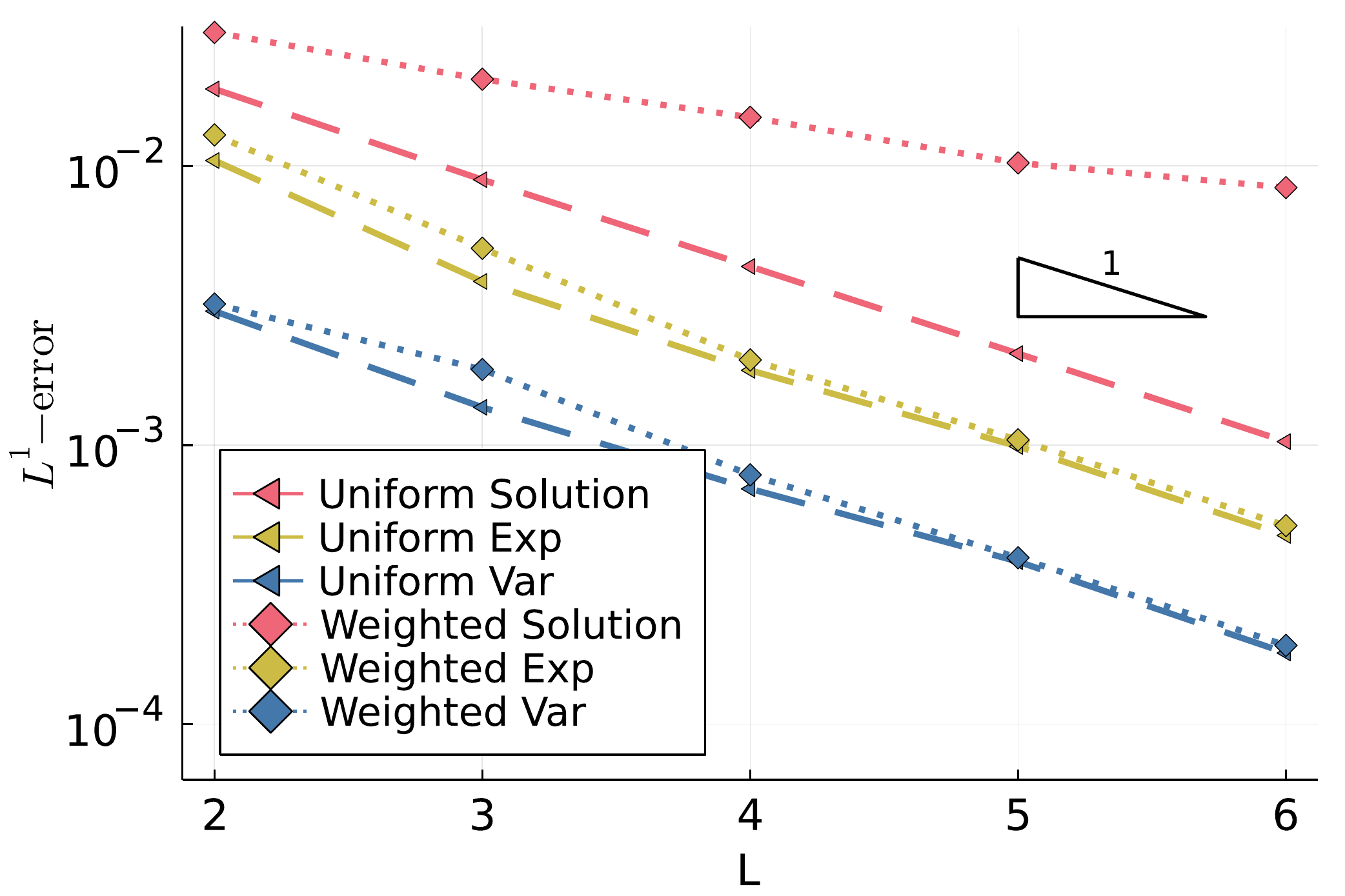}
    \end{adjustbox}
    \caption{Velocity $v$}
  \end{subfigure}
  \hfill
  \begin{subfigure}[b]{0.32\textwidth}
    \begin{adjustbox}{width=\linewidth}
      \includegraphics{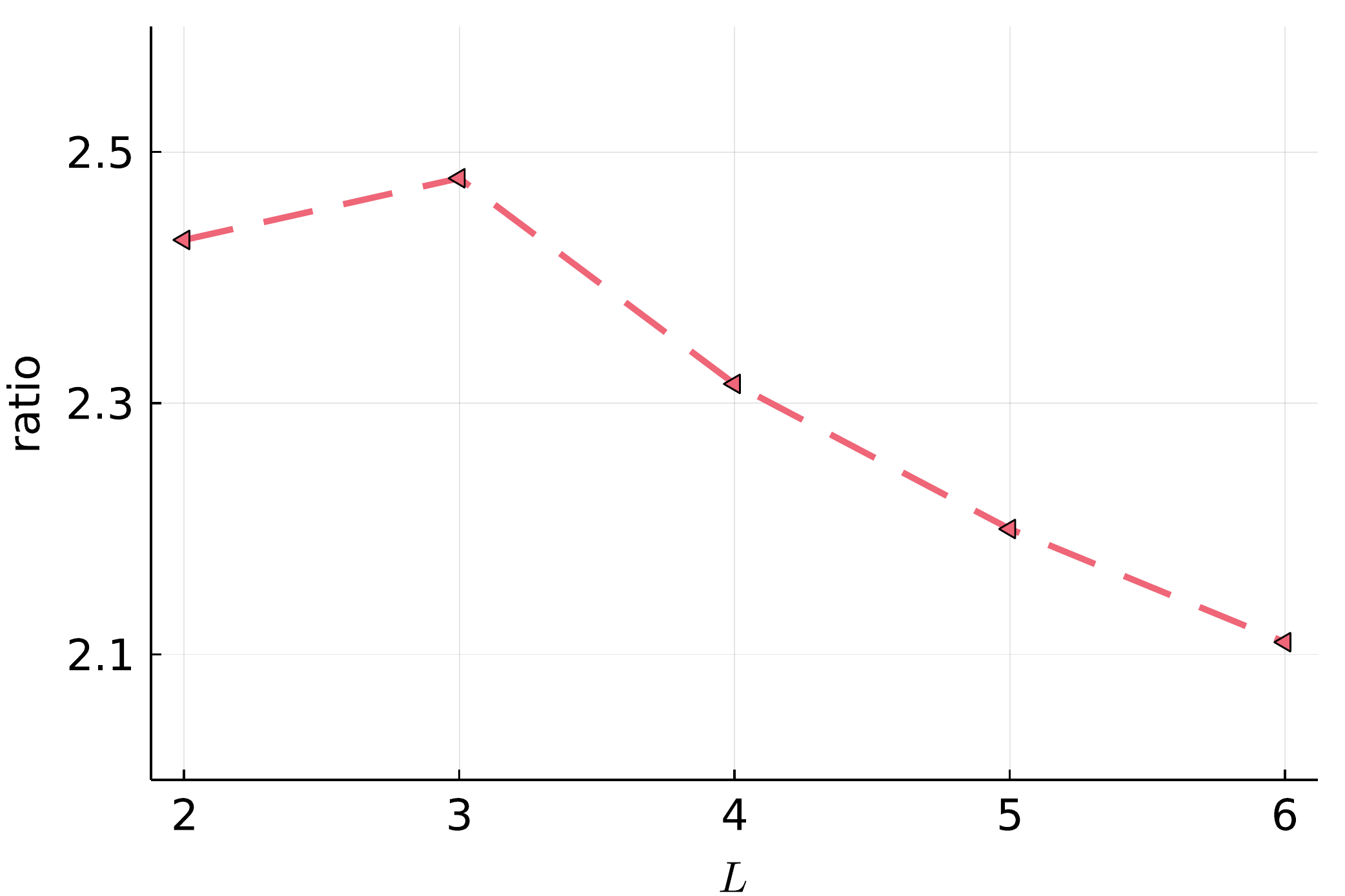}
    \end{adjustbox}
    \caption{Ratio number of cells}
  \end{subfigure}
  \caption{$L^1$-error of the stochastic moments and the solution of \eqref{eq:det_1d_euler}, \eqref{eq:det_ic_euler} at time $t=0.2$ using both uniform and weighted thresholding up to $L=6$ refinement levels and the ratio of the number of cells.}
  \label{fig:euler_l1_err_comp}
\end{figure}

\begin{table*}[htbp]
  \centering
  \setlength{\extrarowheight}{.6em}
  \setlength{\tabcolsep}{0.8em}
  \scalebox{0.683}{
    \begin{tabular}{|l|c||c|c|c|c|c|c|c|c|c|c|c|c|c|c|c|}
      \hline
      \multirow{2}{*}{} &     & \multicolumn{3}{c|}{Density $\rho$} & \multicolumn{3}{c|}{Momentum $\rho v$} & \multicolumn{3}{c|}{Density of energy $\rho E$} & \multicolumn{3}{c|}{Velocity $v$} & \multicolumn{3}{c|}{Pressure $p$}                                                                                           \\
      \cline{2-17}
                        & $L$ & sol                                 & exp                                    & var                                             & sol                               & exp                               & var    & sol    & exp    & var    & sol    & exp    & var    & sol    & exp    & var    \\
      \hline
      uniform           & 3   & 0.9246                              & 1.1590                                 & 0.6287                                          & 1.0154                            & 1.3011                            & 0.8481 & 1.0498 & 1.2710 & 1.0183 & 1.0800 & 1.4388 & 1.1449 & 1.0804 & 1.2820 & 1.0164 \\
                        & 4   & 0.9427                              & 1.1043                                 & 0.7290                                          & 0.9815                            & 1.1791                            & 0.8649 & 1.0122 & 1.0349 & 0.8752 & 1.0330 & 1.0577 & 0.9696 & 1.0331 & 1.0354 & 0.8604 \\
                        & 5   & 0.9604                              & 1.0533                                 & 0.7958                                          & 0.9754                            & 1.0415                            & 0.9037 & 1.0298 & 1.0105 & 1.1142 & 1.0370 & 0.9064 & 0.8662 & 1.0468 & 1.0184 & 1.1164 \\
                        & 6   & 0.9738                              & 1.0515                                 & 0.8587                                          & 0.9905                            & 1.0428                            & 0.8933 & 1.0378 & 1.0685 & 1.0419 & 1.0480 & 1.0594 & 1.0882 & 1.0507 & 1.0656 & 1.0503 \\
      \hline
      weighted          & 3   & 0.5136                              & 1.0689                                 & 0.5683                                          & 0.4608                            & 1.1541                            & 0.6885 & 0.4489 & 1.0425 & 0.4664 & 0.5565 & 1.3490 & 0.7791 & 0.4579 & 1.0629 & 0.4504 \\
      thresholding      & 4   & 0.5125                              & 1.1354                                 & 0.7327                                          & 0.4081                            & 1.2206                            & 0.8940 & 0.3858 & 1.3123 & 1.0538 & 0.4538 & 1.3292 & 1.2542 & 0.4014 & 1.3119 & 1.0517 \\ &5&0.5159&1.0804&0.8249&0.4968&1.0658&0.9590&0.3963&1.0269&1.1015&0.5427&0.9513&0.9869&0.3813&1.0308&1.0955\\
                        & 6   & 0.3485                              & 1.0525                                 & 0.8763                                          & 0.2536                            & 1.0575                            & 0.8934 & 0.2235 & 1.0537 & 1.0474 & 0.2941 & 1.0211 & 1.0431 & 0.2228 & 1.0621 & 1.0763 \\
      \hline
    \end{tabular}}
  \caption{Empirical order of convergence of the $L^1$-error of the solution and the $L^1$-error of the expectation and variance of \eqref{eq:det_1d_euler}, \eqref{eq:det_ic_euler} using adaptive grid with uniform thresholding and weighted thresholding.}
  \label{tab:ex02_eoc}
\end{table*}

\begin{figure}[htbp]
  \centering
  \includegraphics[width=.33\textwidth]{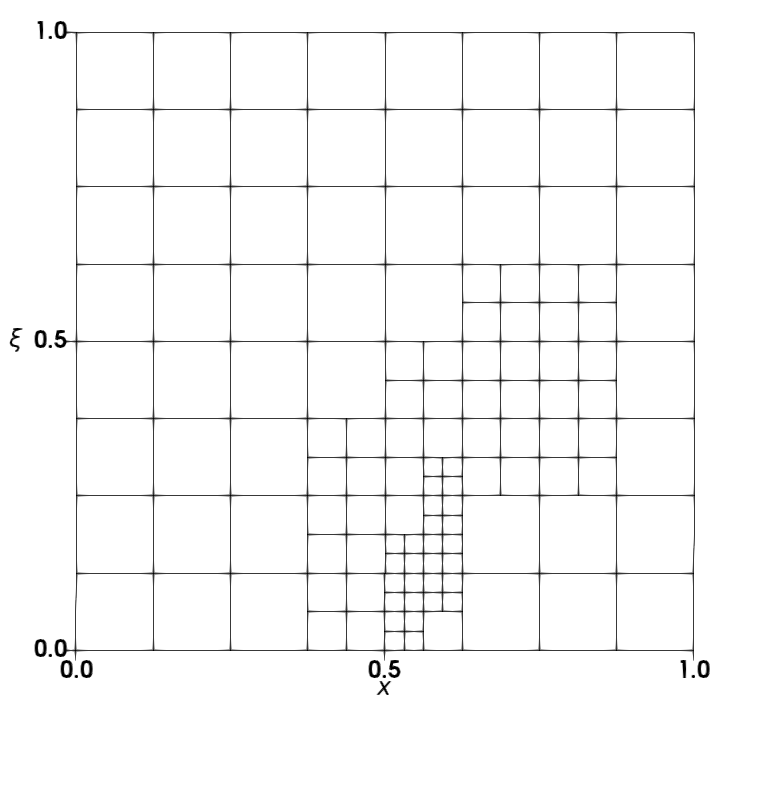}\hfill
  \includegraphics[width=.33\textwidth]{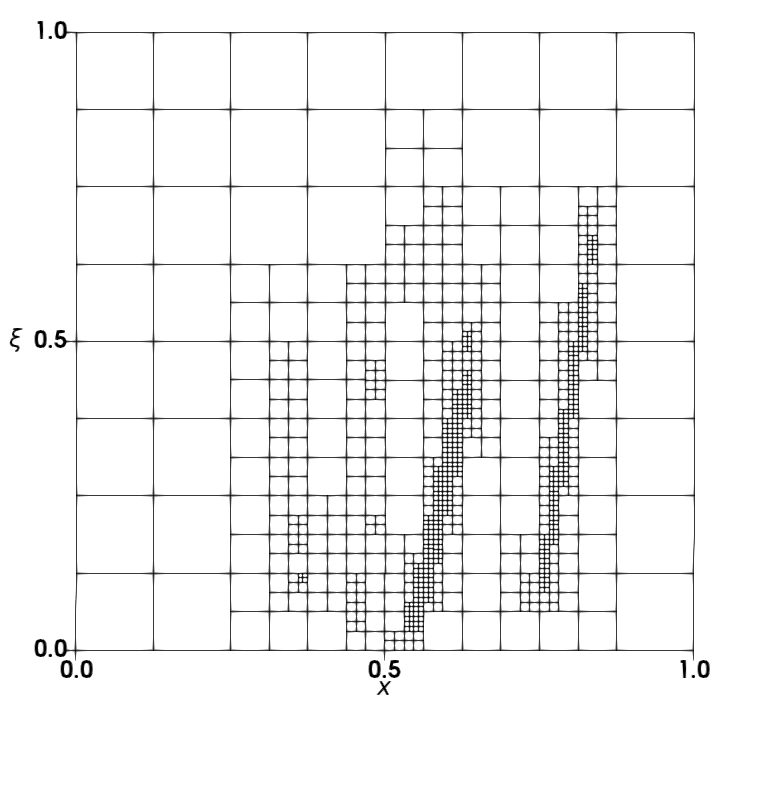}\hfill
  \includegraphics[width=.33\textwidth]{Figures/ex02_beta_2_5_grid.png}
  \caption{Adaptive grids with weighted thresholding for the Euler equations \eqref{eq:det_1d_euler} with uncertain initial data \eqref{eq:det_ic_euler} at time $t=0.2$ for the random variable $\xi\sim\mathcal B(2,5)$ and different maximum refinement levels $L$. From left to right: $L=2$; $L=4$; $L=6$.}
  \label{fig:2d_euler_weighted_diff_grids}
\end{figure}

\section{Conclusion}
\label{sec:Concl}
In the present work we investigate  the solution of conservation laws with uncertain initial data. For this purpose, we formulate the stochastic problem as a higher-dimensional deterministic problem. To both, the solution and its moments determined by averaging over the stochastic direction, we apply a novel multiresolution analysis that allows to investigate the interaction of the spatial scales with the stochastic scales. In particular, we identify the relevant scales in the spatial and stochastic variables in the solution  that affect the moments. Depending on the probability density corresponding to the random variable not all scales in the solution will affect the scales in the moments. This insight is used to design a new multiresolution based grid adaptation strategy for the  approximation of the deterministic problem. The adapted scheme shows higher efficiency and accuracy in the moments of the solution. Numerical results verify the analytical results.

In contrast to stochastic \textit{elliptic} or \textit{parabolic} PDEs, typically discontinuities  occur in the spatial solution of hyperbolic conservation laws. This reduces the regularity in the stochastic variables which can be seen in our numerical results. 
Therefore, we believe that it is important to understand the interaction of spatial and stochastic scales as we have done  here for arbitrary number of stochastic variables $m\in\N$. The analytical results are confirmed numerically for  $m=1$. Because of the curse of dimensionality this will not be feasible for higher dimension $m+d> 3$. For this purpose, the adaptation strategy has to be applied directly to the moments instead of the solution. The current investigations will be helpful in this regard.

\appendix
\section{Appendix}
\label{appendix:proof_lemmata}

\begin{proof}[Theorem \ref{thm:equiv_random_det}]~\\
  For simplicity we only prove the case of a positive probability density $p_\xi$. In the case of a compactly supported probability density $p_\xi$, we define $\Omega_\xi := \supp(p_\xi)$ and  extend the entropy solutions $\bar u\equiv 0$ and $u\equiv 0$ for all $\omega_\xi \in \R^m\setminus \Omega_\xi$. Then, we proceed analogously to the proof in the case of positive density.

  Let $\bar u $ be the unique entropy solution according to  Theorem \ref{thm:stochastic_entropy_solution}. For  $u$ defined by \eqref{eq:equiv_random_det} we have to verify the properties of Definition \ref{def:entropy_solution}.\\
  Let $\varphi\in C^1_0([0,T]\times{\R^{d + m}})$ be a test function. Then for $\omega_\xi\in\Omega_\xi$ the restriction
  \begin{align}
    \label{eq:test-fct-1}
    \bar\varphi(t,\bx;\omega_\xi) := \varphi(t,(\bx,\omega_\xi))\,(p_\xi(\omega_\xi))^{-1}
  \end{align}
  is a test function in $ C^1_0([0,T]\times\R^d)$ and the weak formulation \eqref{eq:stochastic_weak_formulation} holds for $\Prob_\xi$-a.s. $\omega_\xi\in\Omega_\xi$.
  Using the Radon-Nikodym theorem, integration  of \eqref{eq:stochastic_weak_formulation} over the induced probability space leads to
  \begin{align*}
      & \int_{\Omega_\xi}  \left(\int_0^\infty\int_{{\R^{d}}} \right. \bar u(t,\bx;\omega_\xi) \bar \varphi_t(t,\bx;\omega_\xi) + \sum_{j=1}^{d} \bof_j(\bar u(t,\bx;\omega_\xi))\pardev{\bx_j}\bar \varphi(t,\bx;\omega_\xi) \dif \bx\dif t \\
      & \hspace{20mm} \left. + \int_{{\R^{d}}}\bar{u}_0(\bx;\omega_\xi)\bar \varphi(0,\bx;\omega_\xi)\dif \bx \right)\dif\Prob_\xi(\omega_\xi)                                                                                                       \\
    = & \int_{\R^m} \left(\int_0^\infty\int_{\R^d}u(t,\bx,\bxi)\varphi_t(t,\bx,\bxi) + \sum_{j=1}^{d+m} \bof_j(u(t,\bx,\bxi))\pardev{\bx_j}\varphi(t,\bx,\bxi)\dif \bx\dif t \right.                                                         \\ & \hspace{20mm}\left.
    +  \int_{\R^d} u_0(\bx,\bxi)\varphi(0,\bx,\bxi)\dif \bx\right) \,(p_\xi(\bxi))^{-1}\,p_\xi(\bxi)\dif\bxi = 0
  \end{align*}
  for $u(t,\bx,\omega_\xi) := \bar u(t,\bx;\omega_\xi)$ for $\Prob_\xi$-a.s. $\omega_\xi\in\Omega_\xi$ and for a.e. $\bx\in\R^d$ using \eqref{eq:equiv_random_det_init}, \eqref{eq:test-fct-1} and \eqref{eq:equiv_random_det_flux}.
  Using Fubini's theorem we finally obtain the weak formulation \eqref{eq:weak_formulation}.
  Similarly, the entropy condition \eqref{eq:stoachastic_entropy_condition} for $\bar u$ implies the entropy condition \eqref{eq:entropy_condition} for $u$ using \eqref{eq:equiv_random_det_entropyflux}.
  Note that due to \eqref{eq:test-fct-1} the test function $\overline{\varphi}$ is non-negative if and only if $\varphi$ is non-negative.

  To verify that $u\in C_b([0,T],L_{\text{loc}}^1(\R^{d+m}))$ for $T>0$, we observe that
  \begin{align*}
    \| u \|_{C_b([0,T],L^1(\R^{d+m}))} \leq \| (p_\xi)^{-1} \|_{L^\infty(\Omega_\xi)}  \| \bar u \|_{C_b([0,T],L^1(\R^{d}))}
  \end{align*}
  using the Radon-Nikodym theorem.

  Conversely, we assume that $u$  is the entropy solution of the deterministic Cauchy problem \eqref{eq:det_problem} and define
  $  \bar u(t,x;\omega_\xi)  := u(t,(x,\omega_\xi))$ for $\Prob_\xi$-a.s. $\omega_\xi\in\Omega_\xi$ and for a.e. $x\in\R^d$.
  We now verify that  the weak formulation \eqref{eq:weak_formulation} implies the stochastic weak formulation \eqref{eq:stochastic_weak_formulation}. For this purpose, let be $\bar \varphi\in C^1_0([0,T]\times{\R^d})$ an arbitrary test function. Furthermore, for $\varepsilon > 0$ let be $J_\varepsilon:{\R^{m}}\rightarrow\R$ the rescaled mollifier $J_\varepsilon(\bxi):=\frac{1}{\varepsilon^m}J(\bxi/\varepsilon)$ with
  \begin{align*}
    J(\bxi) := \begin{cases}
                 c_m \exp\left(\frac{1}{|\bxi|^2-1}\right) & , |\bxi| < 1    \\
                 0                                         & , |\bxi| \geq 1
               \end{cases}
  \end{align*}
  and $c_m > 0$ chosen such that $\int_{{\R^{m}}}J(\bxi)\dif \bxi = 1$.
  By means of the rescaled mollifier we define for fixed $\bar\bxi\in \Omega_\xi$ and
  $\varepsilon > 0$
  chosen such that $B_\varepsilon(\bar\bxi)\subset \Omega_\xi$
  the smooth function
  \begin{align*}
    \varphi(t,x,\bxi) := \bar\varphi(t,x)  J_\varepsilon(\bar \bxi-\bxi)\, p_\xi(\bxi),\quad \bxi\in\R^m,
  \end{align*}
  where $B_\varepsilon(\bar\bxi)$ is an open ball with radius $\varepsilon >0$ and center $\bar\bxi\in \Omega_\xi$. Note that  the support of $\varphi$ is bounded because $\operatorname{supp} J_\varepsilon(\bar \bxi-\cdot) = B_\varepsilon(\bar \bxi)$ and $\operatorname{supp}\varphi\subset\operatorname{supp}\bar\varphi\times B_\varepsilon(\bar \bxi)\subset\operatorname{supp}\bar\varphi\times \Omega_\xi$. Therefore, it holds $\varphi \in C^1_0([0,T]\times{\R^{d + m}})$.
  Then we rewrite \eqref{eq:weak_formulation} applying Fubini's theorem and \eqref{eq:equiv_random_det_flux}
  \begin{alignat*}{2}
    \begin{split}
      0 = \int_{\R^m} \left(  \int_0^\infty\int_{{\R^d}} \right. & u(t,x,\bxi) \varphi_t(t, x, \bxi) + \sum_{j=1}^d \bof_j(u(t,x,\bxi))\pardev{x_j}\varphi(t,x,\bxi)\dif x\dif t \\
      & \left. + \int_{{\R^d}} u_0(x,\bxi)\varphi(0,x,\bxi)\dif x \right)\dif\bxi
    \end{split}
    \\
    \begin{split}
      = \int_{\R^m} \left(  \int_0^\infty\int_{{\R^d}} \right. & u(t,x,\bxi) \bar\varphi_t(t,x) + \sum_{j=1}^d \bof_j(u(t,x,\bxi))\pardev{x_j}\bar\varphi(t,x)\dif x\dif t \\
      & \left. + \int_{{\R^d}} u_0(x,\bxi)\bar\varphi(0,x)\dif x \right) J_\varepsilon(\bar \bxi-\bxi) p_\xi(\bxi)\dif\bxi.
    \end{split}
  \end{alignat*}
  Since $\varphi$ has compact support we introduce the weighted residual
  \begin{align*}
    R(\bxi) :=
    p_\xi(\bxi)  \int_0^\infty\int_{{\R^d}} & u(t,x,\bxi) \bar\varphi_t(t,x) + \sum_{j=1}^d \bof_j(u(t,x,\bxi))\pardev{x_j}\bar\varphi(t,x)\dif x\dif t \\
                                            & + \int_{{\R^d}} u_0(x,\bxi)\bar\varphi(0,x)\dif x.
  \end{align*}
  With the convolution $ R_\varepsilon( \bar \bxi ) := (J_\varepsilon\ast R)(\bar\bxi)$ it holds $R_\varepsilon(\bar\bxi)\rightarrow R(\bar\bxi)$, $\varepsilon\to 0$, for a.e. $\bar\bxi\in \Omega_\xi$ leading to $R(\bar\bxi) = 0$ for a.e. $\bar\bxi\in \Omega_\xi$ and therefore $R(\bxi) = 0$ for a.e. $\bxi\in\R^m$.
  Integrating the absolute value of the weighted residual over $\R^m$ we obtain for $\bxi=\omega_\xi$:
  \begin{align*}
    \begin{split}
      \int_{\Omega_\xi} \left|  \int_0^\infty\int_{{\R^d}} \right. & u(t,x,\omega_\xi) \bar\varphi_t(t,x) + \sum_{j=1}^d \bof_j(u(t,x,\omega_\xi))\pardev{x_j}\bar\varphi(t,x)\dif x\dif t \\
      & \left. + \int_{{\R^d}} u_0(x,\omega_\xi)\bar\varphi(0,x)\dif x \right|  \dif\Prob_\xi(\omega_\xi) = 0.
    \end{split}
  \end{align*}
  Since the Lebesgue measure is $\sigma$-finite and $\Prob_\xi$ is a finite measure, the null sets of
  the probability measure $\Prob_\xi$ and the Lebesgue measure are consistent and
  therefore the weak formulation holds for $\Prob_\xi$-a.e. $\omega_\xi\in\Omega_\xi$.

  To verify that the entropy condition \eqref{eq:entropy_condition} implies the stochastic entropy condition \eqref{eq:stoachastic_entropy_condition}, we may proceed analogously.

  It remains to verify that $\bar u(\cdot,\cdot;\omega_\xi) \in C_b([0,T],L^1(\R^{d}))$ for $T>0$ and for $\Prob_\xi$-a.s. $\omega_\xi\in\Omega_\xi$. Using the Radon-Nikodym theorem we observe that
  \begin{align*}
    \max_{t\in[0,T]} \| \bar u(t,\cdot;\omega_\xi) \|_{L^1(\R^d)}  \leq
    \|  p_\xi \|_{L^\infty([\omega_\xi-\varepsilon,\omega_\xi+\varepsilon])} \| u \|_{C_b([0,T],L^1(\R^{d}\times [\omega_\xi-\varepsilon,\omega_\xi+\varepsilon]))}
  \end{align*}
  for arbitrary $\varepsilon>0$
  which completes the proof.

  \qed
\end{proof}

\begin{proof}[Lemma \ref{la:estimates-expectation}]~\\
  By Jensen's inequality with $\Phi(x) = |x|^q$ for $q\in[1,\infty)$, $x\in\R$, ($\Phi$ convex and non-negative) and Fubini  we obtain
  \begin{align*}
    \Vert \E[u^k] \Vert^q_{L^q(\Omega_1)} & =
    \int_{\Omega_1} \left| \int_{\Omega_2} u^k(\bx_1,\bx_2) p_\xi(\bx_2) \dif\bx_2\right|^q \dif\bx_1
    \le \int_{\Omega_1}  \int_{\Omega_2} \left|u^k(\bx_1,\bx_2) p_\xi(\bx_2)\right|^q \dif\bx_2 \dif\bx_1 \\
                                          & \le \Vert p_\xi \Vert^q_{L^\infty(\Omega_2)}
    \int_{\Omega_1}  \int_{\Omega_2} \left|u^k(\bx_1,\bx_2) \right|^q \dif\bx_2 \dif\bx_1  =
    \Vert p_\xi \Vert^q_{L^\infty(\Omega_2)}
    \Vert u^k \Vert^q_{L^q(\Omega)} .
  \end{align*}
  This estimate implies \eqref{eq:estimate-expectation-1}.

  In case of $q=1$ we proceed as follows
  \begin{align*}
    \Vert \E[u^k] \Vert_{L^1(\Omega_1)} & =
    \int_{\Omega_1} \left| \int_{\Omega_2} u^k(\bx_1,\bx_2) p_\xi(\bx_2) \dif\bx_2\right| \dif\bx_1
    \le \int_{\Omega_1}  \int_{\Omega_2} \left|u^k(\bx_1,\bx_2) p_\xi(\bx_2)\right|\dif\bx_2 \dif\bx_1 \\
                                        & =
    \int_{\Omega_2}  \int_{\Omega_1} \left|u^k(\bx_1,\bx_2) \right|  \dif\bx_1 p_\xi(\bx_2)\dif\bx_2
    =
    \int_{\Omega_2}  \Vert u^k(\cdot,\bx_2) \Vert_{L^1(\Omega_1)}  p_\xi(\bx_2)\dif\bx_2               \\
                                        & =
    \E[\Vert u^k \Vert_{L^1(\Omega_1)} ]  .
  \end{align*}
  Applying Fubini and using the assumption on $p_\xi$ we conclude
  \begin{align*}
    \E[\Vert u^k \Vert_{L^1(\Omega_1)} ] & =
    \int_{\Omega_2}  \int_{\Omega_1} \left|u^k(\bx_1,\bx_2) \right|  \dif\bx_1 p_\xi(\bx_2)\dif\bx_2 \\
                                         & \le
    \Vert p_\xi \Vert_{L^\infty(\Omega_2)}
    \int_{\Omega_1}  \int_{\Omega_2} \left|u(\bx_1,\bx_2) \right|^k \dif\bx_2 \dif\bx_1
    =
    \Vert p_\xi \Vert_{L^\infty(\Omega_2)}
    \Vert u \Vert^k_{L^k(\Omega)},
  \end{align*}
  i.e., \eqref{eq:estimate-expectation-2} holds.
  To estimate $\E[u^k]$ in the $L^\infty(\Omega_1)$-norm we estimate analogously to the case $q=1$:
  \begin{align*}
    \Vert \E[u^k] \Vert_{L^\infty(\Omega_1)} & =
    \esssup_{\bx_1\in\Omega_1} \left\{  \left| \int_{\Omega_2} u^k(\bx_1,\bx_2) p_\xi(\bx_2) \dif\bx_2\right|  \right\}
    \\
                                             & \le \esssup_{(\bx_1,\bx_2)\in\Omega_1\times \Omega_2} \{ |u^k(\bx_1,\bx_2) | \}  \int_{\Omega_2}  p_\xi(\bx_2)\dif\bx_2 =
    \Vert u^k \Vert_{L^\infty(\Omega)} \le \Vert u \Vert^k_{L^\infty(\Omega)} ,
  \end{align*}
  where we use that $\Vert p_\xi \Vert_{L^1(\Omega_2)}=1$ holds. This proves \eqref{eq:estimate-expectation-3}.

  To verify \eqref{eq:estimate-expectation-1b} we proceed in analogy to the proof of \eqref{eq:estimate-expectation-1} and obtain for $q\in[1,\infty)$
  \begin{align*}
    \Vert \E^k[u] \Vert^q_{L^q(\Omega_1)} & =
    \int_{\Omega_1} \left| \int_{\Omega_2} u(\bx_1,\bx_2) p_\xi(\bx_2) \dif\bx_2 \right|^{k q} \dif\bx_1
    \le
    \int_{\Omega_1}  \int_{\Omega_2} |u(\bx_1,\bx_2) p_\xi(\bx_2) |^{k q} \dif\bx_2  \dif\bx_1 \\
                                          & \le
    \Vert p_\xi \Vert^{k q}_{L^\infty(\Omega_2)}
    \int_{\Omega_1}  \int_{\Omega_2} |u(\bx_1,\bx_2) |^{k q}\dif\bx_2  \dif\bx_1
    =
    \Vert p_\xi \Vert^{k q}_{L^\infty(\Omega_2)} \Vert u^k \Vert^q_{L^q(\Omega)} .
  \end{align*}

  For the case $q=\infty$, using $\|p_\xi\|_{L^1(\Omega_2)}=1$ yields
  \begin{align*}
    \Vert \E^k[u]\Vert_{L^\infty(\Omega)} & = \esssup_{\bx_1\in\Omega_1}\left\{\left|\int_{\Omega_2}u(\bx_1,\bx_2)p_\xi(\bx_2)\dif x_2\right|^k\right\}
    \leq \esssup_{\bx_1\in\Omega_1}\left\{\left(\|u(\bx_1,\cdot)\|_{L^\infty(\Omega_2)}\|p_\xi\|_{L^1(\Omega_2)}\right)^k\right\}                       \\
                                          & = \esssup_{\bx_1\in\Omega_1}\left\{\|u(\bx_1,\cdot)\|_{L^\infty(\Omega_2)}^k\right\}
    \leq \|u\|^k_{L^\infty(\Omega)}.
  \end{align*}
  \qed
\end{proof}

\begin{proof}[Lemma \ref{la:error-differences-expectation-moments}]~\\
  First, we note that  the relation
  \begin{align*}
    a^n-b^n =(a-b) \sum_{k=0}^{n-1} a^{n-1-k} b^k
  \end{align*}
  yields
  \begin{align}
    \label{eq:binom-3}
    |a^n-b^n| \le n\, \max\{|a|,|b|\}^{n-1} \,|a-b| .
  \end{align}
  By the definition of the expectation we deduce
  \begin{align}
    | \E[u^k](\bx_1) - \E[v^k](\bx_1) | & \le
    \int_{\Omega_2} |u^k(\bx_1,\bx_2) - v^k(\bx_1,\bx_2) | p_\xi(\bx_2) \dif\bx_2 \nonumber                                                                 \\
                                        & \le
    \int_{\Omega_2}  k\,(\max\{ | u(\bx_1,\bx_2) |, | v(\bx_1,\bx_2) | \} )^{k-1} |u(\bx_1,\bx_2) - v(\bx_1,\bx_2) |  p_\xi(\bx_2)  \dif\bx_2
    \label{eq:help}
    \\
                                        & \le k\,( \esssup_{\bx\in\Omega}\{ (\max\{ | u(\bx) |, | v(\bx) | \}  \} )^{k-1}
    \int_{\Omega_2} |u(\bx_1,\bx_2) - v(\bx_1,\bx_2) | p_\xi(\bx_2)  \dif\bx_2                                                                    \nonumber \\
                                        & \le k\, (M(u,v))^{k-1}\, \E[ |u-v| ](\bx_1) .\nonumber
  \end{align}
  For $q\in[1,\infty)$ integration then yields the inequality \eqref{eq:error-expectation-k-1}:
  \begin{align*}
    \Vert \E[u^k] - \E[v^k] \Vert^q_{L^q(\Omega_1)}
     & \le
    ( k\, (M(u,v))^{k-1})^q\,\int_{\Omega_1} (\E[ |u-v| ](\bx_1))^q \dif\bx_1 \\
     & =
    ( k\, (M(u,v))^{k-1})^q\,\int_{\Omega_1}| \E[ |u-v| ](\bx_1) |^q  \dif\bx_1 =
    ( k\, (M(u,v))^{k-1})^q\,\Vert  \E[ |u-v| ] \Vert^q_{L^q(\Omega_1)} .
  \end{align*}
  For $q=\infty$
  we take the $\esssup_{\bx_1\in\Omega_1}$ in \eqref{eq:help}.
  To verify the second inequality \eqref{eq:error-expectation-k-2} we first observe by \eqref{eq:binom-3}
  \begin{align*}
    | \E^k[u](\bx_1) - \E^k[v](\bx_1) | & \le
    k\,(\max\{ | \E[u](\bx_1) |, | \E[v](\bx_1) | \} )^{k-1}  | \E[u](\bx_1) - \E[v](\bx_1) |                                                \\
                                        & \le
    k\,(\max\{ \Vert \E[u] \Vert_{L^\infty(\Omega_1)} , \Vert \E[v] \Vert_{L^\infty(\Omega_1)} | \} )^{k-1}  | \E[u](\bx_1) - \E[v](\bx_1) | \\
                                        & = k\, (M_\E(u,v))^{k-1}  | \E[u-v](\bx_1)  |
    \le k\, (M_\E(u,v))^{k-1}   \E[|u-v|](\bx_1) .
  \end{align*}
  From this we conclude by integration for $q\in[1,\infty)$
  \begin{align*}
    \Vert \E^k[u] - \E^k[v] \Vert_{L^q(\Omega_1)} \le
    k\, (M_\E(u,v))^{k-1} \Vert \E[u-v] \Vert_{L^q(\Omega_1)}\le
    k\, (M_\E(u,v))^{k-1} \Vert \E[|u-v|] \Vert_{L^q(\Omega_1)} .
  \end{align*}
  Again, for $q=\infty$ we replace the integration over $\Omega_1$ by $\esssup_{\bx_1\in\Omega_1}$ in the above inequality.

  Finally, we note that by \eqref{eq:estimate-expectation-3} it holds
  \begin{align*}
    M_\E(u,v) \le M(u,v) .
  \end{align*}
  \qed
\end{proof}

\begin{acknowledgement}
  The authors appreciate the reviewer's valuable comments and suggestions, which have helped to significantly improve the manuscript.
\end{acknowledgement}

\bibliographystyle{spmpsci}      % mathematics and physical sciences
% \bibliography{deterministic-mra}   % name your BibTeX data base
\bibliography{main}   % name your BibTeX data base

\section*{Declarations}
\subsection*{Funding}
The authors thank the Deutsche Forschungsgemeinschaft (DFG, German Research Foundation) for the financial support through 320021702/GRK2326,  333849990/IRTG-2379, CRC1481, HE5386/18-1,19-2,22-1,23-1, ERS SFDdM035 and under Germany’s Excellence Strategy EXC-2023 Internet of Production 390621612 and under the Excellence Strategy of the Federal Government and the Länder. Support through the EU project DATAHYKING is also acknowledged.

\subsection*{Conflicts of interest/Competing interests}
There are no conflicts of interest.
\subsection*{Availability of data and material}
Data will be made available on reasonable request.
\subsection*{Code availability}
Code will not be made available.

\end{document}